\documentclass[aos,preprint]{imsart}

\RequirePackage{amsthm,amsmath,amsfonts,amssymb}
\RequirePackage[authoryear]{natbib} 
\RequirePackage[colorlinks,citecolor=blue,urlcolor=blue]{hyperref}
\RequirePackage{graphicx}

\usepackage{xcolor}
\usepackage{multirow}

\usepackage{keyval}
\usepackage{calc}
\usepackage{hyperref}
\usepackage[capitalise,noabbrev]{cleveref}

\usepackage{paralist}
\usepackage{xr-hyper}

\startlocaldefs

\def\journal@name{} 

\numberwithin{equation}{section}
\theoremstyle{plain}
\newtheorem{thm}{Theorem}

\newtheorem{lemm}{Lemma}
\newtheorem{coro}{Corollary}
\theoremstyle{remark}
\newtheorem{rem}{Remark}

\newtheorem{ex}{Example}
\newtheorem{con}{Condition}

\RequirePackage{amsmath}
\RequirePackage{amssymb}
\RequirePackage{amsthm}
\RequirePackage{bm} 
\RequirePackage{url}
\usepackage{multirow}
\usepackage{natbib}
\usepackage{graphicx}
\usepackage{subfigure}
\usepackage{makecell}
\usepackage{booktabs}
\usepackage{array}
\usepackage{url}
\usepackage{algorithm}
\usepackage{algorithmic}
\usepackage{dsfont}

\let\hat\widehat
\let\tilde\widetilde

\newcommand{\cD}{\mathcal{D}}

\newcommand{\cF}{\mathcal{F}}

\newcommand{\cK}{\mathcal{K}}
\newcommand{\cL}{\mathcal{L}}

\newcommand{\cP}{\mathcal{P}}

\newcommand{\cS}{{\mathcal{S}}}
\newcommand{\cT}{{\mathcal{T}}}

\newcommand{\cV}{\mathcal{V}}

\newcommand{\DD}{\mathbb{D}}
\newcommand{\EE}{\mathbb{E}}

\newcommand{\PP}{\mathbb{P}}

\DeclareMathOperator{\ind}{\mathds{1}}  

\crefname{thm}{Theorem}{Theorems}
\crefname{lemm}{Lemma}{Lemmas}
\crefname{prop}{Proposition}{Propositions}
\crefname{coro}{Corollary}{Corollaries}
\crefname{ex}{Example}{Examples}
\crefname{con}{Condition}{Conditions}
\crefname{defi}{Definition}{Definitions}
\crefname{rem}{Remark}{Remarks}

\newcommand{\comments}[1]{}

\allowdisplaybreaks

\newcommand{\Prob}[2]{
	\mathbb{P}_{#1} \left( #2 \right) }

\newcommand{\E}[2]{
	\mathbb{E}_{#1} \left[ #2 \right] }

\newcommand{\cov}[3]{
	\mathbb{C}\text{ov}_{#1} \left( #2; #3 \right) }

\newcommand{\op}[1]{o_P \left( #1 \right)}
\newcommand{\Op}[1]{O_P \left( #1 \right)}

\newcommand{\Ind}[1]{ \mathds{1} \left\{ #1 \right\} }
\newcommand{\argmin}{\mathop{\text{arg}\,\min}}

\newcommand{\wc}{\rightsquigarrow}
\newcommand{\too}{\longrightarrow}

\renewcommand{\d}{d}

\newcommand{\R}{\mathbb{R}}
\newcommand{\Q}{\mathbb{Q}}
\newcommand{\N}{\mathbb{N}}

\newcommand{\eps}{\varepsilon}

\newcommand{\hc}{\hat{c}}
\newcommand{\hth}{\hat{\theta}}
\newcommand{\hz}{\hat{\sigma}}

\newcommand{\leb}{\mu_\text{L}}

\newcommand{\xminy}{\mu}
\newcommand{\xmaxy}{\mathcal{M}}

\newcommand{\Par}[1]{\text{Pareto} \left( #1 \right)}

\newcounter{hypA}

\endlocaldefs

\begin{document}

\begin{frontmatter}
\title{Rank-based Estimation under Asymptotic Dependence and Independence, with Applications to Spatial Extremes}
\runtitle{Rank-Based Estimation under Asymptotic Dependence and Independence}

\begin{aug}
\author[A]{\fnms{Micha\"el} 
	\snm{Lalancette}
	\ead[label=e1, mark]{lalancette@utstat.toronto.edu}},
\author[B]{\fnms{Sebastian} 
	\snm{Engelke}
	\ead[label=e3]{sebastian.engelke@unige.ch}}
\and
\author[A]{\fnms{Stanislav} 
	\snm{Volgushev}
	\ead[label=e2, mark]{stanislav.volgushev@utoronto.ca}}
\address[A]{Department of Statistical Sciences, 
	University of Toronto,
	\printead{e1,e2}
 }

\address[B]{Research Center for Statistics, 
	University of Geneva,
	\printead{e3}
}

\end{aug}

\begin{abstract}
Multivariate extreme value theory is concerned with modeling the joint tail behavior of several random variables. Existing work mostly focuses on asymptotic dependence, where the probability of observing a large value in one of the variables is of the same order as observing a large value in all variables simultaneously. However, there is growing evidence that asymptotic independence is equally important in real world applications. Available statistical methodology in the latter setting is scarce and not well understood theoretically. We revisit non-parametric estimation and introduce rank-based M-estimators for parametric models that simultaneously work under asymptotic dependence and asymptotic independence, without requiring prior knowledge on which of the two regimes applies. Asymptotic normality of the proposed estimators is established under weak regularity conditions. We further show how bivariate estimators can be leveraged to obtain parametric estimators in spatial tail models, and again provide a thorough theoretical justification for our approach.
\end{abstract}

\begin{keyword}[class=MSC2020]
\kwd[Primary ]{62G32}
\kwd[; secondary ]{62F12} \kwd{62G20} \kwd{62G30}
\end{keyword}

\begin{keyword}
\kwd{Multivariate extremes} \kwd{Asymptotic independence} \kwd{Inverted max-stable distribution} \kwd{Spatial process} \kwd{M-estimation}
\end{keyword}

\end{frontmatter}


\section{Introduction}

Assessing the frequency of extreme events is crucial in many different fields such as environmental sciences, finance and insurance. The most severe risks are often associated to a combination of extreme values of several different variables or the joint occurrence of an extreme phenomenon across different spatial locations. Statistical methods for accurate modeling of such multivariate or spatial dependencies between rare events is provided by extreme value theory. Applications include the analysis of extreme flooding \citep{kee2009, eng2014b, eng2018a}, risk diversification between stock returns~\citep{poo2004, zho2010} and climate extremes \citep{wes2011,zsc2017}.

Extremal dependence between largest observations of two random variables $X$ and $Y$ with distribution functions $F_1$ and $F_2$, respectively, can take many different forms. A classical assumption to measure and model this dependence is multivariate regular variation \citep[cf.,][]{R1987}, which is
equivalent to the existence of the stable tail dependence function
\begin{equation} \label{eq:l}
\ell(x,y) := \lim_{t \downarrow 0} \frac{1}{t} \Prob{}{F_1(X)\geq 1-tx \mbox{ or } F_2(Y)\geq 1-ty}, \quad x,y \in [0,\infty);
\end{equation}
see \cite{H1992} and \cite{DF2006}. This condition allows a first broad classification regarding extremal dependence of bivariate random vectors into two different regimes. If $\ell(x,y) = x + y$, $X$ and $Y$ are said to be asymptotically independent; in this case the joint exceedance probability is negligible compared to the marginal exceedance probabilities. Otherwise, a stronger form of extremal dependence, called asymptotic dependence, holds and the joint exceedance probability is of the same order as the probability of one of the components exceeding a high threshold.

Most of the existing probabilistic and statistical theory deals with asymptotic dependence. {A variety of methods exists, including non-parametric estimation \citep{H1992, ein2009, GPS2011}, bootstrap procedures \citep{PL2008,BD2013}, parametric approaches including likelihood estimation \citep{LT1996, HNP2008,pad2010, dom2016b} and M-estimation \citep{EKS2008, EMKS2015}. See also \cite{EKS2012, EKKS2016} for inference in the $d$-dimensional and spatial setting. There is a rich literature on multivariate tail models \citep[see for instance][among many others]{G1960, T1988, HR1989} and generalizations to spatial domains \citep{BR1977, S1990, S2002}.}

Recent studies have shown that in many applications such as spatial precipitation \citep{LDELW2018} and significant wave height \citep{WT2012}, dependence tends to become weaker for the largest observations and asymptotic independence is therefore the more appropriate regime. In this case, the stable tail dependence function in~\eqref{eq:l} does not contain information on the degree of asymptotic independence and is therefore not suitable for statistical modeling.  A remedy to this problem was proposed by \cite{LT1996,LT1997} who introduced a more flexible condition on the joint exceedance probabilities. Their setting implies the existence of
\begin{equation} \label{eq:defc}
c(x,y) := \lim_{t\downarrow 0} \frac{1}{q(t)}  \Prob{}{F_1(X) \geq 1-tx, F_2(Y) \geq 1-ty },\quad x,y\in[0,\infty),
\end{equation}
where $q$ is a suitable measurable function that makes the limit nontrivial. Necessarily, $q$ is regularly varying at zero with index $1/\eta \in [1,\infty)$. 
The residual tail dependence coefficient $\eta$ describes the strength of residual dependence in the tail and $\eta<1$ implies asymptotic independence.
One speaks about positive and negative association between extremes if $\eta>1/2$ and $\eta<1/2$, respectively.
Early works focus on estimating the degree of asymptotic independence $\eta$ and various estimators have been proposed and studied \citep{LT1997, P1999, DDFH2004}.
A more complete description of the residual extremal dependence structure is given by the function $c$ in \cref{eq:defc}; in fact, the value of $\eta$ can be deduced from $c$ (see \Cref{sec:2} below). Several parametric families exist for multivariate \citep[e.g.,][]{wel2014} and spatial applications \citep[e.g.,][]{WT2012}. Other statistical approaches for modeling asymptotic independence are also related to this function, including hidden regular variation \citep{res2002, hef2007} and the conditional extreme value model \citep{hef2004}. Note that \cref{eq:defc} includes the asymptotic dependence case if $\lim_{t \downarrow 0} q(t)/t > 0$, and the function $c(x,y) \propto x+y - \ell(x,y)$ then contains the same information as $\ell$.

Since it is typically not known \emph{a priori} whether asymptotic dependence or independence is present in a data set, recent parametric models are able to capture both regimes as different sub-sets of the parameter space \citep[e.g.,][]{RL2009, WTDE2017, hus2017,  EOW2019, HW2019}. Most of the literature in this domain is concerned with constructing parametric models, and estimation is usually based on censored likelihood and discussed informally while no theoretical treatment of the corresponding estimators is provided. Moreover, it is typically assumed that extreme observations used to construct estimators already follow a limiting model, and the bias which results from this type of approximation is ignored.

The present paper is motivated by a lack of generic estimation methods that work under both asymptotic dependence and independence and have a thorough theoretical justification. We first revisit a non-parametric, rank-based estimator of the function $c$ in \cref{eq:defc} which appeared in \citep{DDFH2004} and provide a new fundamental result on its asymptotic properties, which completely removes any smoothness assumptions on $c$. This result is the crucial building block for the second major contribution of this paper: a new M-estimation framework that is applicable across dependence classes.

M-estimators for the stable tail dependence function $\ell$ have been proposed by \cite{EKS2008, EKS2012, EKKS2016}. Under asymptotic dependence, $c$ can be recovered from $\ell$ and thus any method for estimating $\ell$ also yields an estimator for $c$. However, this is no longer true under asymptotic independence. Estimators of $\ell$ can therefore not be used to fit statistical models with asymptotic independence or models bridging both dependence classes. We define a new class of M-estimators based on $c$ for parametric extreme value models that can be applied regardless of the dependence class. A major challenge under asymptotic independence is due to the fact that the scaling function $q$ is unknown. Additionally, $c$ loses some of the regularity (such as concavity) that it enjoys under asymptotic dependence. Nevertheless, we are able to prove asymptotic normality of our estimators under weak regularity conditions, which are shown to be satisfied for popular models such as the class of inverted max-stable distributions \citep[see][]{WT2012}. 

The challenges described above become even greater for spatial data. Even at the level of pairwise distributions, real data can exhibit asymptotic dependence at locations that are close but asymptotic independence at locations that are far apart. This necessitates models that can incorporate both, asymptotic dependence and independence at the same time. Estimation in such models calls for methods that can deal with both regimes simultaneously, and we show that our findings in the bivariate case can be leveraged to construct estimators in this setting.

In \Cref{sec:2}, we provide the necessary background on  asymptotic dependence and independence for bivariate distributions, discuss an extension to the spatial setting, and provide several examples. The estimation methodology is introduced in \Cref{framework}, while theoretical results are collected in \Cref{results}. The methodology is illustrated via simulation studies in \Cref{simulations}, while an application to extreme rainfall data is given in \Cref{application}. All proofs are deferred to \cref{proofs,tech,proof IMS,proof rsc} in the Supplementary Material.

\section{Multivariate extreme value theory} \label{sec:2}

\subsection{Bivariate models}\label{sec:2.1}

Let $(X,Y)$ be a bivariate random vector with joint distribution function $F$ and marginal distribution functions $F_1$ and $F_2$, respectively. There is a variety of approaches to describe the joint tail behavior of $(X,Y)$.

The assumption of multivariate regular variation \citep[cf.,][]{R1987} is classical in extreme value theory and the stable tail dependence function in~\eqref{eq:l} has been extensively studied. Its margins are normalized, $\ell(x,0) = \ell(0,x) = x$, and it satisfies $x \vee y \leq l(x,y) \leq x+y$ for all $x,y \in [0,\infty)$. Moreover, it is a convex and homogeneous function of order one, the latter meaning that $\ell(tx, ty) = t\ell(x,y)$ for all $t>0$.
The importance of stable tail dependence functions stems from their connection to max-stable distributions. A bivariate random vector $(Z_1, Z_2)$ has max-stable dependence with standard uniform margins iff its distribution function is given by
\begin{align}\label{max_stable}
\Prob{}{ Z_1 \leq x, Z_2 \leq y} = \exp\{- \ell(-\log x,-\log y) \}, \quad x,y \in [0,1],
\end{align}
where $\ell$ is the stable tail dependence function pertaining to $(Z_1, Z_2)$. Note that any max-stable distribution associated with $\ell$ satisfies \cref{eq:l} with that same $\ell$, this follows after a simple Taylor expansion. Two examples of max-stable distributions (equivalently, stable tail dependence functions) that will repeatedly appear in the present paper are as follows.
\begin{itemize}
		\item[(i)] The bivariate H\"usler--Reiss distribution \citep[][]{HR1989, EMKS2015} is defined by
		\[
		\ell(x, y) = x \Phi\Big( \lambda + \frac{\log x - \log y}{2\lambda} \Big) + y \Phi\Big( \lambda + \frac{\log y - \log x}{2\lambda} \Big),
		\]
		where $\Phi$ is the standard normal distribution function and $\lambda \in [0, \infty]$ parametrizes between perfect independence ($\lambda = \infty$) and dependence ($\lambda = 0$).
		\item[(ii)] The asymmetric logistic distribution \citep{T1988}, is given by
		\[
		\ell(x, y) = (1-\nu)x + (1-\phi)y +(\nu^r x^r + \phi^r y^r)^{1/r}, \quad \nu, \phi \in [0, 1], r \geq 1.
		\]
		Note that $\nu=\phi=1$ yields the classical logistic model \citep{G1960}.
\end{itemize}

While multivariate regular variation and max-stability have been very popular due to their nice theoretical properties, they are not informative under asymptotic independence which limits their use in many applications.

Assumption~\eqref{eq:defc} allows for more flexible tail models since the limiting function $c$ is non-trivial even under asymptotic independence and contains information on the structure of residual extremal dependence in the vector $(X,Y)$. For the sake of identifiability, we scale $q$ such that $c(1, 1) = 1$. We will refer to $c$ and $q$ as the survival tail function and the scaling function associated to $(X, Y)$. It turns out that $q$ has to be regularly varying of order $1/\eta \in [1, \infty)$ and that $c$ is a homogenous funcion of order $1/\eta$, that is,
\[
c(tx, ty) = t^{1/\eta}c(x,y), \quad t>0;
\]
see for example \cite{DDFH2004} or \cref{RV1} in the supplement. Note that the extremal dependence coefficient \citep[see][]{CHT1999} can be defined as $\chi := \lim_{t \downarrow 0} q(t)/t$. Asymptotic independence is then equivalent to $\chi=0$, or $q(t) = o(t)$, whereas asymptotic dependence corresponds to $\chi>0$.

Furthermore, the homogeneity property of $c$ implies a spectral representation. More precisely, there exists a finite measure $H$ on $[0, 1]$ such that
\[
c(x, y) = \int_{[0, 1]} \Big( \frac{x}{1-w} \wedge \frac{y}{w} \Big)^{1/\eta} H(\d w), \quad x,y \in [0,\infty);
\]
see Theorem 1 in \cite{RL2009} or \cref{int rep} in the supplement.

We provide several examples that illustrate the concepts discussed above without going too deeply into technical details. A more thorough discussion of the corresponding theory is given throughout \Cref{results}.

\begin{ex}[\textit{Domain of attraction of max-stable distributions}]
\label{ex:ms}
Suppose that $(X, Y)$ satisfies \cref{eq:l} for a stable tail dependence function $\ell$ which is not everywhere equal to $x+y$. Then \cref{eq:defc} holds with $q(t) = \chi t$ and $c(x, y) = (x+y - \ell(x, y))/\chi$, where the extremal dependence coefficient $\chi$ is positive. We further note that \cref{eq:l} holds whenever $(X, Y)$ is in the max domain of attraction of a max-stable random vector $Z$ satisfying \cref{max_stable}, see \cite{DF2006} for a definition and additional details.
\end{ex}

\begin{ex}[\textit{Inverted max-stable distributions}]
	\label{ex:ims}
	
	The family of inverted max-stable distributions \citep[e.g.,][Definition 2]{WT2012} is parametrized by all stable tail dependence functions. More precisely, let $G$ be the distribution function of a bivariate distribution with max-stable dependence, uniform margins and stable tail dependence function $\ell$. A random vector $(X, Y)$ with uniform marginal distributions is said to have an inverted max-stable distribution with stable tail dependence $\ell$ if $(1-X, 1-Y) \sim G$. Assuming that $\ell$ satisfies a quadratic expansion (see \cref{ex:ims:cont}), the law of $(X, Y)$ satisfies \cref{eq:defc} with
	\[
	q(t) = t^{\ell(1, 1)}, \quad c(x, y) = x^{\dot\ell_1(1, 1)} y^{\dot\ell_2(1, 1)}, 
	\]
	where $\dot\ell_j$ denotes the $j$-th directional partial derivative of $\ell$ from the right, $j=1,2$. Any such stable tail dependence function satisfies $\ell(1, 1) = \dot\ell_1(1, 1) + \dot\ell_2(1, 1) \in (1, 2]$, and therefore this is an asymptotically independent model with $\eta = 1 / \ell(1, 1)$.
	
	Any existing parametric class of stable tail dependence functions can be used to define a parametric class of inverted max-stable distributions. In particular we consider the two families discussed earlier 
	
	\begin{itemize}
		
		\item[(i)] Provided that $\lambda>0$, the inverted H\"usler--Reiss distribution has
		\begin{equation} \label{eq:sym}
		q(t) = t^{2\theta}, \quad c(x, y) = (xy)^\theta,
		\end{equation}
		where $\theta := \Phi(\lambda) \in (1/2, 1]$.
		
		\item[(ii)] The inverted asymmetric logistic distribution has
		\begin{equation} \label{eq:asym}
		q(t) = t^{\theta_1 + \theta_2}, \quad c(x, y) = x^{\theta_1} y^{\theta_2},
		\end{equation}
		where $\theta_1 := 1-\nu + \nu^r (\nu^r + \phi^r)^{1/r-1}$ and $\theta_2 := 1-\phi + \phi^r (\nu^r + \phi^r)^{1/r-1}$. Note that by suitable choices of the parameters $r,\nu,\phi$ any value of $(\theta_1, \theta_2) \in (0, 1]^2$ such that $\theta_1 + \theta_2 \in (1,2]$ can be obtained.

	\end{itemize}

\end{ex}

\begin{ex}[\textit{A random scale construction}]
	\label{ex2}
	
	Bivariate random scale constructions are a popular way of creating distributions with rich extremal dependence structures; see \cite{EOW2019} and references therein for an overview. They are random vectors of the form $(X, Y) = R(W_1, W_2)$ where the radial variable $R$ is  {assumed independent} of the angular variables $W_j$, $j \in \{1, 2\}$. This motivates the following model with parameters $\alpha_R, \alpha_W > 0$:
	\begin{equation} \label{model}
	(X, Y) = R(W_1, W_2), \quad R \sim \Par{\alpha_R}, W_j \sim \Par{\alpha_W} 
	\end{equation}
	where  {$W_1,W_2$ are independent} and a $\Par{\alpha}$ distribution has distribution function $1 - x^{-\alpha}$ for $x \geq 1$.  By \cref{ex2:cont} below, $(X, Y)$ satisfies \cref{eq:defc} with functions $q$ and $c$ depending only on the value of the ratio $\lambda := \alpha_R / \alpha_W$. In particular, we obtain asymptotic dependence if $\lambda<1$ and asymptotic independence otherwise. Detailed expressions for $q$ and $c$ are provided in \cref{ex2:cont}.
\end{ex}

\subsection{Spatial models}

Spatial extreme value theory is an extension of multivariate extremes to continuous index sets. It is particularly useful for modeling extremes of natural phenomena over spatial domains and examples include heavy rainfall, high wind speeds and heatwaves \citep[e.g.,][]{dav2012, LDELW2018}.

Let $\cT$ be a spatial domain (typically a subset of $\R^2$) and $Y = \{Y(u): u \in \cT\}$ be a stochastic process whose extremal behavior we are interested in. We impose the condition in \cref{eq:defc} on all bivariate margins of $Y$ so that for each pair $s = (u, u')$ of locations,  {and all $x,y \in [0, \infty)$} the limit
\begin{equation} \label{eq:defc_spatial}
c^{(s)}(x, y) := \lim_{t\downarrow 0} \frac{1}{q^{(s)}(t)} \Prob{}{F^{(u)}(Y(u)) \geq 1-tx, F^{(u')}(Y(u')) \geq 1-ty }
\end{equation}
exists and is non-trivial; here $F^{(u)}$ is the distribution function of $Y(u)$. Similarly to the bivariate case, $q^{(s)}$ must be regularly varying with index $1/\eta^{(s)} \in [1,\infty)$ and $c^{(s)}$ is homogeneous of order $1/\eta^{(s)}$. 

In applications, spatial extreme value theory can be linked to  multivariate extreme value theory through the fact that spatial processes are usually measured at a finite set of locations. However, generic multivariate models do not take into account the additional structure arising from spatial features of the domain. Statistical models for processes, in contrast, make use of geographical information to construct structured, low-dimensional parametric models \citep[see, e.g.,][]{eng2021}. 

 {Similarly to max-stable distributions in \cref{max_stable}, max-stable processes play an important role in modeling spatial extremes. The stochastic process $Z = \{Z(u): u \in \cT\}$ is called max-stable if all its finite dimensional distributions are max-stable, which implies in particular that for each pair $s = (u, u')$, the bivariate margin $(Z(u), Z(u'))$ satisfies \cref{max_stable} with stable tail dependence function $\ell^{(s)}$. Hence \cref{eq:defc_spatial} follows for any max-stable process $Z$ for which $(Z(u),Z(u'))$ are not independent for all $u,u'\in\cT$.}

Brown--Resnick processes \citep{BR1977} provide an important subclass of max--stable processes. A Brown--Resnick process $\mathcal{B} = \{\mathcal{B}(u): u \in \cT\}$ is parametrized by a variogram function $\gamma: \mathcal{T}^2 \to \R_+$, and any pair $(\mathcal{B}(u),\mathcal{B}(u'))$ is a bivariate H\"usler--Reiss distribution with parameter $\lambda = \sqrt{\gamma(u, u')}/2$ \citep[][]{HR1989}. Parametric models can be constructed by imposing a parametric form for $\gamma$. An example when $\cT \subset \R^d$ is the fractal family of variograms given by $\gamma(s) = (\|s_1 - s_2\| / \beta)^\alpha$, where $s = (s_1, s_2)$, $\| \cdot \|$ is the Euclidean norm and $\alpha \in (0, 2]$, $\beta > 0$ are the model parameters \citep{KSH2009}. We next discuss two classes of processes for which \cref{eq:defc_spatial} holds.

\begin{ex}[\textit{Domain of attraction of max-stable processes}] \label{ex1.0}
 {	
A process $Y = \{Y(u): u \in \cT\}$ is in the max-domain of attraction of the max-stable process $Z$ if there exist sequences of continuous functions $a_n, b_n: \cT \to \R$ such that
\begin{equation} \label{eq:doaproc}
\{\max_{i=1,...,n} Y_i(\cdot) - a_n(\cdot)\}/b_n(\cdot) \wc Z(\cdot), \quad n \to \infty
\end{equation}
for i.i.d. copies $Y_1,Y_2,...$ of the process $Y$ where weak convergence takes place in the space of continuous functions on $\cT$ equipped with the supremum norm; see \cite{DL2001} and Chapter 9 of \cite{DF2006} for the special case $\cT = [0,1]$. }

 {\Cref{eq:doaproc} implies that any pair $(Y(u),Y(u'))$ with $u \neq u' \in \cT$ is in the max-domain of attraction of the pair $(Z(u),Z(u'))$. If every such pair is not independent, \cref{eq:defc_spatial} holds for all $s = (u,u')$ by the discussion in \cref{ex:ms}.}
	 
\end{ex}

While max-stable processes allow for flexible spatial dependence structures, they can only be used as models for asymptotic dependence. This often violates the characteristics observed in real data, especially for locations $u, u' \in \cT$ that are far apart. To model data in such cases, asymptotically independent spatial models have been constructed that satisfy \cref{eq:defc_spatial} and where the residual tail dependence coefficients $\eta^{(s)}$ vary with the distance between the pair $s$ of locations. Most of the models are identifiable from the bivariate margins so that statistical methods for $c^{(s)}$ will provide estimators for spatial tail dependence parameters; see \cref{estimation spatial} for the methodology. A broad class of asymptotically independent stochastic processes are the inverted max-stable processes \citep{WT2012}. 

\begin{ex}[\textit{Inverted max-stable processes}]
\label{ex1.1}
Let $Z = \{Z(u): u \in \cT\}$ be a process with max-stable dependence, uniform margins and bivariate tail dependence functions $\ell^{(s)}$. The process $Y = \{1 - Z(u): u \in \cT\}$ is called inverted max-stable. For a pair $s \in \cT^2$, assuming that $\ell^{(s)}$ satisfies the smoothness condition mentioned in \cref{ex:ims}, $Y$ satisfies \cref{eq:defc_spatial} with
\[
q^{(s)}(t) = t^{\ell^{(s)}(1,1)}, \quad c^{(s)}(x, y) = x^{\dot\ell_1^{(s)}(1, 1)} y^{\dot\ell_2^{(s)}(1, 1)},
\]
so that $\eta^{(s)} = 1/\ell^{(s)}(1,1)$ is a (usually non-constant) function on $\cT^2$. In particular, if a Brown--Resnick process is parametrized by a variogram function $\gamma: \cT^2 \to \mathbb R_+$ then the corresponding inverted Brown--Resnick process has $1/\eta^{(s)} = 2\Phi(\sqrt{\gamma(s)}/2)$. 
\end{ex}

\section{Estimation} \label{framework}

In this section we present the proposed estimators. First, we recall the non-parametric estimator of a survival tail function from \cite{DDFH2004} in \Cref{estimation-nonpar}. Using this as building block, we construct M-estimators for bivariate survival tail functions (\Cref{estimation-par}) and leverage those estimators to introduce methodology for spatial tail estimation (\Cref{estimation spatial}).

\subsection{Non-parametric estimators of survival tail functions}
\label{estimation-nonpar}

Recall that $(X, Y)$ is a random vector with joint distribution function $F$ that satisfies \cref{eq:defc}, and assume that its marginal distribution functions $F_1$ and $F_2$ are continuous. Denoting by $Q$ the joint distribution function of $(1 - F_1(X), 1 - F_2(Y))$, we can rephrase \cref{eq:defc} as
\begin{equation} \label{Q}
\frac{Q(tx, ty)}{q(t)} = c(x, y) + O(q_1(t)), \quad x,y\in[0,\infty),
\end{equation}
for some function $q_1(t) \to 0$ as $t \to 0$. Suppose that $(X_1, Y_1), \dots, (X_n, Y_n)$ are independent samples from $F$. Since $F_1, F_2$ are unknown, the observations $(1 - F_1(X_i), 1 - F_2(Y_i))$ are not available and can not be used to construct a feasible estimator of $Q$. A typical solution to this problem is to replace $F_j$ by its empirical counterpart $\hat F_j$, which leads to the estimator
\begin{equation}\label{Qnhat}
\hat Q_n(x,y) := \frac{1}{n} \sum_{i=1}^n \Ind{n \hat F_1(X_i) \geq n+1-\lfloor nx \rfloor, n \hat F_2(Y_i) \geq n+1-\lfloor ny \rfloor}; 
\end{equation}
see \cite{H1992,DH1998,EKS2008,EKS2012} among others for related approaches in the estimation of stable tail dependence functions.

Given  $\hat Q_n$ and the expansion in \Cref{Q}, an intuitive plug-in estimator of the function $c$ is given by
\begin{equation} \label{eq:hatc}
\hc_n(x, y) = \frac{\hat Q_n\left(kx/n, ky/n \right)}{q(k/n)},
\end{equation}
where we set $t = k/n$ in \Cref{Q} for an intermediate sequence $k = k_n$ such that $k \to \infty$, $k/n \to 0$. Note, however, that this estimator is infeasible under asymptotic independence since the function $q$ is in general unknown. A simple remedy is to recall that we considered the normalization $c(1,1) = 1$ and construct a ratio type estimator
\begin{equation} \label{eq:tildec}
\tilde c_n(x,y) := \frac{\hc_n(x, y)}{\hc_n(1, 1)} = \frac{\hat Q_n\left(kx/n, ky/n \right)}{\hat Q_n\left(k/n, k/n \right)}
\end{equation}
to cancel out the unknown scaling factor $q(k/n)$. This leads to a fully non-parametric estimator of $c$, which is interesting in its own right. Some comments on the asymptotic properties of this estimator will be provided in \Cref{sec:asy-non}.

\begin{rem}\label{rem:hatk}
 {In practice, and especially in a spatial context, it is sometimes appropriate to select directly the effective number of observations used for estimating $c$ \citep{WT2012}. This can be achieved by selecting $k = \hat k$ such that $n Q_n(\hat k/n,\hat k/n) = m$ for a given value $m$. This leads to a data-dependent parameter $\hat k$ which will also be covered by our theory.}
\end{rem}

\subsection{M-estimation in (bivariate) parametric model classes}
\label{estimation-par}

While the non-parametric estimators from the previous section possess attractive theoretical properties, they have certain practical drawbacks. For any finite sample size $n$ they are neither continuous nor homogeneous, hence they are not admissible survival tail functions. Additionally, it is difficult to use purely non-parametric estimators in spatial settings. A solution to this problem, which also yields easily interpretable estimators, is to fit parametric models.

In what follows, assume that $c$ belongs to a family $\{c_\theta : \theta \in \Theta\}$, where $\Theta \subseteq \R^p$ and the true parameter $\theta_0 \in \Theta$ is unknown. Our aim is to leverage the non-parametric estimators from \Cref{estimation-nonpar} to construct an estimator for $\theta_0$.  {For stable tail dependence functions which are only informative under asymptotic dependence such a program was carried out in \cite{EKS2008,EKS2012}.} A direct application of the corresponding ideas in our setting would be to estimate $\theta$ through
\[
\breve{\theta} := \argmin_{\theta \in \Theta} \Big\| \int_{[0, T]^2} g(x, y) c_\theta(x, y) \d x \d y - \int_{[0, T]^2} g(x, y) \tilde c_n(x, y) \d x \d y \Big\|,
\]
for an integrable vector-valued weight function $g: \R^2 \to \R^q$, where $\|\cdot\|$ denotes the Euclidean norm. However, as we will discuss in Remark~\ref{rem-hypi1}, the use of $\tilde c_n$ would place unnecessarily strong assumptions on the function $c$ in the case of asymptotic dependence. Hence we propose to consider the following alternative approach. Define 
\begin{equation}\label{eq:Psinstar}
\Psi_n^*(\theta, \zeta) := \zeta \int_{[0, T]^2} g(x, y) c_\theta(x, y) dx dy - \int_{[0, T]^2} g(x, y) \hat Q_n(kx/n, ky/n) dx dy
\end{equation}
and let 
\begin{equation} \label{eq:hatthetazeta}
(\hth_n, \hat \zeta_n) := \argmin_{\theta \in \Theta,\zeta>0} \|\Psi_n^*(\theta, \zeta)\|.
\end{equation}

To understand the rationale of this approach, note that $\hat c_n(x,y)$ is proportional to $\hat Q_n(kx/n,ky/n)$ but the proportionality constant involves $q$ and is thus unknown. We thus essentially propose to add this unknown normalization factor as an additional scale parameter~$\zeta$. More precisely, write $\leb$ for the Lebesgue measure on $[0, T]^2$, let
$$\Psi_n(\theta, \sigma) = \sigma \int gc_\theta \d\leb - \int g\hat c_n \d\leb,$$
and note that $\Psi_n^*$ and $\Psi_n$ are linked through $\Psi_n^*(\theta, \zeta) = q(k/n)\Psi_n(\theta, \zeta/q(k/n))$. Thus $(\hth_n, \hat \zeta_n)$ minimizes $\| \Psi_n^* \|$ if and only if $(\hth_n, \hat \zeta_n /q(k/n))$ minimizes $\| \Psi_n \|$. Furthermore, under suitable assumptions on $g$ and $\Theta$ we have $\sigma \int gc_\theta \d\leb = \int gc_{\theta_0} \d\leb$ if and only if $\theta = \theta_0$ and $\sigma=1$. Hence, if $\hc_n$ is close to $c_{\theta_0}$, it is expected that $\hth_n$ will be close to $\theta_0$ and that $\hat\zeta_n / q(k/n)$ will be approximately $1$.

Note that \cref{eq:Psinstar} only involves an integral of $\hat Q_n$ while $\tilde c_n$ also involves point-wise evaluation of this function. Since integration acts as smoothing, it can be expected that studying $\Psi_n^*$ will require less regularity conditions than working with $\breve{\theta}$; see \Cref{rem-hypi1} for additional details.

\subsection{Parametric estimation for spatial tail models}
\label{estimation spatial}

In this section, we show how the bivariate estimation procedures discussed earlier can be leveraged to obtain two different estimators for parametric spatial models, which can include both asymptotic dependence and independence. Assume that we observe $n$ independent copies $Y_1, \dots, Y_n$ of a spatial process $Y$ at a finite set of locations $u_1,\dots,u_d\in\mathcal T$. Denote the corresponding observations by $X_1,\dots, X_n$ where $X_i = (X_i^{(1)}, \dots, X_i^{(d)}) := (Y_i(u_1),\dots,Y_i(u_d))$ are independent copies of the random vector $X = (X^{(1)}, \dots, X^{(d)}) := (Y(u_1), \dots, Y(u_d)) \in \R^d$; see \cite{EKKS2016} for a similar framework.

Let $\cP$ denote the set of all subsets of $\{1, \dots, d\}$ of size 2  interpreted as ordered pairs, so that elements of $\cP$ will take the form $s = (s_1,s_2)$ with $s_1 < s_2$. In what follows, we will need to repeatedly make use of vectors $x \in \R^{|\cP|}$ that are indexed by all pairs $s \in \cP$. For such vectors we will assume that the pairs in $\cP$ are ordered in a pre-specified order and will write $x^{(s)}$ for the entry of the vector $x$ that corresponds to pair $s$. 

Assume that for each pair $s$ the random vector $(X^{(s_1)}, X^{(s_2)})$ satisfies \cref{Q} with scale function $q^{(s)}$ and survival tail function $c^{(s)}$. Following the ideas laid out in \Cref{estimation-nonpar}, define $\hat Q_n^{(s)}$ as in \cref{Qnhat} but based on the bivariate observations $(X_i^{(s_1)}, X_i^{(s_2)}), i = 1, \dots, n$. We now discuss two parametric estimators for the functions $c^{(s)}$.

Assume that we start with a parametric model $\{c_\theta: \theta \in \tilde \Theta\}$, $\tilde\Theta \subseteq \R^{\tilde p}$, for bivariate survival tail functions and that each $c^{(s)}$ falls in this class. This implies that $\tilde\Theta$ can be linked to a spatial parameter space $\Theta \subseteq \R^p$ through the relations $c^{(s)} = c_{h^{(s)}(\vartheta_0)}$, where $h^{(s)}: \Theta \to \tilde \Theta$ for each pair $s$. To make this idea more concrete, consider the following example, which we will revisit in \Cref{sim spatial,application}.

\begin{ex} \label{ex:BR_revisited}
If the process $Y$ is an inverted Brown--Resnick process on $\R^2$ (see \cref{ex1.1}) then $X$ has an inverted H\"usler--Reiss distribution and the bivariate survival tail functions are of the form $c^{(s)}(x, y) = (xy)^{\theta^{(s)}}$, for some $\theta^{(s)} \in (1/2, 1)$. This determines the parametric class $\tilde \Theta$. A more specific model of Brown--Resnick processes corresponds to the sub-family of fractal variograms \citep{KSH2009, EMKS2015}, where
\begin{equation}\label{eq:fracvari}
\theta^{(s)} = h^{(s)}((\alpha,\beta)) = \Phi\left( \frac{(\|u_{s_1} - u_{s_2}\|/\beta)^{\alpha/2}}{2} \right), \quad s \in \cP,
\end{equation} 
  {where $u_j \in \R^2$ is the coordinate of the location $j$;} see \cref{application} for more motivation of this particular parametrization. The global parameter $\vartheta$ thus takes the form $\vartheta = (\alpha, \beta)$ and $\Theta = (0, 2] \times (0, \infty)$.
\end{ex}

Given the setting above, we can thus compute parametric estimators $\hat \theta_n^{(s)}, s \in \cP$, by the methods for bivariate estimation discussed in \cref{estimation-par}, i.e., $(\hat\theta_n^{(s)}, \hat\zeta_n^{(s)})$ is the minimizer of $\big\| \Psi_n^{*(s)}(\theta, \zeta) \big\|$, where $\Psi_n^{*(s)}$ is defined as $\Psi_n^*$ in \eqref{eq:Psinstar} with $\hat Q_n^{(s)}$ and an intermediate sequence $k^{(s)}$ replacing $\hat Q_n$ and $k$. We obtain an estimator of the spatial parameter by least squares minimization,
\begin{equation}\label{eq:defthetahatls}
\hat\vartheta_n := \argmin_{\vartheta \in \Theta} \sum_{s \in \cP} \left\| h^{(s)}(\vartheta) - \hat \theta_n^{(s)} \right\|^2.
\end{equation}

As an alternative, one may use the relations $h^{(s)}$ between the spatial and bivariate parameters and minimize all the objective functions $\Psi_n^{*(s)}$ simultaneously, leading to the estimator
\begin{equation} \label{eq:defthetahat}
(\tilde\vartheta_n, \tilde\zeta_n) := \argmin_{\vartheta \in \Theta, \zeta \in \R_+^{|\cP|}} \sum_{s \in \cP} \Big\| \Psi_n^{*(s)}(h^{(s)}(\vartheta), \zeta^{(s)}) \Big\|^2,
\end{equation}
A theoretical analysis of the estimators $\hat\vartheta_n$ and $(\tilde\vartheta_n, \tilde\zeta_n)$ is provided in \cref{thm spatial}. We further remark that the computational complexity of the proposed estimators is much lower than that of methods based on full likelihood and it compares favorably to pairwise likelihood. Additional details regarding the latter point can be found in \cref{sec:cc} of the supplement.

\begin{rem}
	At first glance the minimization problem in \cref{eq:defthetahat} seems to be computationally intractable since it contains $|\cP| + \dim(\Theta)$ parameters and since $|\cP|$ can be very large even for moderate dimension $d$. However, a closer inspection reveals that for given $\vartheta$, partially minimizing the objective function in \eqref{eq:defthetahat} over $\zeta \in \R_+^{|\cP|}$ has the exact solution
	\[
	\hat \zeta_n^{(s)}(\vartheta) = \frac{\sum_{j=1}^q \int g_j(x,y) \hat Q_n^{(s)}(k^{(s)}x/n, k^{(s)}y/n) dxdy}{\sum_{j=1}^q \int g_j(x,y) c_{h^{(s)}(\vartheta)}(x, y) dxdy},
	\]
	whenever the right-hand side is positive for all $s$. This is satisfied if for instance $g$ is positive everywhere and each $\hat Q_n^{(s)}$ is based on at least one data point. Thus only numerical optimization over $\vartheta$, which is usually low-dimensional, is required. 
\end{rem}

\section{Theoretical results} \label{results}

We now present our main results on the asymptotic distributions of the various estimators introduced in \Cref{framework}. First, functional central limit theorems are stated for $\hc_n$, followed by our main result on the bivariate M-estimator. Finally, asymptotic normality of the processes $\hc_n^{(s)}$ and of the two parametric estimators in the spatial setting is established. The proofs of all main results are deferred to \Cref{proofs} in the supplement.

\subsection{The bivariate setting}

All results in this section will be derived under the following fundamental assumption.

\begin{con} \label{con-proc}
	\begin{enumerate} 
		\item[(i)] \Cref{Q} holds uniformly on $\cS^+ = \{(x, y) \in [0, \infty)^2: x^2 + y^2 = 1\}$ with a function $q_1(t) = O(1/\log(1/t))$ and the function $q$ is such that $\chi := \lim_{t\downarrow 0} q(t)/t \in [0,1]$ exists.
		\item[(ii)] As $n \to \infty$, $m = m_n := nq(k/n) \to \infty$ and $\sqrt{m} q_1(k/n) \to 0$.
	\end{enumerate} 
\end{con}

We note that in the proofs, \cref{Q} is required to hold locally uniformly on $[0, \infty)^2$, but by \cref{RV1} uniformity on $\cS^+$ implies uniformity over compact subsets of $[0, \infty)^2$. \Cref{con-proc}(ii) is a standard assumption that makes certain bias terms negligible. It is not a model assumption; under \cref{con-proc}(i), there always exists a sequence $k$ such that \cref{con-proc}(ii) is satisfied and thus all of the following discussion will focus on \cref{con-proc}(i). Notably and in contrast to most of the existing literature involving non-parametric estimation, \cref{con-proc} does not assume any differentiability of the function $c$. In fact, our proofs show that all the regularity required on $c$ can be derived from \cref{Q}.  {Considering \cref{rem:hatk}, it is possible to use a data-dependent value $\hat k$. In following results, when this is done, we will assume that there is an unknown sequence $k$ that satisfies \cref{con-proc}(ii), that $m$ is defined as therein, and that $\hat k$ is chosen so that $n \hat Q_n(\hat k/n, \hat k/n) = m$.}

We next discuss this condition in the examples introduced in \cref{sec:2.1}. Proofs for the claims in the examples below can be found  in \cref{proof IMS,proof rsc} of the supplement.

\begin{ex}[\textit{\cref{ex:ms}, continued}] 
{ Most of the literature on asymptotic analysis of estimators of the stable tail dependence function $\ell$ or related quantities under domain of attraction conditions makes some version of the following assumption 
\begin{equation} \label{eq:l-new}
\frac{1}{t} \Prob{}{F_1(X)\geq 1-tx \mbox{ and } F_2(Y)\geq 1-ty} - R(x,y) = O(\tilde q_1(t)) \quad x,y \in [0,\infty);
\end{equation} 
for a non-vanishing function $R$ on $[0, \infty)^2$ where $\tilde q_1(t) = o(1)$, see for instance condition (C2) in \cite{EKS2008} or the discussion in \cite{BVZ2019}. A simple computation involving the inclusion-exclsion formula further shows that this is equivalent to assuming that convergence in \cref{eq:l} takes place with rate $O(\tilde q_1(t))$ and that $\ell(x,y) = x + y - R(x,y)$. Clearly \cref{eq:l-new} implies \cref{con-proc}(i) with $q(t) = tR(1,1)$, $c(x, y) = R(x,y)/R(1,1)$ and $q_1(t) = \tilde q_1(t)$. }	
\end{ex}

\begin{ex}[\textit{\cref{ex:ims}, continued}] \label{ex:ims:cont}

Let $(X, Y)$ be a bivariate inverted max-stable distribution and assume that there exists a constant $C < \infty$ such that for all $u, v > 0$,
\[
\Big| \ell(1 + u, 1 + v) - \ell(1, 1) - \dot\ell_1(1, 1) u - \dot\ell_2(1, 1) v \Big| \leq C \left( u^2 + v^2 \right),
\]
where $\dot\ell_j$ represent the directional partial derivatives of $\ell$ from the right.  {In particular, it suffices for $\ell$ to be twice differentiable.} Then the random vector $(X, Y)$ satisfies \cref{con-proc}(i) with $q(t) = t^{\ell(1, 1)}$, $c(x, y) = x^{\dot\ell_1(1, 1)} y^{\dot\ell_2(1, 1)}$ and $q_1(t) = 1/\log(1/t)$. Moreover, $\dot\ell_j(1, 1) \in (0, 1]$ and $\dot\ell_1(1, 1) + \dot\ell_2(1, 1) = \ell(1, 1) \in (1, 2]$.

\end{ex}

\begin{ex}[\textit{\cref{ex2}, continued}]\label{ex2:cont} Let $(X, Y)$ be {a random scale construction} as defined in \cref{model} and set $\lambda = \alpha_R / \alpha_W$. Then $(X, Y)$ satisfies \cref{con-proc}(i) with functions $q$, $c$ and $q_1$ determined by $\lambda$ as in \cref{table} below.
	
	\begin{table}[ht] \begin{center}
			$\begin{array}{c | c | c | c}
			
			\text{Range of } \lambda & q(t) & c(x, y) & q_1(t) \\
			
			\hline
			
			(0, 1) & K_\lambda t & \frac{2-\lambda}{2(1-\lambda)} \xminy - \frac{\lambda}{2(1-\lambda)} \xminy^{1/\lambda} \xmaxy^{1 - 1/\lambda} & t^{1/\lambda - 1} \\
			
			\hline
			
			1 & \frac{K_\lambda t}{\log(1/t) + \log\log(1/t)} & \xminy \Big( 1 + \frac{1}{2} \log \Big( \frac{\xmaxy}{\xminy} \Big) \Big) & 1/\log(1/t)  \\
			
			\hline
			
			(1, 2) & K_\lambda t^\lambda & \frac{\lambda}{2(\lambda-1)} \xminy  \xmaxy^{\lambda-1} - \frac{2-\lambda}{2(\lambda-1)} \xminy^\lambda & t^{(\lambda-1) \wedge (2 - \lambda)} \\
			
			\hline
			
			2 & K_\lambda t^2 \log(1/t) & \xminy\xmaxy & 1/\log(1/t) \\
			
			\hline
			
			(2, \infty) & K_\lambda t^2 & \xminy\xmaxy & t^{\lambda-2}
			
		\end{array}$
		
		\caption{Tail expansion of the random scale model in \cref{model}, here we set $\xminy := x \wedge y, \xmaxy := x \vee y$, and $K_\lambda$ is a positive constant given in \cref{eq:K_lambda} of the supplement.}
		\label{table}
		
\end{center} \end{table}
\end{ex}

\subsubsection{Asymptotic theory for non-parametric estimators} \label{sec:asy-non}

In this section we consider the estimator $\hc_n$ from \cref{eq:hatc}. Since the process convergence results differ under asymptotic dependence and independence, we discuss these settings separately. Our first result deals with asymptotic independence.

\begin{thm}[Asymptotic normality of $\hc_n$ under asymptotic independence] \label{big lemma AI}

Assume \Cref{con-proc}. Then under asymptotic independence, i.e., when $\chi = 0$,
$$
W_n := \sqrt{m} (\hc_n - c) \wc W,
$$
in $\ell^\infty([0, T]^2)$, for any $T < \infty$. Here, $W$ is a centered Gaussian process with covariance structure given by $\E{}{W(x, y)W(x', y')} = c(x \wedge x', y \wedge y')$.  {The same remains true if $k$ is replaced by $\hat k$ as described after \cref{con-proc}.}

\end{thm}

Note that process convergence of the estimator $\tilde c_n$ from \cref{eq:tildec} can be obtained from the above result through a straightforward application of the functional delta method. This will not be needed in the theory for M-estimators in the next section and details are omitted for the sake of brevity.

Asymptotic properties of $\hc_n$ were considered in \cite{DDFH2004}. However, the proof of the corresponding result (Lemma 6.1) in the latter reference makes the additional assumption that the partial derivatives of $c$ exist and are continuous on $[0,T]^2$ \citep[cf.][Theorem 2.2]{P1999}. In contrast, we are able to show that no condition on existence or continuity of partial derivatives is required. This is a considerable strengthening of the result which further allows to handle many interesting examples that were not covered by the results of~\cite{DDFH2004}. Indeed, both the popular class of inverted max-stable distributions in \Cref{ex:ims} and the random scale construction in \Cref{ex2} lead to functions $c$ that fail to have continuous or even bounded partial derivatives. Before moving on to discussing results under asymptotic dependence, we briefly comment on some of the main ideas of the proof.

\begin{rem}[Main ideas of the proof of \Cref{big lemma AI}] \label{rem:process_conv}
The proof relies on the decomposition
\[
\hc_n(x,y) - c(x,y) = \Big\{ \frac{Q_n(\frac{ku_n(x)}{n},\frac{kv_n(y)}{n})}{q(k/n)} - c(u_n(x),v_n(y)) \Big\} + \big( c(u_n(x),v_n(y)) - c(x,y) \big),
\]
where
\[
u_n(x) := \frac{n}{k} U_{n, \lfloor kx \rfloor} \quad \text{and} \quad v_n(y) := \frac{n}{k} V_{n, \lfloor ky \rfloor},
\]
and $U_{n, k}$ and $V_{n, k}$ denote the $k$th order statistics of $F_1(X_1),\dots,F_1(X_n)$ and $F_2(Y_1),\dots,F_2(Y_n)$, respectively with $U_{n,0} = V_{n,0} = 0$. The core difficulty is to show that the difference $c(u_n(x),u_n(y)) - c(x,y)$ is negligible. Under the assumption of the existence and continuity of partial derivatives of $c$ on $[0,T]^2$ made in \cite{DDFH2004} this is a direct consequence of the fact that under asymptotic independence $\sqrt{m}(u_n(x)-x) = o_P(1)$. Dropping this assumption considerably complicates the theoretical analysis. The proof strategy is to derive bounds on increments of $c(x,y)$ for $x,y$ close to $0$ where the partial derivatives of $c$ can become unbounded (see \Cref{bound increments c,bound_c}) and to combine those bounds with subtle results on weighted weak convergence of $u_n(x)-x$ as a process in $x$; see \Cref{Csorgo} where we essentially leverage the findings of \cite{CH1987}.
\end{rem}

We next turn to the case of asymptotic dependence. Results on convergence of $\hat c_n$ in the space $\ell^\infty$ are well known under this regime; they are equivalent to similar results about estimated stable tail dependence functions \citep[cf.][]{H1992}. However, they require the existence and continuity of partial derivatives of $\ell$ or, equivalently, $c$. As shown in \cite{EKS2008,EKS2012}, the latter condition is restrictive and in fact not necessary to derive asymptotic normality of M-estimators.

The treatment of M-estimators in \cite{EKS2008,EKS2012} involves a direct analysis of certain integrals without using process convergence in $\ell^\infty([0,T]^d)$. While this approach could be transferred to our setting, we will instead follow a strategy put forward in \cite{BSV2014} and prove weak convergence of $\hc_n$ with respect to the hypi-metric introduced therein. This approach will turn out to generalize much more easily when we deal with spatial estimation problems. Convergence with respect to this metric holds without any assumptions on the existence of partial derivatives and is sufficiently strong to guarantee convergence of integrals which is needed for the analysis of M-estimators. 

Let $\dot c_1$ denote the partial derivative of $c$ with respect to $x$ from the left and $\dot c_2$ denote its partial derivative with respect to $y$ from the right. Under asymptotic dependence, $c(x, y) \propto x + y - \ell(x, y)$ is concave since $\ell$ is convex \citep[][Proposition 6.1.21]{DF2006}, hence those directional partial derivatives exist everywhere on $(0, \infty)^2$, by Theorem 23.1 of \cite{R1970}. The definition can be extended to $[0, \infty)^2$ be setting $\dot c_1(0, y)$ to be the derivative from the right instead of from the left.

To describe the limiting distribution, recall that $\chi = \lim_{t \to 0} q(t)/t \in [0, 1]$ is positive only in the case of asymptotic dependence. For $(x, y), (x', y') \in [0, \infty)^2$, define
\begin{equation} \label{cov matrix}
\Lambda((x, y), (x', y')) = \begin{bmatrix}
c(x \wedge x', y \wedge y') & \chi c(x \wedge x', y) & \chi c(x, y \wedge y') \\
\chi c(x \wedge x', y') & \chi (x \wedge x') & \chi^2 c(x, y') \\
\chi c(x', y \wedge y') & \chi^2 c(x', y) & \chi (y \wedge y')
\end{bmatrix},
\end{equation}
and let $(W, W^{(1)}, W^{(2)})$ be an $\R^3$-valued, zero mean Gaussian process on $[0, \infty)^2$ with covariance function $\Lambda$. Note that $W$ is the limiting process in \cref{big lemma AI}, that $W^{(1)}(x, y)$ is constant in $y$ and that $W^{(2)}(x, y)$ is constant in $x$.

\begin{thm}[Asymptotic normality of $\hc_n$ under asymptotic dependence] \label{big lemma AD}

Assume \Cref{con-proc}. Then under asymptotic dependence, i.e., when $\chi > 0$,
\[
W_n \wc B := W - \dot c_1 W^{(1)} - \dot c_2 W^{(2)}
\]
in $(L^\infty([0, T]^2), d_\text{hypi})$, for any $T < \infty$. Here, $W_n$ is defined as in \cref{big lemma AI}.  {The same remains true if $k$ is replaced by $\hat k$ as described after \cref{con-proc}.}

\end{thm}

Note that weak convergence in the above theorem takes place in $(L^\infty([0, T]^2), d_\text{hypi})$ where $L^\infty([0, T]^2)$ corresponds to equivalence classes of functions in $\ell^\infty([0, T]^2)$ with respect to the hypi-(semi-)metric $d_\text{hypi}$, see~\cite{BSV2014} for additional details.

The proof of \Cref{big lemma AD} follows by adapting the arguments given in \cite{BSV2014} for the function $\ell$ and builds on the fact that under asymptotic dependence the function $c$ is differentiable almost everywhere. Note however that, in contrast to similar results in~\cite{BSV2014}, our limiting process is stated without appealing to lower semi-continuous extensions. This type of statement is inspired by the representation of certain integrals in \cite{EKS2012} and is possible {in the bivariate setting} due to concavity of $c$ under asymptotic dependence. Additional comments on the representation of the limiting process are given in \cref{rem-integrals} below. 

\begin{rem} \label{rem-integrals}
In order to obtain asymptotic results for our M-estimator, weak convergence of $\int gW_n \d\leb$ to $\int gB \d\leb$ is sufficient. Under asymptotic dependence, this is seen to follow from \cref{big lemma AD} (see the proof of \cref{big thm}). However, this process convergence result is not necessary. An approach that is used in \cite{EKS2012} is to write an expression for the random vector $\int gW_n \d\leb$ and directly work out its weak limit. With this strategy, $\dot c_j$ may be defined as left or right derivatives without problem as $\int \dot c_j W^{(j)} \d\leb$ will be unchanged. In contrast, proving weak hypi-convergence of $W_n$ to $B$ makes our results  {more general and more easily generalized to the spatial framework}. The cost of doing so is that the  {directional} derivatives $\dot c_j$ must be chosen in a specific way; see \cref{lemma:Hadamard}.
\end{rem}

\begin{rem} \label{rem-hypi1}
Recall that under asymptotic independence, process convergence of $\tilde c_n$ could be obtained from \cref{big lemma AI} by a simple application of the delta method. This is no longer the case in the general setting of \cref{big lemma AD} because weak convergence with respect to the hypi-metric does not imply convergence of $W_n(1,1)$, unless the limiting process $B$ has sample paths which are a.s.~continuous in $(1,1)$. The latter happens only if the partial derivatives of $c$ exist and are continuous in $(1,1)$.  {Under this additional assumption} convergence of $\tilde c_n$ with respect to the hypi-metric can be obtained.
\end{rem}

\subsubsection{Asymptotic theory for bivariate M-estimators} \label{sec:bivasymp}

Equipped with the process convergence tools from the previous section, we proceed to analyze the M-estimator introduced in \Cref{estimation-par}. Consistency is established by standard arguments,  {and for the sake of brevity we do not state the corresponding results here.}
In the present section, we focus on the asymptotic distribution. Define the objective function $\Psi: \Theta \times \R_+ \to \Psi(\Theta \times \R_+) \subseteq \R^q$ by
\begin{equation} \label{Psi}
\Psi(\theta, \sigma) := \sigma \int g c_\theta \ \d\leb - \int g c \ \d\leb.
\end{equation}
Clearly, $\Psi(\theta_0, 1) = 0$. In addition, assume that $(\theta_0, 1)$ is a unique, well separated zero of $\Psi$ and let $J_\Psi(\theta, \sigma)$ denote the Jacobian matrix of $\Psi$ for points $(\theta, \sigma) \in \Theta \times \R_+$ where it exists.

Define $\Gamma((x, y), (x', y')) = c(x \wedge x', y \wedge y')$ under asymptotic independence and otherwise
\[
\Gamma((x, y), (x', y')) = (1, -\dot c_1(x, y), -\dot c_2(x, y)) \Lambda((x, y), (x', y')) (1, -\dot c_1(x', y'), -\dot c_2(x', y'))^\top,
\]
where $\Lambda$ is defined in \cref{cov matrix}. Recall from the previous section that these directional derivatives always exist when $\chi>0$ since in this case $c$ is concave. In fact, $\Gamma((x, y), (x', y'))$ is the covariance between $W(x, y)$ and $W(x', y')$ (under asymptotic independence) or between $B(x, y)$ and $B(x', y')$ (under asymptotic dependence). Hence in those two regimes,
\[
A := \int_{[0, T]^4} g(x, y) g(x', y')^\top \Gamma((x, y), (x', y')) \d x \d y \d x' \d y' \in \R^{q \times q}
\]
is the covariance matrix of the random vector $\int g W \d\leb$ or $\int g B \d\leb$, respectively. We are now ready to state the main result of this section: asymptotic normality of $(\hth_n, \hat\zeta_n)$, which holds under both asymptotic dependence and independence.

\begin{thm}[Asymptotic normality of $\hth_n$] \label{big thm}
Assume that $\Psi$ has a unique, well separated zero at $(\theta_0, 1)$ and is differentiable at that point with Jacobian $J := J_\Psi(\theta_0, 1)$ of full rank $p+1$, $p = \dim(\Theta)$. Further assume \Cref{con-proc}. Then the estimators $(\hth_n, \hat\zeta_n)$ defined in \cref{eq:hatthetazeta} satisfy
\[
\sqrt{m} \Big( \Big( \hth_n, \frac{n\hat\zeta_n}{m} \Big) - (\theta_0, 1) \Big) \wc N(0, \Sigma)
\]
where $\Sigma := (J^\top J)^{-1} J^\top A J (J^\top J)^{-1}$.  {The same remains true if $k$ is replaced by $\hat k$ as described after \cref{con-proc}.}
\end{thm}

While for simplicity the estimator is defined as an exact minimizer, the same result can be obtained for an approximate minimizer. Precisely, it is obvious from the proof of \cref{big thm} that as long as $\Psi_n^*(\hth_n, \hat\zeta_n) = \inf_{\theta, \zeta} \Psi_n^*(\theta, \zeta) + \op{\sqrt{m}/n}$, the conclusion still holds.
Finally, recall that the coefficient of tail dependence $\eta$ can be recovered from the function $c$ since the latter is homogeneous of order $1/\eta$, and this relation always holds. Therefore, inside the assumed parametric model, $\eta$ can be represented as a function $\eta(\theta)$.  {The asymptotic distribution of the resulting estimator can be obtained by a direct application of the delta method and details are omitted for the sake of brevity.}

\subsection{The spatial setting}
\label{sec:asy-spatial}

In this section we assume the framework of \Cref{estimation spatial} and establish asymptotic properties of the estimators therein. For each pair $s \in \cP$, let $k^{(s)}$ be an intermediate sequence and define
\[
\hc_n^{(s)}(x, y) := \frac{\hat Q_n^{(s)} \left(k^{(s)} x/n, k^{(s)} y/n \right)}{q^{(s)}(k^{(s)}/n)}.
\]
From \cref{sec:asy-non}, the asymptotic distribution of $\hc_n^{(s)}$ is known under suitable conditions. However, as the spatial estimators $\hat\vartheta_n$ and $\tilde\vartheta_n$ are based on all pairs, a joint convergence statement about all processes $\hc_n^{(s)}$ is necessary. This will require an additional assumption which we present and discuss next.

Let $F^{(1)}, \dots, F^{(d)}$ denote the marginal distribution functions of the random vector $X$, which itself consists of the spatial process $Y$ evaluated at $d$ different locations. In order to obtain the asymptotic covariance between different processes $\hat c_n^{(s)}$, we need to ensure that certain multivariate tail probabilities converge. Partition the set $\cP$ into $\cP_I$ and $\cP_D$, consisting of the asymptotically independent and asymptotically dependent pairs, respectively. In the formulation of the following assumption, $s = (s_1, s_2)$ and $s^i = (s_1^i, s_2^i)$ are used to denote pairs. For brevity, $x^i = (x_1^i, x_2^i)$ is also used to denote a point in $[0, \infty)^2$. 

\begin{con} \label{con-spatial-new}

For every $s \in \cP$, $(X^{(s_1)}, X^{(s_2)})$ satisfies \cref{con-proc}(i) with functions $q^{(s)}, q_1^{(s)}, c^{(s)}$ and $\chi^{(s)} := \lim_{t \downarrow 0} q^{(s)}(t)/t$ exists. Intermediate sequences $k^{(s)}$ are chosen so that $m^{(s)} := nq^{(s)}(k^{(s)}/n) \to \infty$ and $\sqrt{m^{(s)}} q_1^{(s)}(k^{(s)}/n) \to 0$. For pairs $s^1, s^2 \in \cP$, points $x^1, x^2 \in [0, \infty)^2$ and sets $J$ of two-dimensional vectors with entries in $\{1, 2\}$, let
\[
\Gamma_n\left( s^1, s^2, x^1, x^2; J \right) = \frac{n}{\sqrt{m^{(s^1)} m^{(s^2)}}} \mathbb{P}\Big( F^{(s_j^{i})}(X^{(s_j^{i})}) \geq 1 - \frac{k^{(s^{i})} x_{j}^i}{n}, \quad (i, j) \in J \Big).
\]
We assume that the sequences $k^{(s)}$ are chosen such that the limits
\begin{align*}
\Gamma^{(s^1, s^2)}(x^1, x^2) &:= \lim_{n \to \infty} \Gamma_n\left(s^1, s^2, x^1, x^2; \{(1,1), (1,2), (2,1), (2,2)\} \right), ~ s^1, s^2 \in \cP, 
\\
\Gamma^{(s^1, s^2, j)}(x^1, x^2) &:= \chi^{(s^2)} \lim_{n \to \infty} \Gamma_n\left(s^1, s^2, x^1, x^2; \{(1,1), (1,2), (2,j)\} \right), ~~ s^1 \in \cP, s^2 \in \cP_D, 
\\
\Gamma^{(s^1, j^1, s^2, j^2)}(x^1, x^2) &:= \chi^{(s^1)} \chi^{(s^2)} \lim_{n \to \infty} \Gamma_n\left(s^1, s^2, x^1, x^2; \{(1,j^1), (2,j^2)\} \right), ~ s^1, s^2 \in \cP_D,
\end{align*}
exist for all $j, j^i \in \{1, 2\}$, and that the convergence is locally uniform over $x^1, x^2 \in [0, \infty)^2$.

\end{con}

We next discuss the above condition in three special cases of particular interest. The first two are processes in the domain of attraction of max-stable processes and inverted max-stable processes. The third one is a mixture process appearing in \cite{WT2012}, which can have asymptotically dependent and independent pairs simultaneously.

\begin{ex}[\textit{\cref{ex1.0}, continued}] \label{ex_spatial:ms}
If $Y$ is in the max-domain of attraction of a max-stable process, then
$X$ is in the max-domain of attraction of a max-stable distribution $G$ on $\R^d$ with stable tail dependence function
\[
\ell(x_1, \dots, x_d) := \lim_{t \downarrow 0} \frac{1}{t} \Prob{}{F^{(1)}(X^{(1)}) \geq 1-tx_1 \text {  or } \dots \text{ or  } F^{(d)}(X^{(d)}) \geq 1-tx_d}, \quad x_j \geq 0;	
\]
see \cref{eq:l}. If moreover the convergence is locally uniform over $(x_1, \dots, x_d) \in [0, \infty)^d$ and if every pair is asymptotically dependent, then \cref{con-spatial-new} holds. Note that this is automatically satisfied if $Y$ itself is max-stable. The sequences $k^{(s)}$ can be chosen all equal to $k$, say, and for every pair $s$, $m^{(s)}/k \to \chi^{(s)} > 0$. The sequences $m^{(s)}$ can also be chosen all asymptotically equivalent to $m$, say, by choosing $k^{(s)} = m/\chi^{(s)}$. The limiting covariance terms can all be deduced from $\ell$ by straightforward calculations.
\end{ex}

\begin{ex}[\textit{\cref{ex1.1}, continued}]
\label{ex_spatial:ims}
If $Y$ is an inverted max-stable process, then $X$ has an inverted max-stable distribution, and we assume that the associated stable tail dependence function $\ell$ is component-wise strictly increasing. The latter is trivially satisfied if $X$ has a positive density. Then if all the pairwise functions $\ell^{(s)}$ satisfy the quadratic expansion introduced in \cref{ex:ims:cont}, \cref{con-spatial-new} is satisfied and the sequences $k^{(s)}$ can be chosen so that the $m^{(s)}$ are all equal, that is, for every pair $s\in\cP$, $m^{(s)} = m$ for some intermediate sequence $m$. Here, $\cP_D$ is empty so the only required covariance terms are (see \cref{proof IMS})
\[
\Gamma^{(s^{1}, s^{2})}(x^1, x^2) = \begin{cases}
c^{(s)}(x_1^1 \wedge x_1^2, x_2^1 \wedge x_2^2), \quad &s^1 = s^2 = s, \\
0, \quad &s^1 \neq s^2
\end{cases}.
\]

For instance, any inverted Brown--Resnick process (or rather the implied inverted $d$-dimensional H\"usler--Reiss distribution corresponding to the $d$ observed locations) satisfies \cref{con-spatial-new} as long as the aforementioned $d$-variate distribution has a density. The latter can easily be checked \citep[e.g.,][Corollary 2]{eng2018a}.
\end{ex}

\begin{ex}[\textit{\cite{WT2012}, Section 4}]
\label{ex_spatial:mix}
Let $Z$ be a max-stable process and $Z'$ be an inverted max-stable process, both with unit Fr\'echet margins. Suppose that $Z'$ satisfies the monotonicity condition stated in \cref{ex_spatial:ims}, and additionally that none of its pairwise distributions $(Z'(u_1), Z'(u_2))$ is perfectly independent. Let $a \in (0, 1)$ and define the process $Y$ by
\[
Y(u) := \max\{aZ(u), (1-a)Z'(u)\}.
\]
Then $Y$ also has unit Fr\'echet margins. 
If $Z$ becomes independent at a certain spatial distance, the process $Y$ transitions between asymptotic dependence and independence at that distance. An instance of such a max-stable process $Z$ is found in the second example after Theorem~1 of \cite{S2002}, assuming that the Radius $R$ of the random disks is bounded \citep[see also][eq.~(23) and the discussion that precedes]{DPR2012}.

The process $Y$ can be shown to satisfy \cref{con-spatial-new} if the sequences $k^{(s)}$ are chosen so that the $m^{(s)}$ are all equal. The terms $\Gamma^{(s^1, s^2)}$, $\Gamma^{(s^1, s^2, j)}$ and $\Gamma^{(s^1, j^1, s^2, j^2)}$ are mostly determined by the process $Z$, as in \cref{ex_spatial:ms}; see \cref{proof IMS} in the supplement for details.
\end{ex}

\subsubsection{Joint distribution of non-parametric estimators}

The joint limiting behavior of the processes $\hc_n^{(s)}$ relies on $\left( (W^{(s)})_{s \in \cP}, (W^{(s, j)})_{s \in \cP_D, j \in \{1, 2\}} \right)$, a collection of centered Gaussian processes on $[0, \infty)^2$. The covariance between $W^{(s)}(x, y)$ and $W^{(s')}(x', y')$ is given by $\Gamma^{(s, s')}((x, y), (x', y'))$, the covariance between $W^{(s)}(x, y)$ and $W^{(s', j)}(x', y')$ takes the form $\Gamma^{(s, s', j)}((x, y), (x', y'))$, and the covariance between $W^{(s, j)}(x, y)$ and $W^{(s', j')}(x', y')$ is equal to $\Gamma^{(s, j, s', j')}((x, y), (x', y'))$. For $s \in \cP_I$, let $B^{(s)} = W^{(s)}$ and for $s \in \cP_D$, let
\[
B^{(s)} = W^{(s)} - \dot c_1^{(s)} W^{(s, 1)} - \dot c_2^{(s)} W^{(s, 2)},
\]
where $\dot c_j^{(s)}$ are defined similarly to $\dot c_j$ in \cref{sec:asy-non}.

\begin{thm}[Asymptotic normality of $\hc_n^{(s)}$] \label{thm spatial np}
Assume \cref{con-spatial-new}. Then
\[
\big( W_n^{(s)} \big)_{s \in \cP} := \big( \sqrt{m^{(s)}} (\hc_n^{(s)} - c^{(s)}) \big)_{s \in \cP} \wc \big( B^{(s)} \big)_{s \in \cP}
\]
in the product space $\left( L^\infty([0, T]^2), d_\text{hypi} \right)^{|\cP|}$, for any $T < \infty$.  {The same remains true if each $k^{(s)}$ is replaced by the data-dependent sequence $\hat k^{(s)}$ as described after \cref{con-proc}.}
\end{thm}

The preceding result can be applied in all generality as long as the four-dimensional tails of the spatial process of interest are sufficiently smooth. The admissible settings include, but are far from limited to, \cref{ex_spatial:ms,ex_spatial:ims,ex_spatial:mix}.

According to \cite{BSV2014}, convergence in the hypi-metric is equivalent to uniform convergence when the limit is a continuous function. The process $B^{(s)}$ clearly has almost surely continuous sample paths under asymptotic independence, as well as under asymptotic dependence if the partial derivatives of $c$ exist everywhere and are continuous. It follows that in those cases $W_n^{(s)}$ converges in $(\ell^\infty([0, T]^2),\|\cdot\|_\infty)$. In fact, one may replace the product space in the result above by $\otimes_{s \in \cP} \DD^{(s)}$, where $\DD^{(s)}$ represents either $\ell^\infty([0, T]^2)$ equipped with the supremum distance (if $s \in \cP_I$ or $c$ has continuous partial derivatives) or $L^\infty([0, T]^2)$ equipped with the hypi-metric (otherwise). In particular, for processes where every pair is asymptotically independent such as inverted max-stable processes, the hypi-metric can be replaced by the supremum distance everywhere.

\subsubsection{Asymptotics for parametric estimators}

We now show how \cref{thm spatial np} leads to asymptotic results for the parametric estimators $\hat\vartheta_n$ and $\tilde\vartheta_n$ introduced in \cref{eq:defthetahat,eq:defthetahatls}. Recall the setting of \Cref{estimation spatial}, and in particular the functions $h^{(s)}: \Theta \to \tilde\Theta$ and the relation $c^{(s)} = c_{h^{(s)}(\vartheta_0)}$. Similarly to the bivariate setting, define 
\[
\Psi^{(s)}: \tilde\Theta \times \R_+ \to \R^q, \quad \Psi^{(s)}(\theta, \sigma) = \sigma \int g c_\theta \d\leb - \int g c^{(s)} \d\leb.
\]

In the bivariate setting, we required $\Psi$ to be differentiable and have a unique well-separated zero. In the spatial setting we need a  comparable assumption.

\begin{con} \label{con-spatial-par}
For every pair $s \in \cP$, the functions $\Psi^{(s)}$ and $h^{(s)}$ are continuously differentiable at the points $(h^{(s)}(\vartheta_0), 1)$ and $\vartheta_0$, respectively, with Jacobian matrices $J_{\Psi^{(s)}}(h^{(s)}(\vartheta_0), 1)$ and $J_{h^{(s)}}(\vartheta_0)$ of full ranks $\tilde p + 1$ and $p$.  {Additionally (i) or (ii) holds.}
\begin{itemize}
\item[(i)] The functions $\Psi^{(s)}$ and $\vartheta \mapsto (h^{(s)}(\vartheta) - h^{(s)}(\vartheta_0))_{s \in \cP}$ have a unique, well separated zero at the points $(h^{(s)}(\vartheta_0), 1)$ and $\vartheta_0$, respectively.
\item[(ii)] The function $(\vartheta, \sigma) \mapsto (\Psi^{(s)}(h^{(s)}(\vartheta), \sigma^{(s)}))_{s\in \cP}$ as a function on $\Theta \times \R_+^{|\cP|}$ has a unique, well separated zero at the point $(\vartheta_0, 1, \dots, 1)$.
\end{itemize}
\end{con}

Assuming both parts of \cref{con-spatial-par}, we now introduce the notation that is needed to define the limiting covariance matrices of the two estimators. In the following, elements of a vector $x \in \R^{q|\cP|}$ are ordered by pair $s \in \cP$ first, and then by dimension $j \in \{1, \dots, q\}$. The same convention is used when ordering the rows or columns of a matrix.

Letting $B^{(s)}$ denote the limiting Gaussian processes appearing in \cref{thm spatial np}, consider the matrix $A \in \R^{q|\cP| \times q|\cP|}$ with blocks of the form
\[
A^{(s, s')} := \int_{[0, T]^4} g(x, y) g(x', y')^\top \cov{}{B^{(s)}(x, y)}{B^{(s')}(x', y')} \d x \d y \d x' \d y'.
\]
Let $\cD \in \R^{\tilde p |\cP| \times q|\cP|}$ be a block-diagonal matrix with blocks given by
\begin{equation}\label{eq:defcDs}
\cD^{(s)} := \left[ \left( J_{\Psi^{(s)}}(h^{(s)}(\vartheta_0), 1)^\top J_{\Psi^{(s)}}(h^{(s)}(\vartheta_0), 1) \right)^{-1} J_{\Psi^{(s)}}(h^{(s)}(\vartheta_0), 1)^\top \right]_{1:\tilde p, 1:q} \in \R^{\tilde p \times q},
\end{equation}
where $s \in \cP$ and $[M]_{1:\tilde p, 1:q}$ indicates the sub-matrix consisting of rows 1 to $\tilde p$ and columns 1 to $q$ of the matrix $M$. Define $J_1 \in \R^{\tilde p |\cP| \times p}$ by stacking the matrices $J_{h^{(s)}}(\vartheta_0)$, $s \in \cP$, on top of each other. Denote by $(e^{(s)})^\top$ the unit vector in $\R^{|\cP|}$ with a one in the position corresponding to the pair $s$ and let $J_2 \in \R^{q|\cP| \times (p + |\cP|)}$ be obtained by stacking the matrices
\[
J_{\Psi^{(s)}}(h^{(s)}(\vartheta_0), 1) \begin{bmatrix}
J_{h^{(s)}}(\vartheta_0) & 0 \\
0 & e^{(s)}
\end{bmatrix} \in \R^{q \times (p + |\cP|)}, \quad s \in \cP,
\]
on top of each other. Finally, define
\[
\Sigma_1 = (J_1^\top J_1)^{-1} J_1^\top \cD A \cD J_1 (J_1^\top J_1)^{-1}, \quad\quad \Sigma_2 = (J_2^\top J_2)^{-1} J_2^\top A J_2 (J_2^\top J_2)^{-1}.
\]

\begin{thm}[Asymptotic normality of the estimators of $\vartheta$] \label{thm spatial ls} \label{thm spatial}

Assume \cref{con-spatial-new} and suppose that the sequences $m^{(s)}$ are all asymptotically equivalent to $m$, say. Then under \cref{con-spatial-par}(i), the estimator defined in \cref{eq:defthetahatls} satisfies
\[
\sqrt{m} \left( \hat\vartheta_n - \vartheta_0 \right) \wc N(0, \Sigma_1)
\]
and under \cref{con-spatial-par}(ii), the estimators defined in \cref{eq:defthetahat} satisfy
\[
\sqrt{m} \Big( \Big( \tilde\vartheta_n, \frac{n \tilde\zeta_n}{m} \Big) - (\vartheta_0, 1, \dots, 1) \Big) \wc N(0, \Sigma_2),
\]
where $\Sigma_1$ and $\Sigma_2$ are as above.  {The same remains true if each $k^{(s)}$ is replaced by the data-dependent sequence $\hat k^{(s)}$, based on the same sequence $m$, as described after \cref{con-proc}.}

\end{thm}

The assumption of asymptotic equivalence of all $m^{(s)}$ can be substantially relaxed. Otherwise, a simple way to satisfy it is to select one $m$ and use data-driven sequences $\hat{k}^{(s)}$.

\section{Simulations} \label{simulations}

\subsection{Bivariate distributions}
\label{sim bivariate}

In this section we study the finite sample behavior of the estimator introduced in the paper. We simulate samples from the bivariate vector $(X+X', Y+Y')$, where $(X,Y)$ is the signal and $(X',Y')$ is and independent noise vector. We consider three different models for the bivariate distributions $(X,Y)$.
\begin{compactenum}
\item[(M1)] The inverted H\"usler--Reiss model from \cref{ex:ims}(i) with unit Fr\'echet margins, whose corresponding class of functions $c$ takes the form $c_\theta(x,y) = (xy)^\theta$ where $\theta \in (1/2, 1]$.
\item[(M2)] The inverted asymmetric logistic model from \cref{ex:ims}(ii) with fixed $r = 2$ and unit Fr\'echet margins. We fit the full parametric model $\{c_\theta(x, y) = x^{\theta_1} y^{\theta_2}: \theta \in \Theta\}$, where $\Theta := \{(\theta_1, \theta_2) \in (0, 1]^2: \theta_1 + \theta_2 > 1\}$, even though due to our choice of $r$ the only attainable parameters are approximately the square $[0.7, 1]^2$; see \cref{Color plots}.
\item[(M3)] The random scale construction from \cref{ex2} where we fix $\alpha_W = 1$ and vary $\alpha_R$. The collection of possible functions $c = c_\lambda$, $\lambda \in (0,2)$ is given in \cref{table}.
\end{compactenum}
\Cref{fig:IHR-rea,fig:IAlog-rea,fig:P-rea} in the supplement show realizations of models M1--M3 corresponding to different parameter values and rescaled to unit exponential margins for illustration.

As a noise vector we simulate samples of $(X',Y')$, where $X'$ and $Y'$ are independent with Pareto distribution function $1-1/x^4$, $x\geq 1$. Note that this tail is lighter than that of the marginal distributions in all three models; it can be shown that this additive noise does not affect the functions $q$ and $c$ of $(X,Y)$.

All of the results that follow are based on $1000$ simulation repetitions and samples of size $n=5000$. In all the simulations, we use the same weight function (represented by $g$ in \cref{eq:Psinstar}), which we now describe. Consider the following rectangles: $I_1 := [0, 1]^2$, $I_2 := [0, 2]^2$, $I_3 := [1/2, 3/2]^2$, $I_4 := [0, 1] \times [0, 3]$ and $I_5 := [0, 3] \times [0, 1]$. The function $g: \R^2 \to \R^5$ is given by
\begin{equation} \label{eq:defg}
g(x, y) := \big(\Ind{(x, y) \in I_1}/a_{1,\theta_\text{REF}},\dots,\Ind{(x, y) \in I_5}/a_{5,\theta_\text{REF}}\big)^\top 
\end{equation}
where $a_{j,\theta_\text{REF}} := \int_{I_j} c_{\theta_\text{REF}} d\leb$ and $\theta_\text{REF}$ is simply a reference point in the parameter space that ensures that all components of $g$ have comparable magnitude. In the three models above, the reference points are $0.6$, $(0.6, 0.6)$ and $1$, respectively. The rectangles are chosen in order to capture various aspects of the function $c$: $I_3$ contains information about the unknown scale $\zeta$ (recall that we scale $c$ so that $c(1,1) = 1$). The rectangles $I_1, I_2$ are geared towards determining homogeneity properties of $c$ since $I_2 = 2 I_1$ and are especially useful for estimating $\eta$. The rectangles $I_4,I_5$ are informative about asymmetry of the function $c$ with respect to its arguments. Different choices of the weight function would be possible, and the best choice will be different for each model under consideration and even for each specific parameter value within a given model class. Nevertheless, the aforementioned choice seems close to optimal for all the models considered here. In \cref{simulations-add} of the supplement, a sensitivity analysis is carried out where we repeat the simulation study with different weight functions that are constructed by considering only some of the rectangles $I_1, \dots, I_5$ instead of all five. See also \cite{EKS2008, EKS2012} for a related discussion in the estimation of stable tail dependence functions.

\subsubsection{The inverted H\"usler--Reiss model (M1)}\label{sec:ihrbiv} 

\cref{Choice of k IHR} shows the effect of $k$ on the estimation performance of $\hat \theta_n$ from \cref{eq:hatthetazeta} in terms of absolute bias and root MSE for the three parameter values $\theta = 0.6$, $0.75$, and $0.9$. We observe that for {larger values of $\theta$ (or smaller values of $\eta$, corresponding to more independence in the extremes)} larger values of $k$ lead to the best RMSE. This is in line with our theory as, for fixed $k$, smaller $\eta$ corresponds to smaller values of $m$ and hence larger asymptotic variance.

\begin{figure}[H]
\centering
\includegraphics[scale = 0.35, trim = 0 50 0 60]{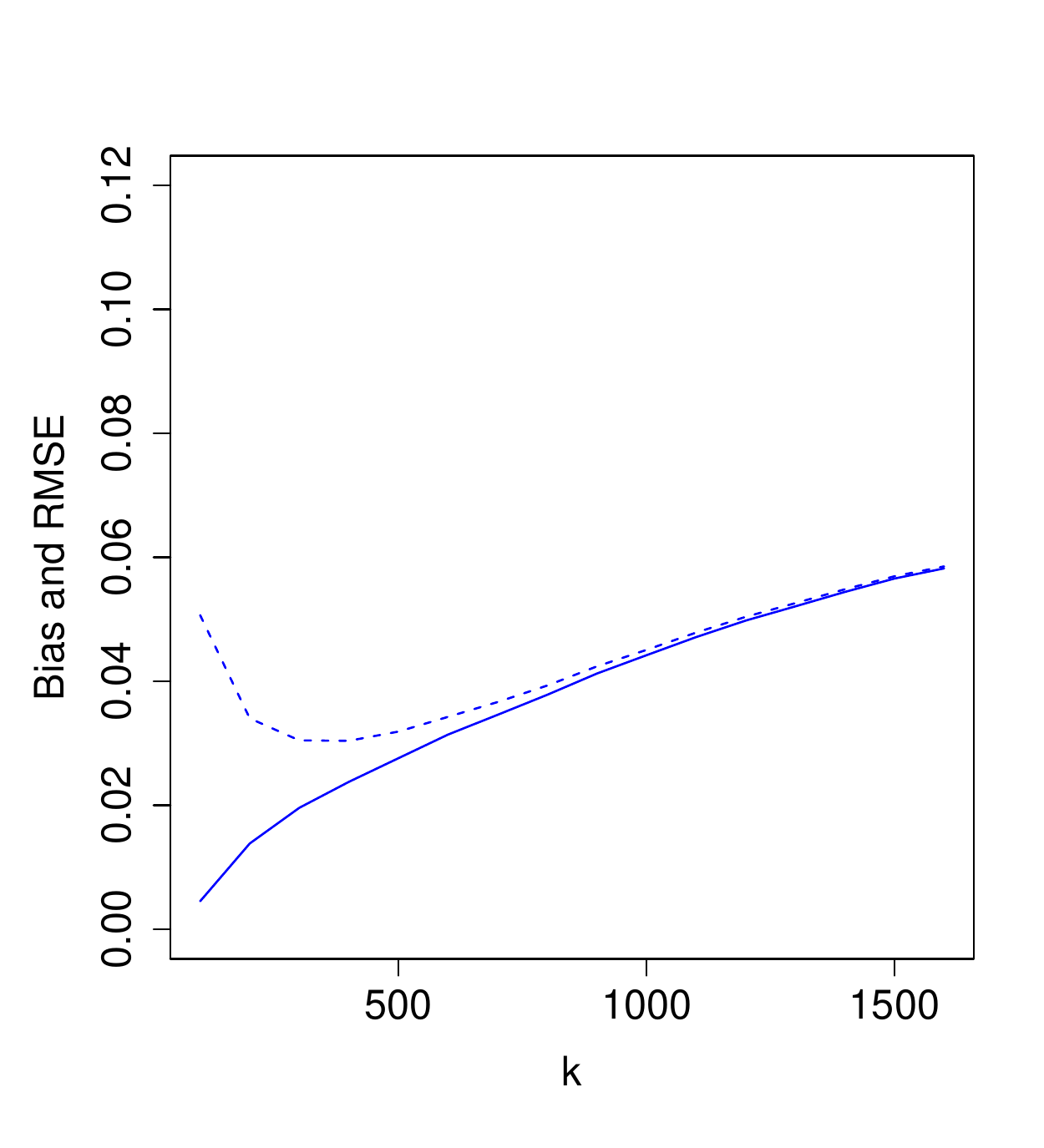}
\includegraphics[scale = 0.35, trim = 0 50 0 60]{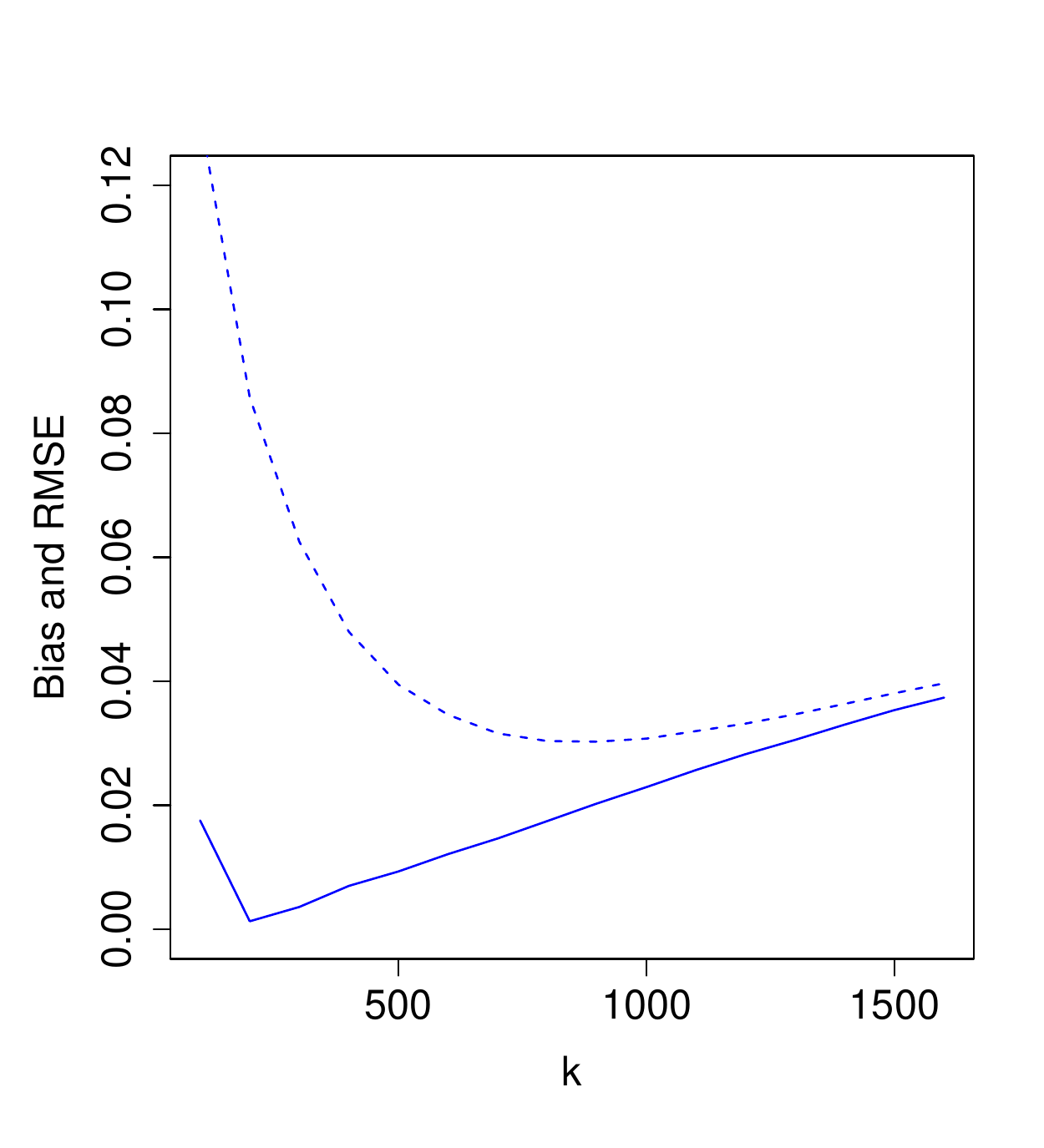}
\includegraphics[scale = 0.35, trim = 0 50 0 60]{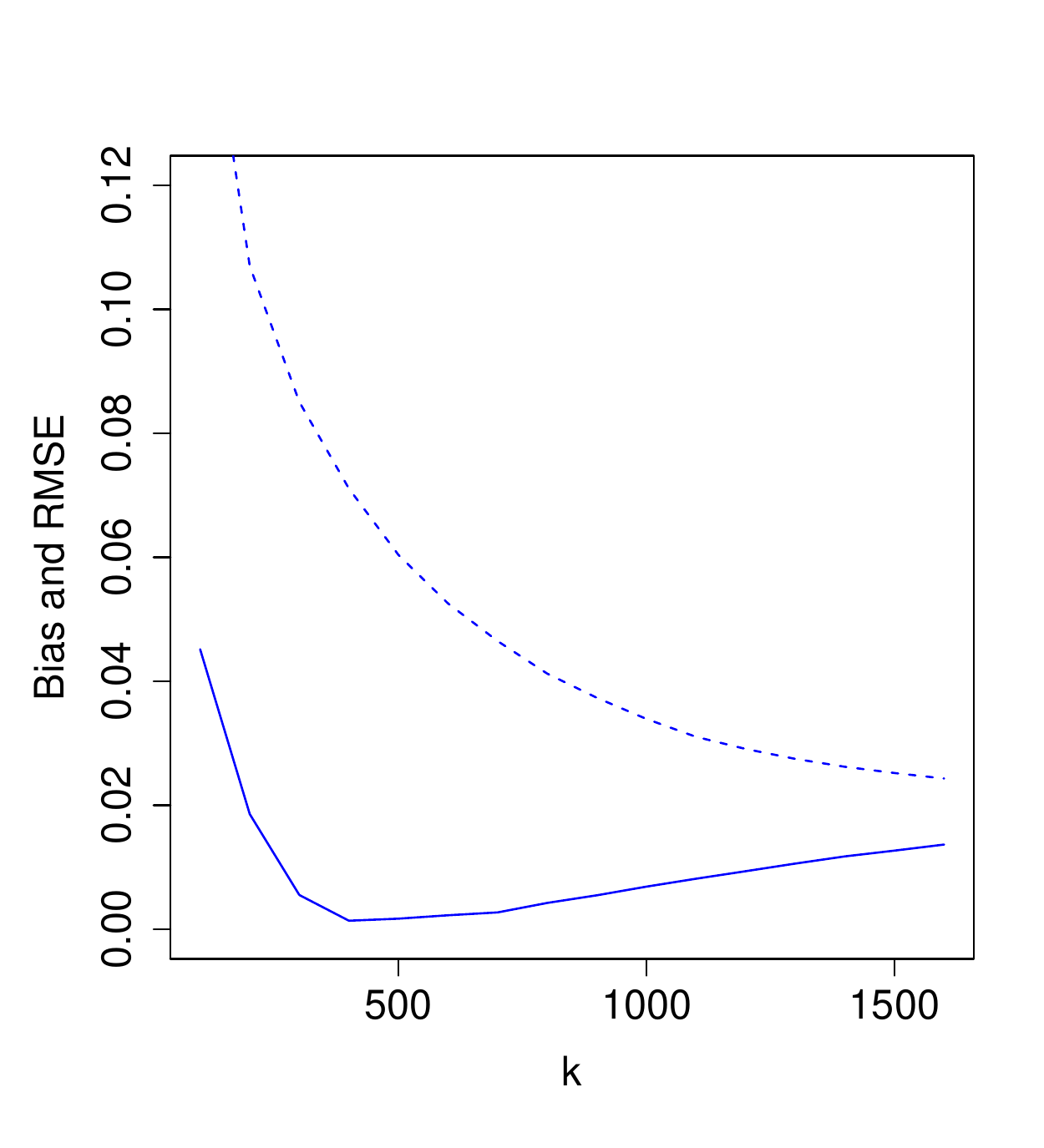}
\caption{Absolute bias (solid lines) and RMSE (dashed lines) of the M-estimator of $\theta$ as a function of $k$, based on 1\,000 samples of size 5\,000 from model M1 with parameter values 0.6, 0.75 and 0.9, from left to right.}
\label{Choice of k IHR}
\end{figure}

An analysis of $\hat \theta_n$ for a finer range of parameter values is provided in \cref{fig:IHR-box}. Motivated by the findings in \cref{Choice of k IHR} we fix $k = 800$; this choice leads to reasonable performance across all parameter values. Overall the results are satisfactory, with a more pronounced negative bias for smaller values of $\theta$ and more variance for increasing $\theta$.

\begin{figure}[H]
\centering
\includegraphics[scale = 0.55, trim = 0 35 0 50]{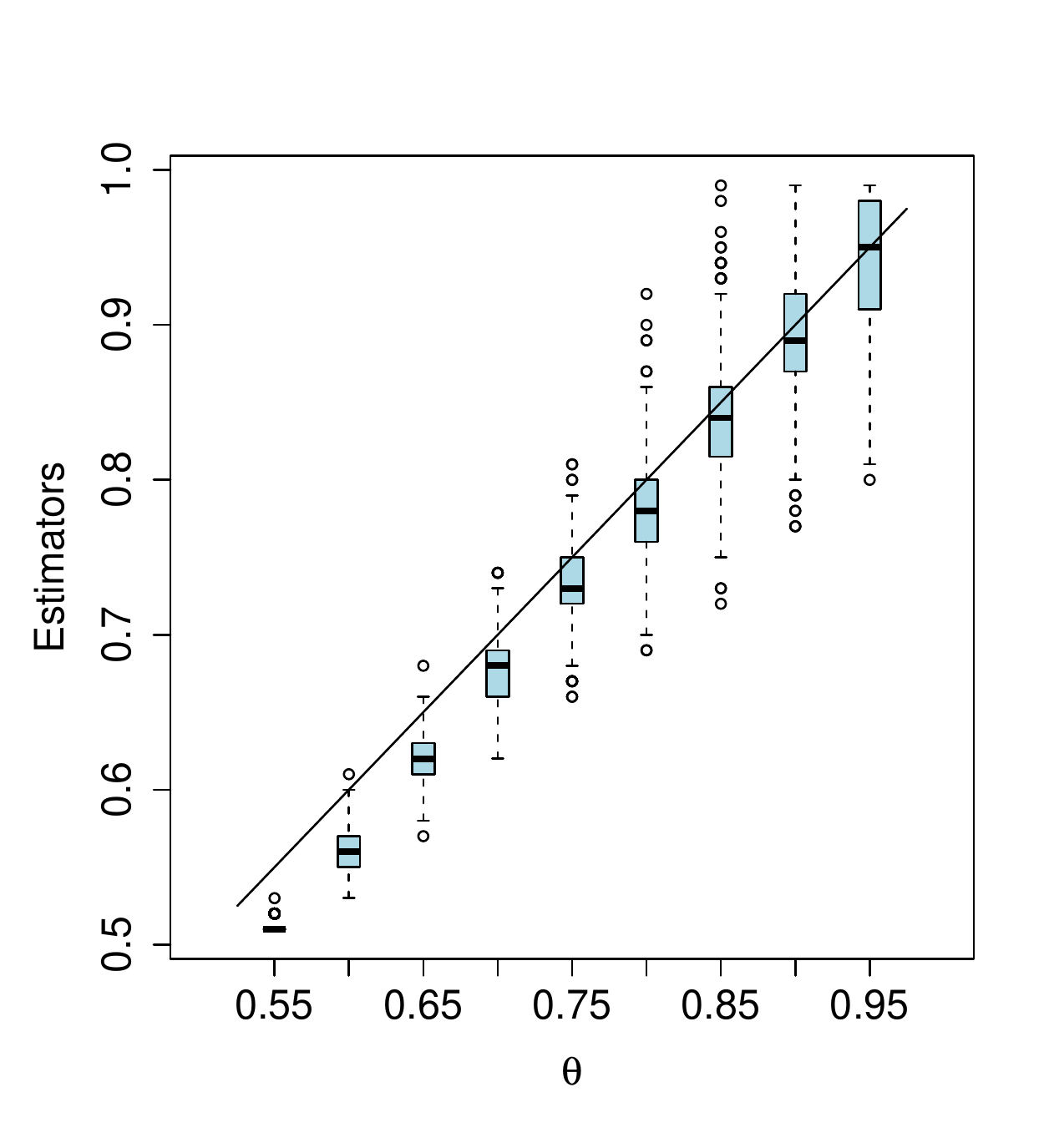}
\caption{Box plots of the M-Estimators of $\theta$ 1\,000 samples of size 5\,000 for each parameter value.} \label{fig:IHR-box}
\end{figure}

\subsubsection{The inverted asymmetric logistic model (M2)}

\cref{Choice of k IALog} shows the impact of $k$ on estimated parameter values for three different choices of $\theta$. Since here the parameter is two-dimensional, we consider (and estimate) the Euclidean bias and RMSE of the estimator $\hth_n$, defined as $\|\mathbb{E}[\hth_n - \theta]\|$ and $(\mathbb{E}\|\hth_n - \theta\|^2)^{1/2}$, respectively.

Similarly to the pattern observed in \cref{Choice of k IHR} we see that smaller values of $\eta$ necessitate larger values of $k$ in order to achieve a good balance between bias and variance.

\begin{figure}[H]
\centering
\includegraphics[scale = 0.35, trim = 0 50 0 60]{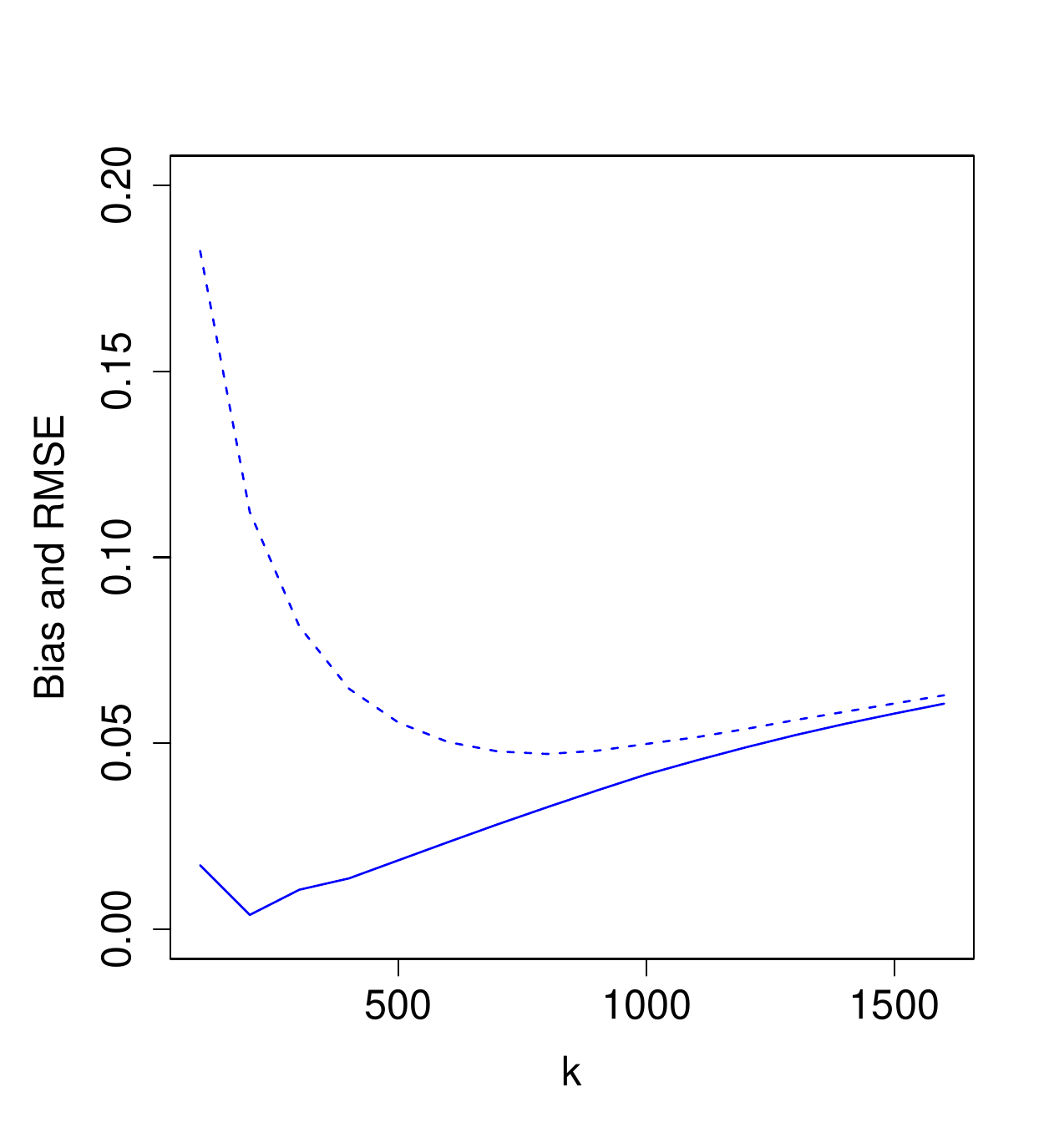}
\includegraphics[scale = 0.35, trim = 0 50 0 60]{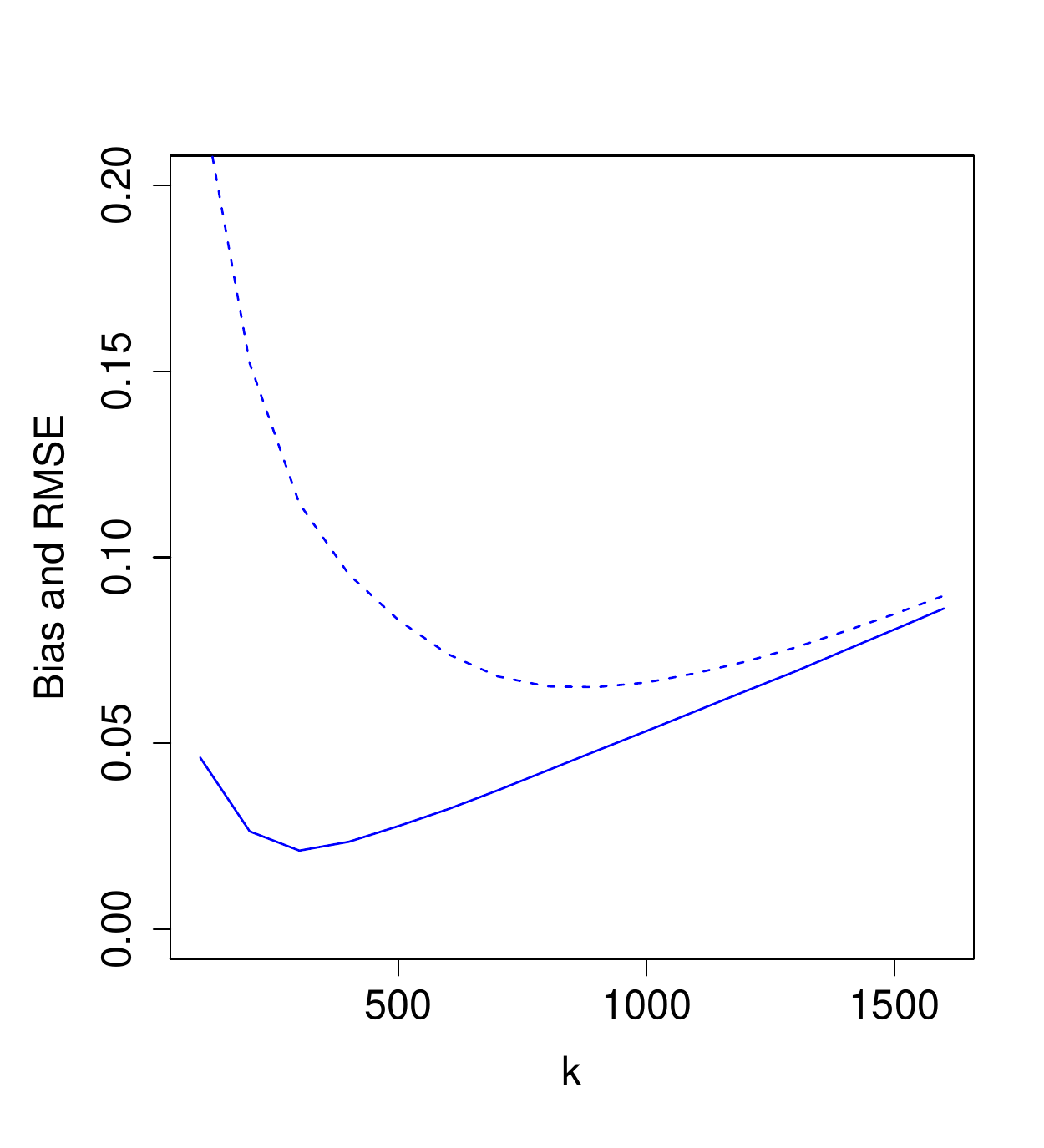}
\includegraphics[scale = 0.35, trim = 0 50 0 60]{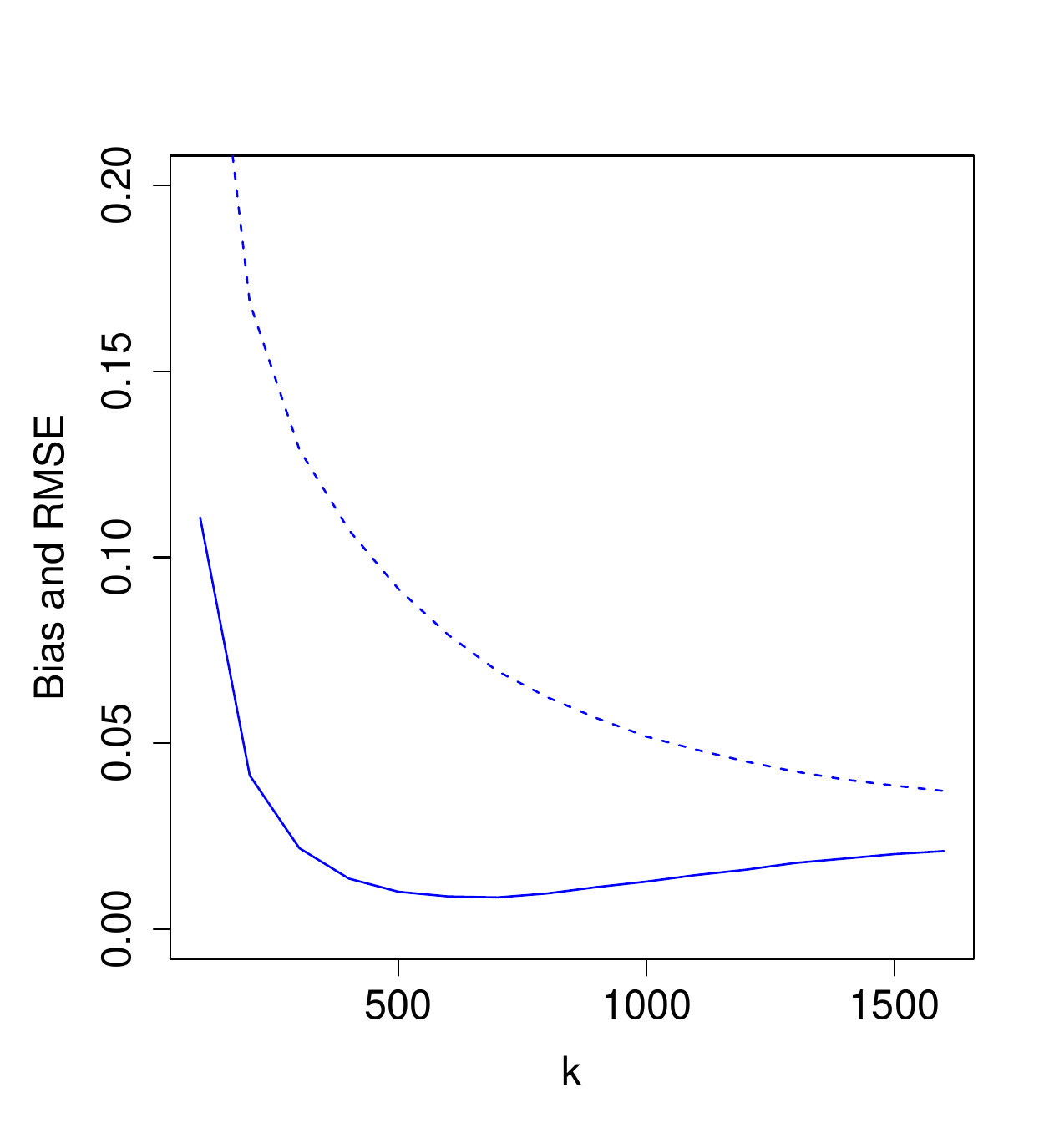}
\caption{Absolute bias (solid lines) and RMSE (dashed lines) of the M-estimator of $\theta$ as a function of $k$, based on 1\,000 samples of size 5\,000 from model M2 with parameter $\theta$ equal to $(0.72, 0.72)$, $(0.75, 0.91)$ and $(0.91, 0.91)$, from left to right. In the original parametrization, the corresponding values of $(\nu, \phi)$ are $(0.94, 0.94)$, $(0.44, 0.94)$ and $(0.31, 0.31)$, respectively.}
\label{Choice of k IALog}
\end{figure}

\cref{Color plots} shows the performance of the proposed M-estimator for a range of different parameters $(\theta_1,\theta_2)$ with {Euclidean bias in the left panel and RMSE in the right panel}; the value $k = 800$ is fixed throughout. Since the relation $(\nu, \phi) \mapsto (\theta_1, \theta_2)$ is not easily invertible, we selected a grid of values of $(\nu, \phi) \in [0, 1]^2$, calculated all the corresponding points $\theta$ and kept the values for which $\theta_j \leq 0.95$, $j=1,2$.

We observe that the estimators perform better for parameter values close to the diagonal, with larger bias and variance for more asymmetric parameter values. The overall estimation accuracy is reasonably good, with worst case RMSE values around 0.07.  

\begin{figure}[H]
\centering
\includegraphics[scale = 0.55, trim = 0 50 0 35]{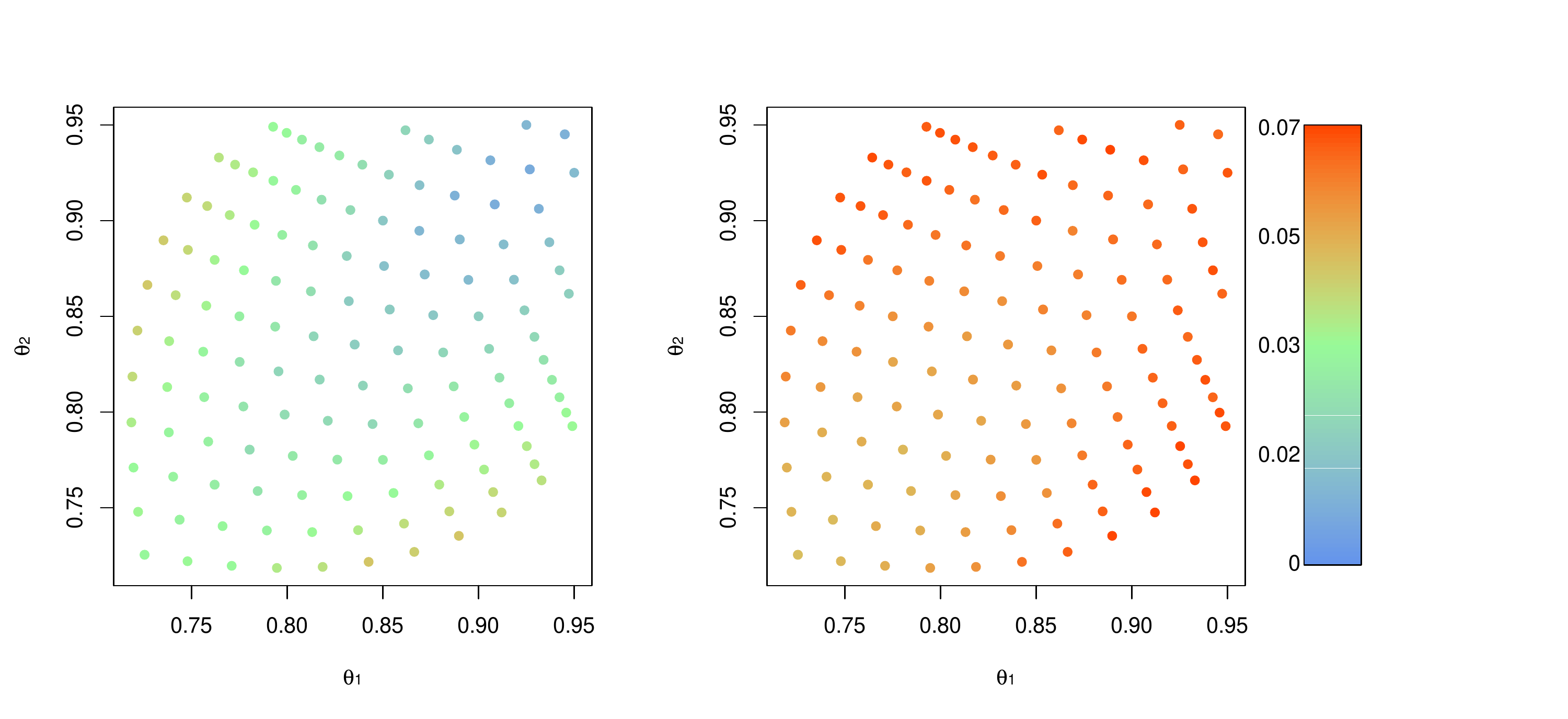}
\caption{Absolute bias (left) and RMSE (right) of the M-estimator of $\theta = (\theta_1, \theta_2)$ as a function of $\theta$, based on 1\,000 samples of size 5\,000 from model M2.}
\label{Color plots}
\end{figure}

\subsubsection{The Pareto random scale model (M3)}

\cref{Choice of k P} shows the effect of $k$ on the performance of our M-estimator $\hat \lambda_n$ in terms of absolute bias and root MSE for the three parameter values $\lambda = 0.4$, $1$, and $1.6$. We notice that the estimator is considerably more biased at $\lambda=1$ than at other parameter values. This is expected as, according to \cref{table}, the bias function $q_1$ vanishes only at a logarithmic rate when $\lambda=1$, compared to a polynomial rate elsewhere. Moreover, like in the other models, we observe that for more independent data (characterized by larger $\lambda$), larger values of $k$ are required to drive down the variance of the estimator.

\begin{figure}[H]
\centering
\includegraphics[scale = 0.35, trim = 0 50 0 60]{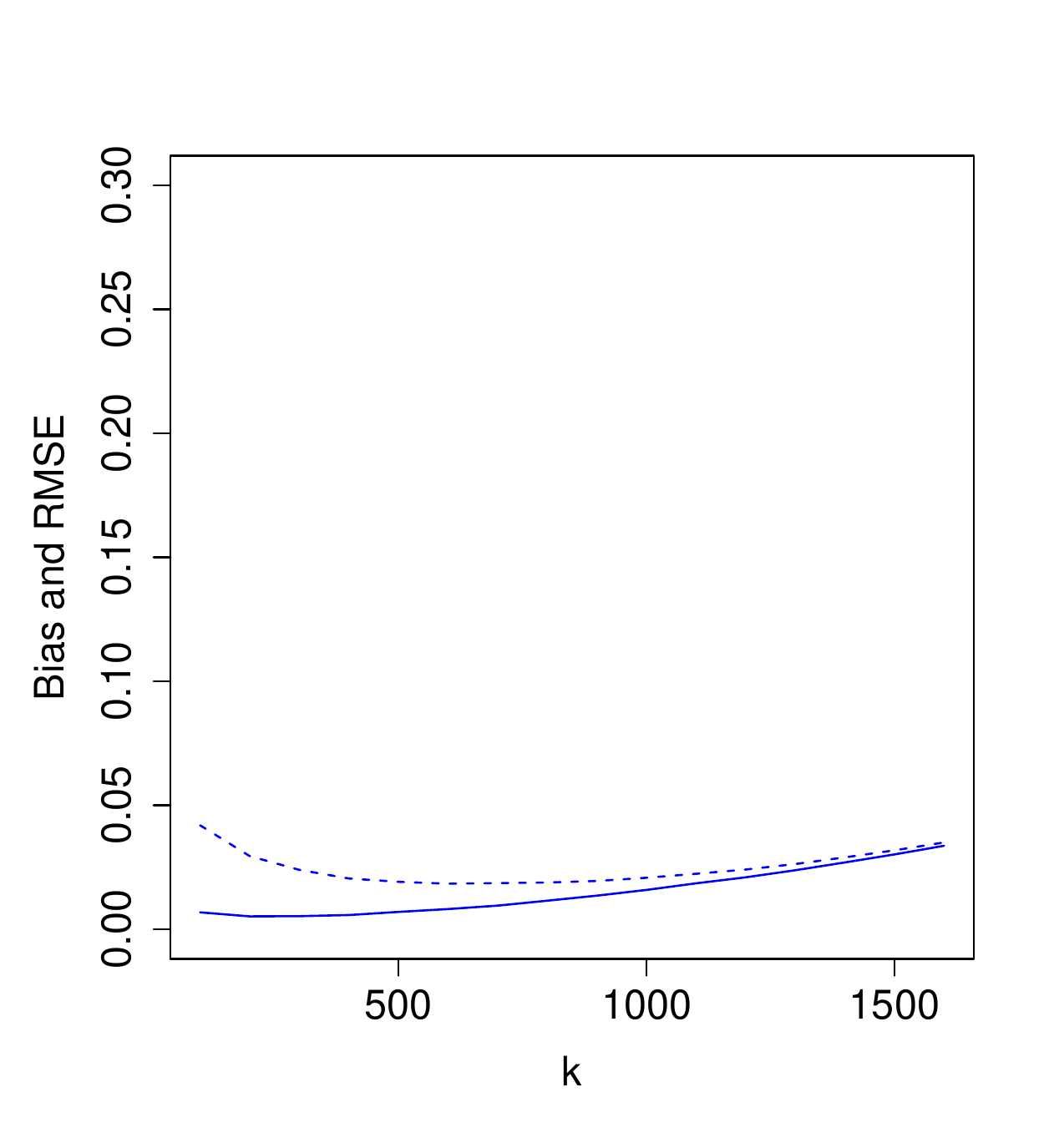}
\includegraphics[scale = 0.35, trim = 0 50 0 60]{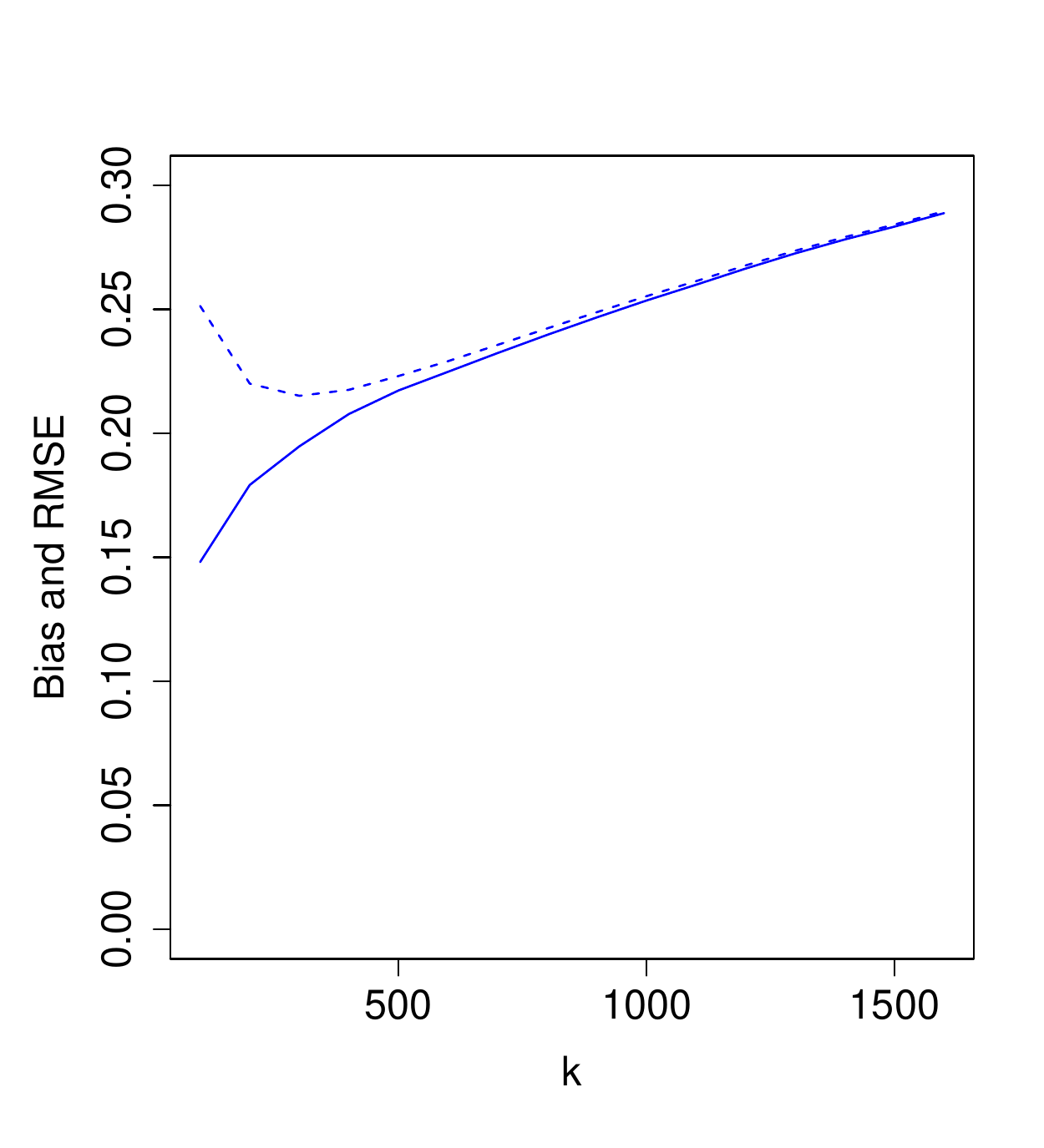}
\includegraphics[scale = 0.35, trim = 0 50 0 60]{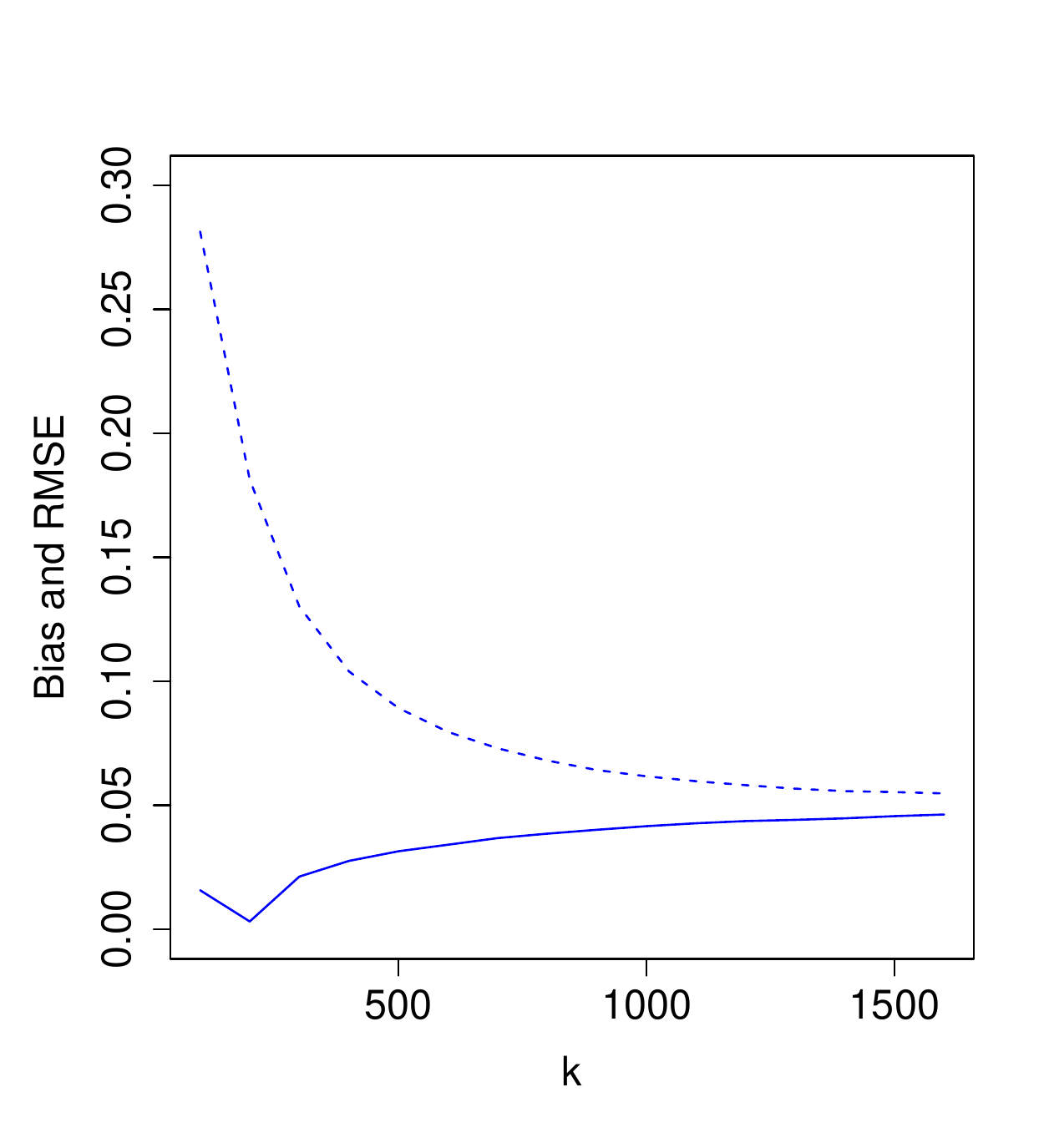}
\caption{Absolute bias (solid lines) and RMSE (dashed lines) of the M-estimator of $\lambda$ as a function of $k$, based on 1\,000 samples of size 5\,000 from model M3 with parameter values 0.4, 1 and 1.6, from left to right.}
\label{Choice of k P}
\end{figure}

An analysis of $\hat \lambda_n$ for a finer range of parameter values is provided in \cref{fig:P-box}. Motivated by \cref{Choice of k P} we fix $k = 400$, which approximately minimizes the maximal RMSE. Overall the estimator is very precise for small values of $\lambda$, but incurs a bias around $\lambda = 0.8$ where it struggles to distinguish between values slightly smaller and slightly larger than 1. This phenomenon is not completely unexpected; a close look at \cref{table} reveals that $c_\lambda$ has almost (but not quite) a symmetry around the point $\lambda=1$, e.g. $c_{0.8}$ is very similar in shape to $c_{1.2}$. This point also corresponds to the transition between asymptotic dependence and independence, which makes estimation challenging.

\begin{figure}[h]
\centering
\includegraphics[scale = 0.5, trim = 0 40 0 60]{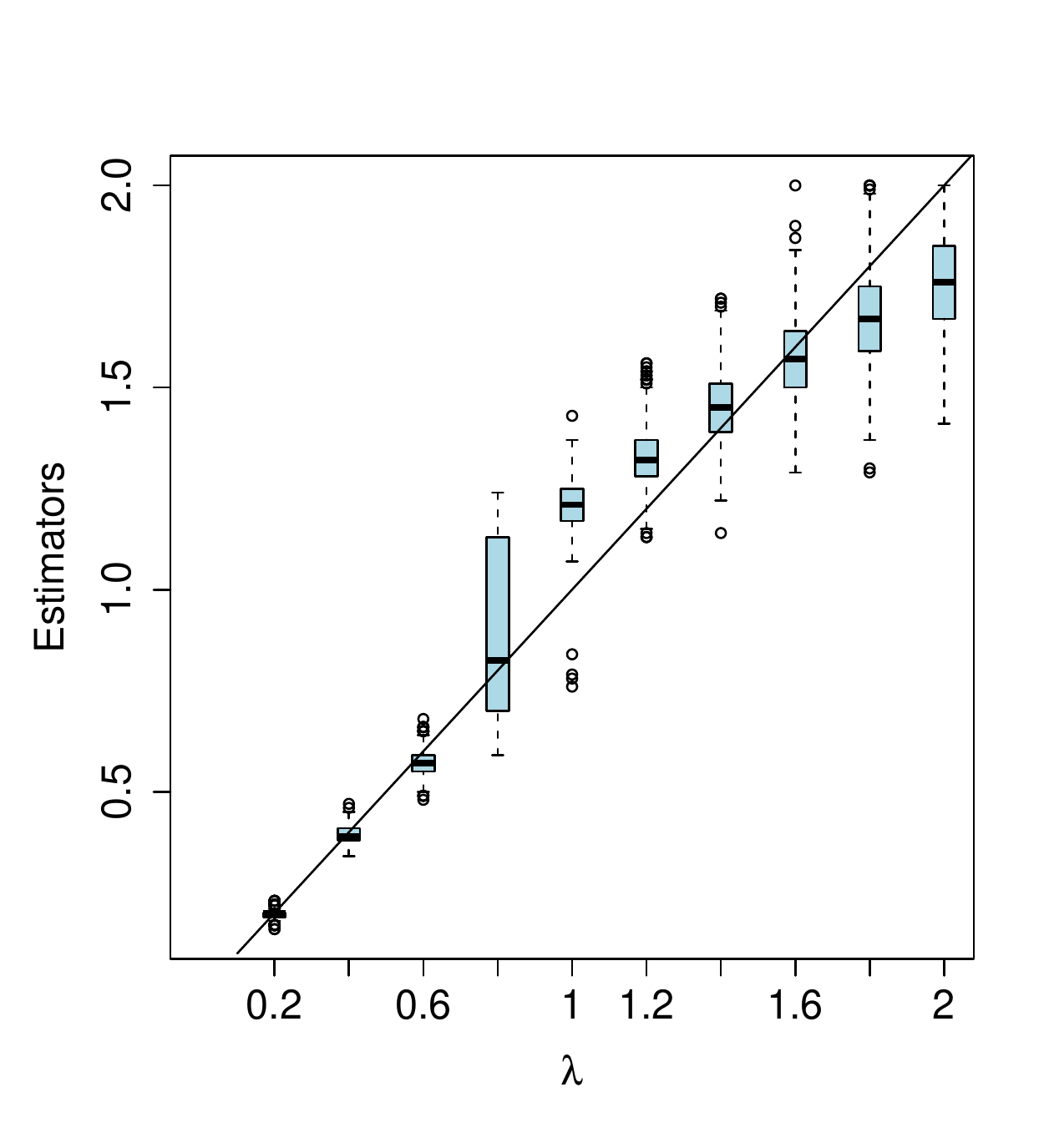}
\caption{Box plots of the M-Estimators of $\lambda$ based on 1\,000 samples of size 5\,000 for each parameter value.}\label{fig:P-box}
\end{figure}

\subsection{Spatial models}
\label{sim spatial}

{In this section we illustrate the performance of the proposed methodology for spatial data. The candidate class for $c_\theta$ results from inverted Brown--Resnick processes with fractal variograms (see \cref{ex:BR_revisited}) and takes the form 
\begin{equation} \label{model_sp}
c^{(s)}_\vartheta(x, y) = (xy)^{\theta^{(s)}}, \quad \theta^{(s)} = \theta(\Delta^{(s)}; \vartheta) := \Phi\Big( \frac{1}{2} (\Delta^{(s)}/\beta)^{\alpha/2} \Big), \quad s \in \cP,
\end{equation}
where $\vartheta = (\alpha, \beta) \in (0, 2] \times \R_+$ and $\Delta^{(s)}$ is the Euclidean distance between the two locations in pair $s$ (measured in units of latitude). Motivated by the data application in the following section, the true parameter values are set as $\vartheta_0 = (1, 3)$ and the values for $\Delta^{(s)}$ are obtained from $40$ randomly sampled pairs of locations in that data set; see \cref{fig:distances} in the supplement for a histogram of the distances in this sample.

To evaluate the performance of our estimators we simulate $1000$ independent data sets, each of size $5000$, of an inverted Brown--Resnick process with unit Fr\'echet margins and fractal variogram from \cref{eq:fracvari} with $\alpha =1, \beta = 3$. Following the bivariate simulations, to each of the 40 components of the data we add an independent random variable with Pareto distribution function $1-1/x^4$, $x\geq 1$. Using the same weight function $g$ as in the bivariate simulations (see \cref{eq:defg}), we compute the two estimators introduced in \cref{eq:defthetahatls,eq:defthetahat}. Since the performance of both estimators turns out to be very similar, we only report results for the least squares estimator from \cref{eq:defthetahatls} here and defer all simulations for the estimator \eqref{eq:defthetahat} to \cref{simulations-add} in the supplement.

Following the discussion in \cref{rem:hatk}, we fix a value $m$ and select each $k^{(s)}$ such that $\hat Q_n^{(s)}(k^{(s)}/n,k^{(s)}/n) = m$. The first two panels of \cref{Choice of k sp 1} show the absolute bias and RMSE of the estimators $\hat\alpha$ and $\hat\beta$, respectively, as functions of $m \in \{75, 100, \dots, 500\}$. We observe that the RMSE for both estimators is relatively large across all values of $m$. Interestingly, this does not result in a bad performance in estimating the function $\theta(\cdot; \vartheta)$. Indeed, the last panel of \cref{Choice of k sp 1} shows averaged (over simulation runs) values for $\sup_{0 \leq \Delta \leq 3} |\theta(\Delta; \hat\vartheta) - \theta(\Delta; \vartheta)|$ and indicates a good overall performance; note that the observed values of $\Delta$ are all smaller than 3 (see \cref{fig:distances} in the supplement). This can be explained by the fact that different values of $(\alpha,\beta)$ can lead to somewhat similar curves in the range of interest. This is further illustrated in the left panel of \cref{fig:sp-multibox-1} where a random sample of 50 estimated functions $\theta(\Delta; \hat\vartheta)$ is displayed.

\begin{figure}[H]
	\centering
	\includegraphics[scale = 0.35, trim = 0 50 0 60]{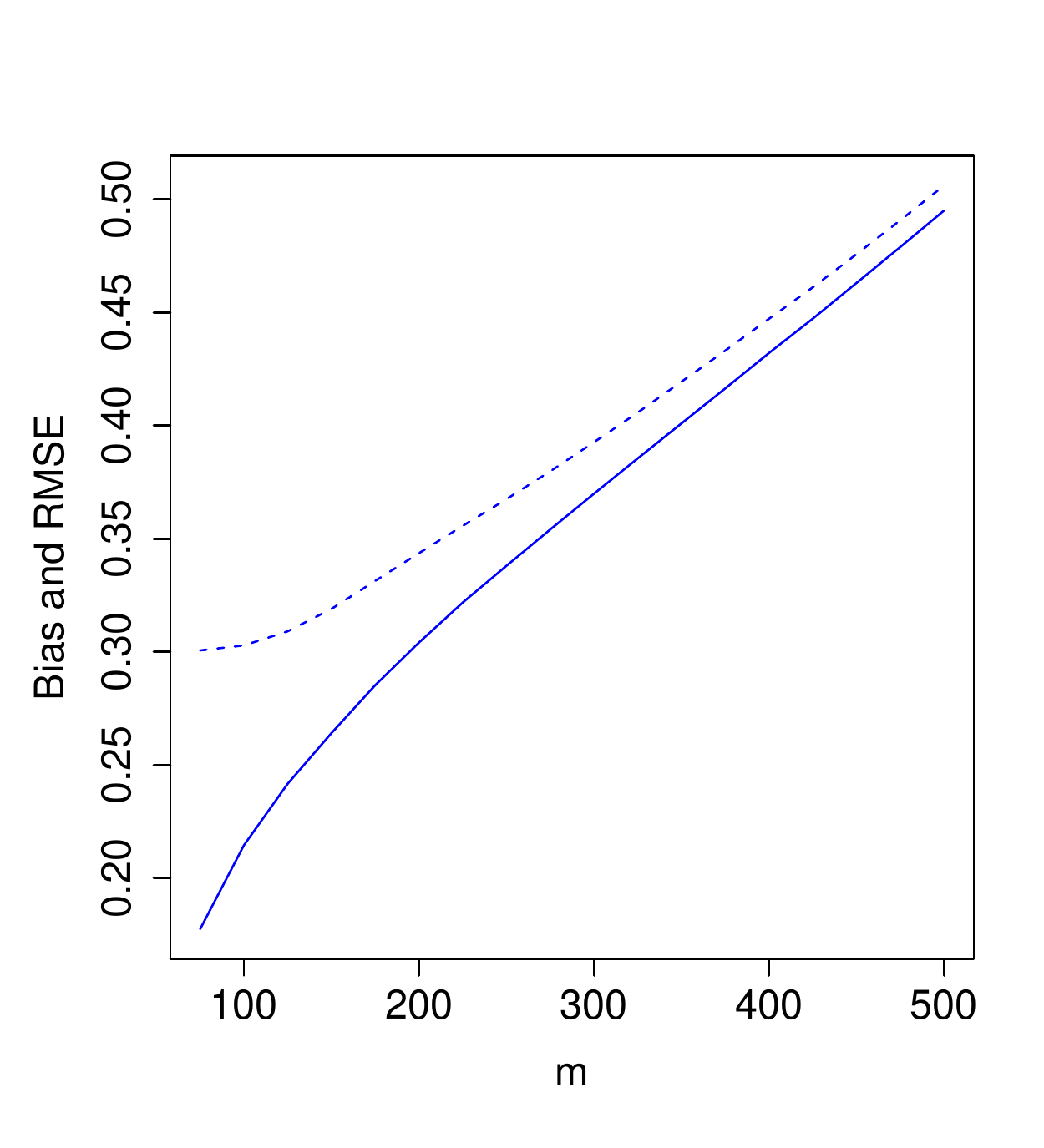}
	\includegraphics[scale = 0.35, trim = 0 50 0 60]{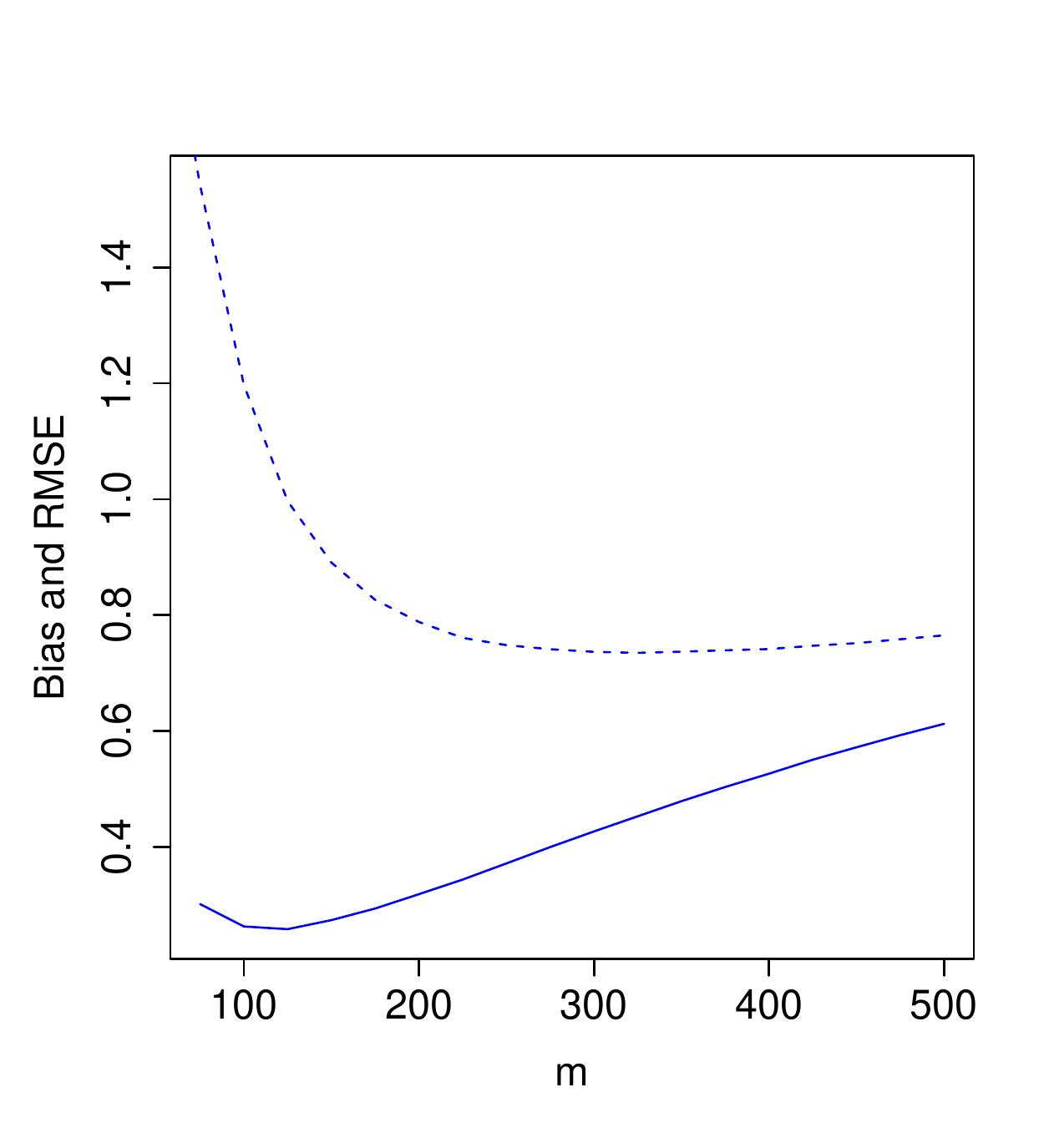}
	\includegraphics[scale = 0.35, trim = 0 50 0 60]{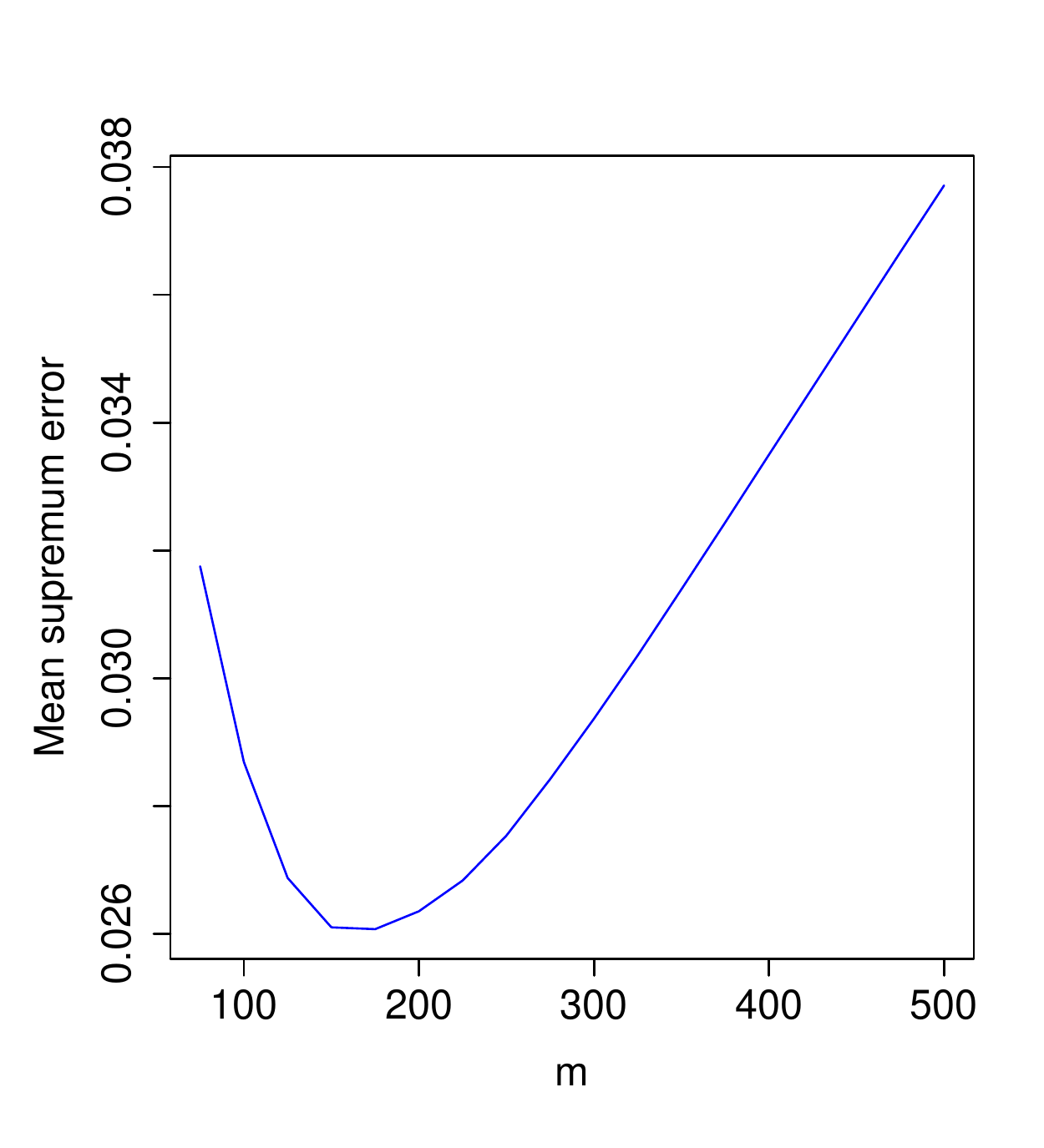}
	\caption{Left and middle columns: Bias (solid line) and RMSE (dotted line) of the estimators of the two spatial parameters $\alpha$ (left) and $\beta$ (middle) as a function of $m$. Right: Mean of the supremum error $\sup_{0 \leq \Delta \leq 3} |\theta(\Delta; \hat\alpha, \hat\beta) - \theta(\Delta; \alpha, \beta)|$ as a function of $m$.}
	\label{Choice of k sp 1}
\end{figure}

We conclude this section by fixing $m=150$ and comparing the performance of estimators for $\theta^{(s)}$ based on a bivariate sample at a given distance and the spatial estimator discussed above. Boxplots corresponding to five pairs of stations with distances $\Delta^{(s)} \approx 0.5, 1, \dots, 2.5$ are shown in the left panel of \cref{fig:sp-multibox-1}. As expected from the theory, using the spatial estimator is advantageous as it allows to combine information from different distances and leads to a reduced variance.
	
\begin{figure}[h]
\centering
\includegraphics[scale = 0.55, trim = 0 35 0 50]{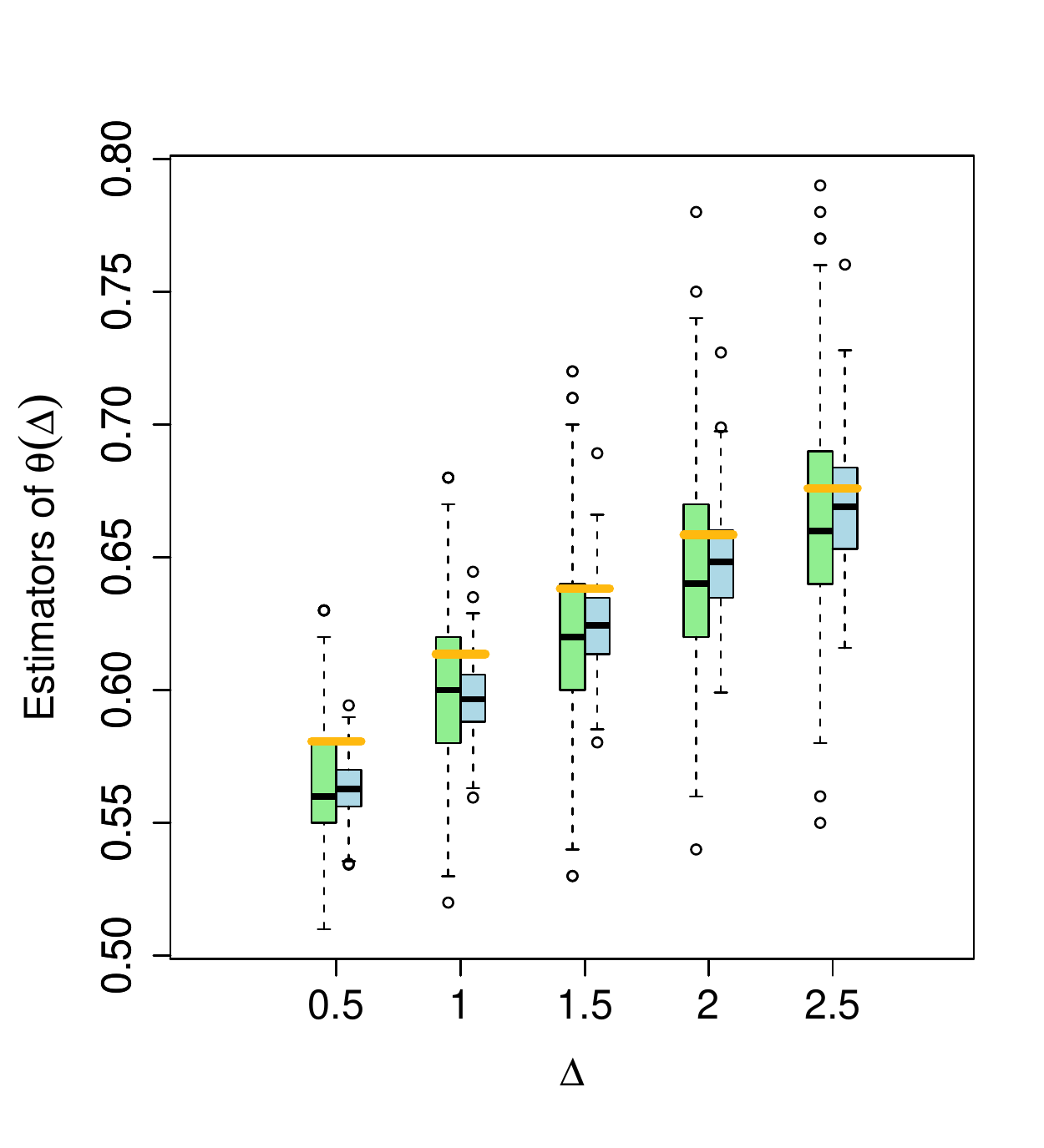}
\includegraphics[scale = 0.55, trim = 0 35 0 50]{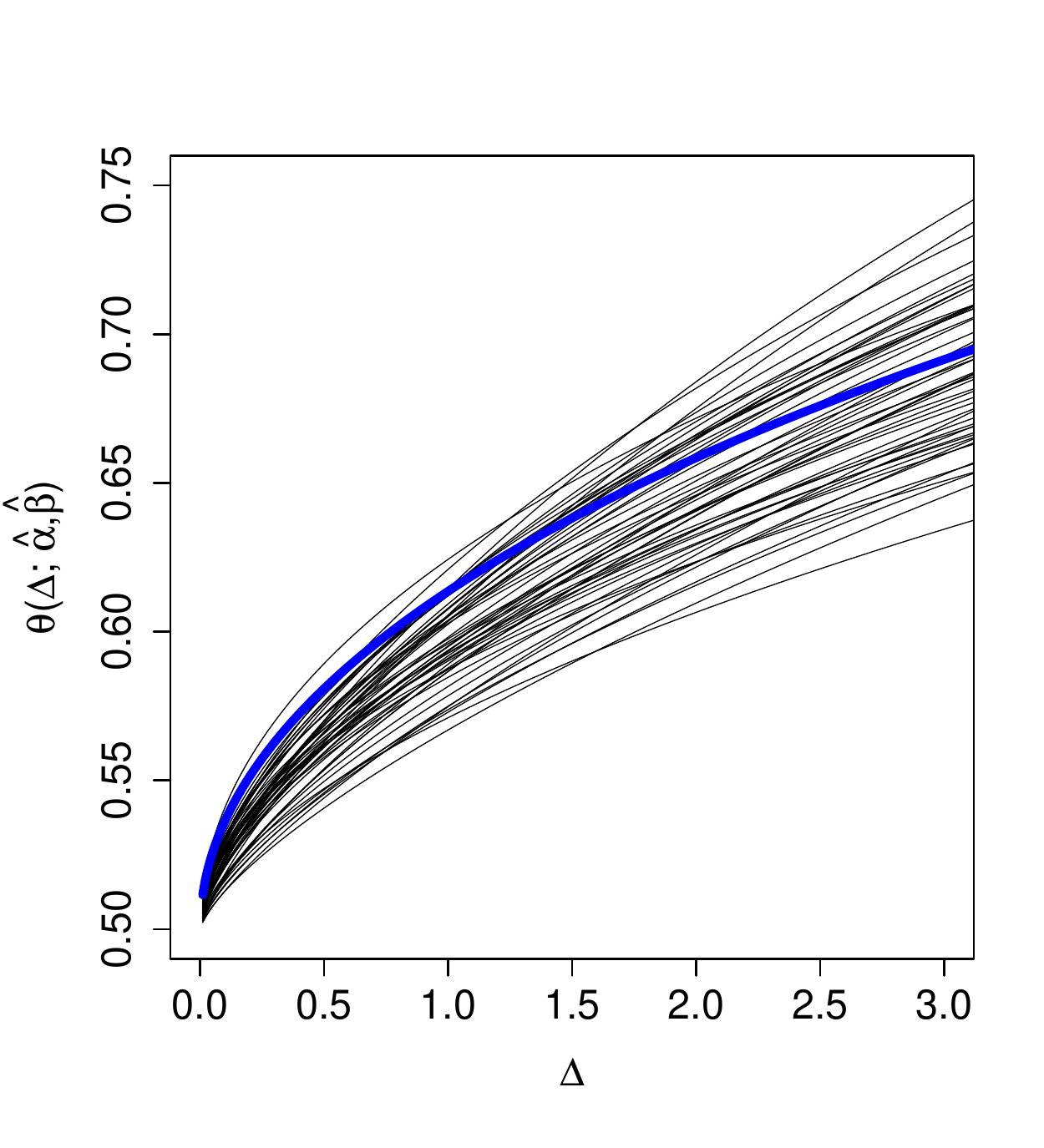}
\caption{Left panel: Estimators of $\theta(\Delta)$ for 5 different distances. For each distance, bivariate M-estimator $\hth_n^{(s)}$ (green) and spatial estimator $\theta(\Delta^{(s)}; \hat\alpha, \hat\beta)$ (blue) based on the $d=40$ locations. Right panel: 50 sampled curves $\theta(\cdot; \hat\alpha, \hat\beta)$. Blue represents the true curve $\theta(\cdot; \alpha, \beta)$.}
\label{fig:sp-multibox-1}
\end{figure}

\section{Application to rainfall data} \label{application}

In a data set introduced in \cite{LDELW2018}, rainfall was measured daily from 1960 to 2009 at a set of 92 different locations in the state of Victoria, southeastern Australia, for a total of $n = 18\,263$ measurements. The conclusions in that paper are that an asymptotically independent model is suitable. A subset of 40 locations, for a total of 780 pairs, was randomly sampled; see the right panel of \cref{fig:data}. To the data at those selected locations we fit the same tail model as in \cref{sim spatial}, given in \cref{model_sp}. The weight function $g$ that we use is the same as before and as in \cref{sim spatial}, we make use of \cref{rem:hatk} by fixing a value $m$ and choosing each $k^{(s)}$ accordingly.

We set $m=400$. The left panel of \cref{fig:data} shows the 780 pairwise estimators $\hth_n^{(s)}$ plotted against the distances $\Delta^{(s)}$. Despite some estimates at the boundary of the parameter space, the results do not provide much evidence for asymptotic dependence, whereas all estimates are away from the boundary for distances of at least 0.3 units of latitude, strongly suggesting asymptotic independence at these distances. Our two estimators \eqref{eq:defthetahatls} and \eqref{eq:defthetahat} of $(\alpha, \beta)$ yield estimates $(\hat\alpha, \hat\beta)$ of (1.55, 2.24) and (1.56, 2.24), respectively. They are extremely similar, as hinted by the simulation study from \cref{sim spatial}. The curve $\theta(\cdot; \hat\alpha, \hat\beta)$ corresponding to the least squares estimator is also shown in the left panel of \cref{fig:data}. The middle panel of \cref{fig:data} displays similar curves for the least squares estimator when $m$ varies from 200 to 1\,000. It shows that the estimated curve is robust with respect to the choice of $m$.

\begin{figure}[ht]
\centering
\includegraphics[scale = 0.35, trim = 0 50 0 60]{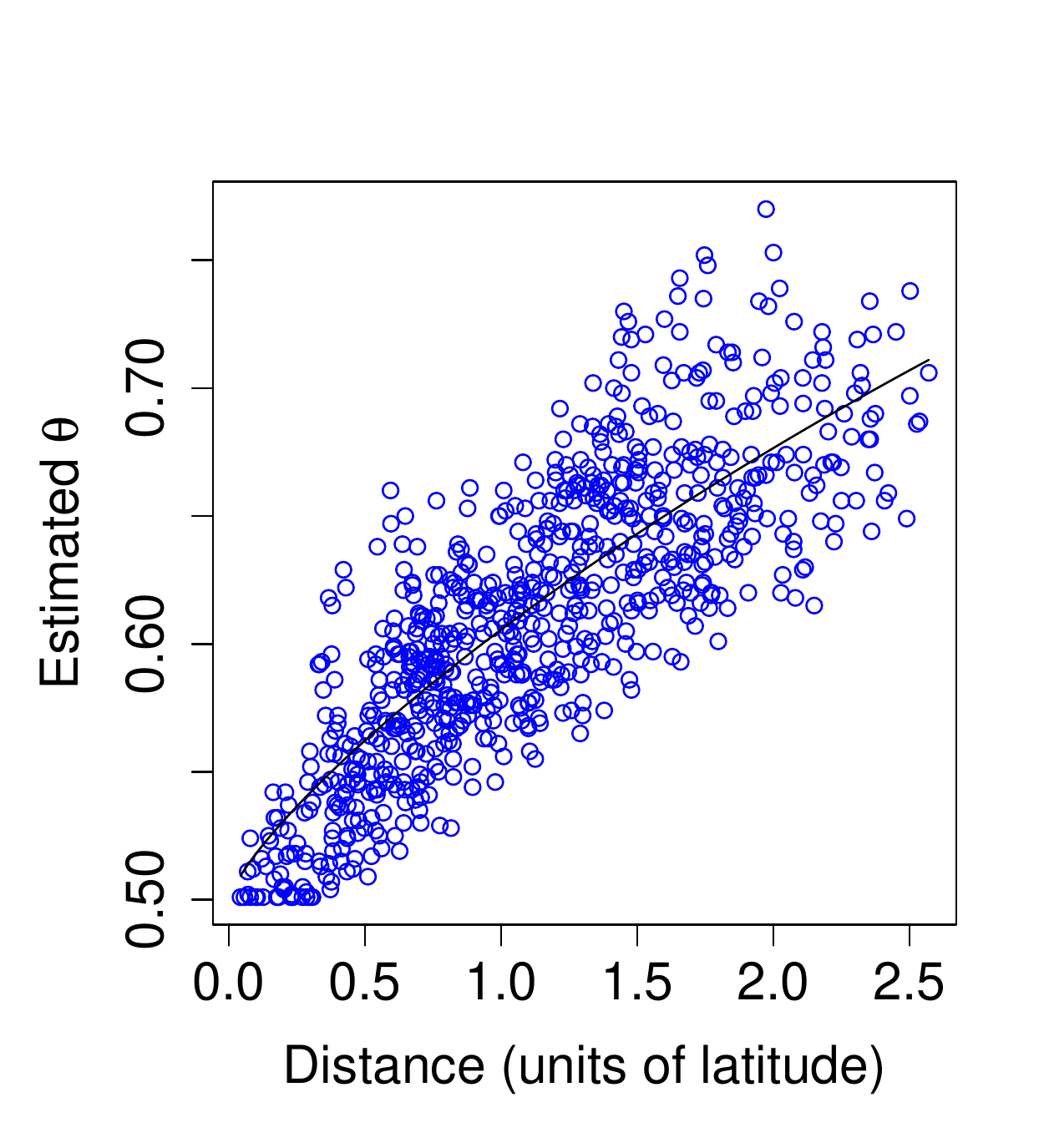}
\includegraphics[scale = 0.35, trim = 0 50 0 60]{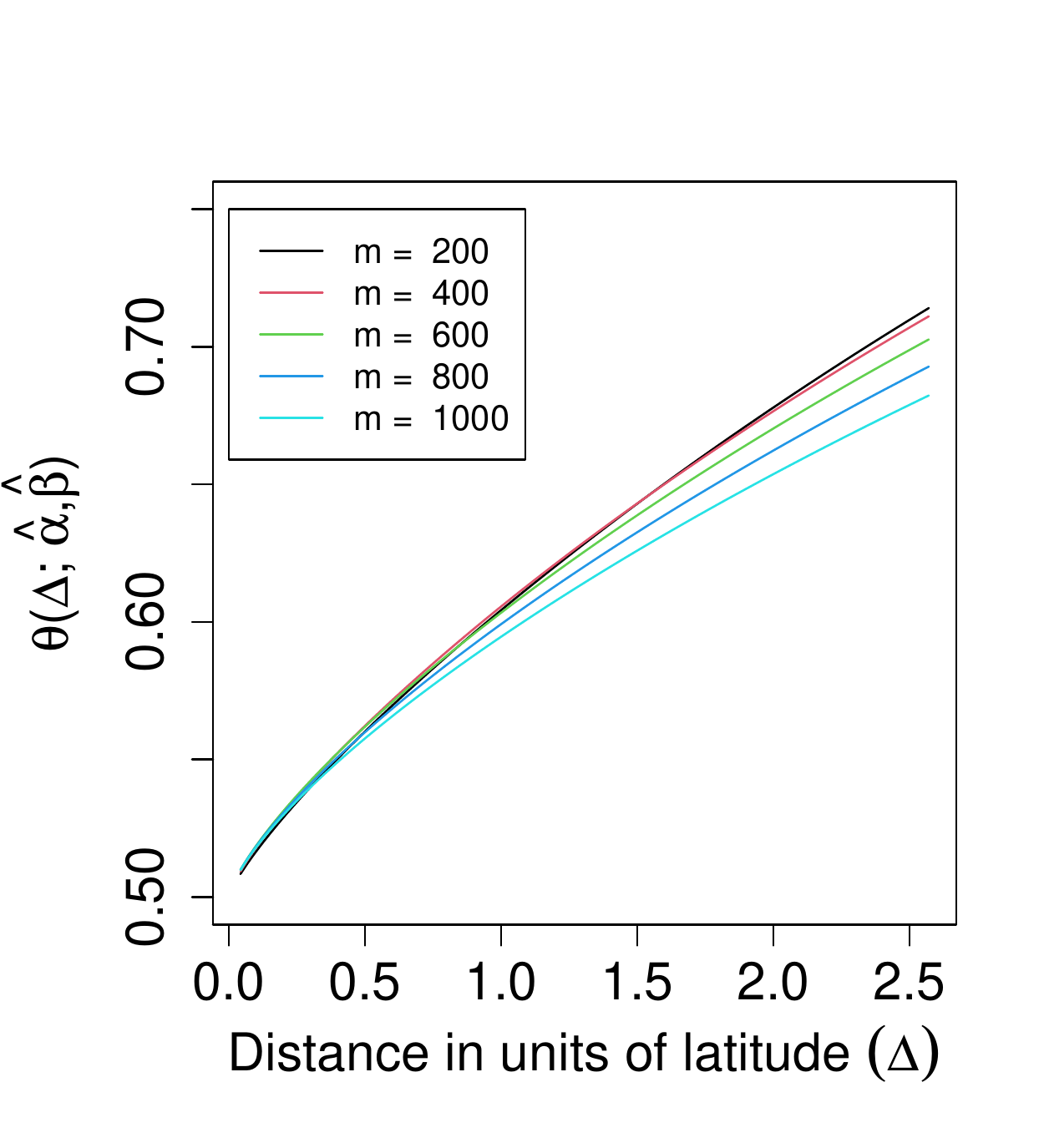}
\includegraphics[scale = 0.35, trim = 0 50 0 60]{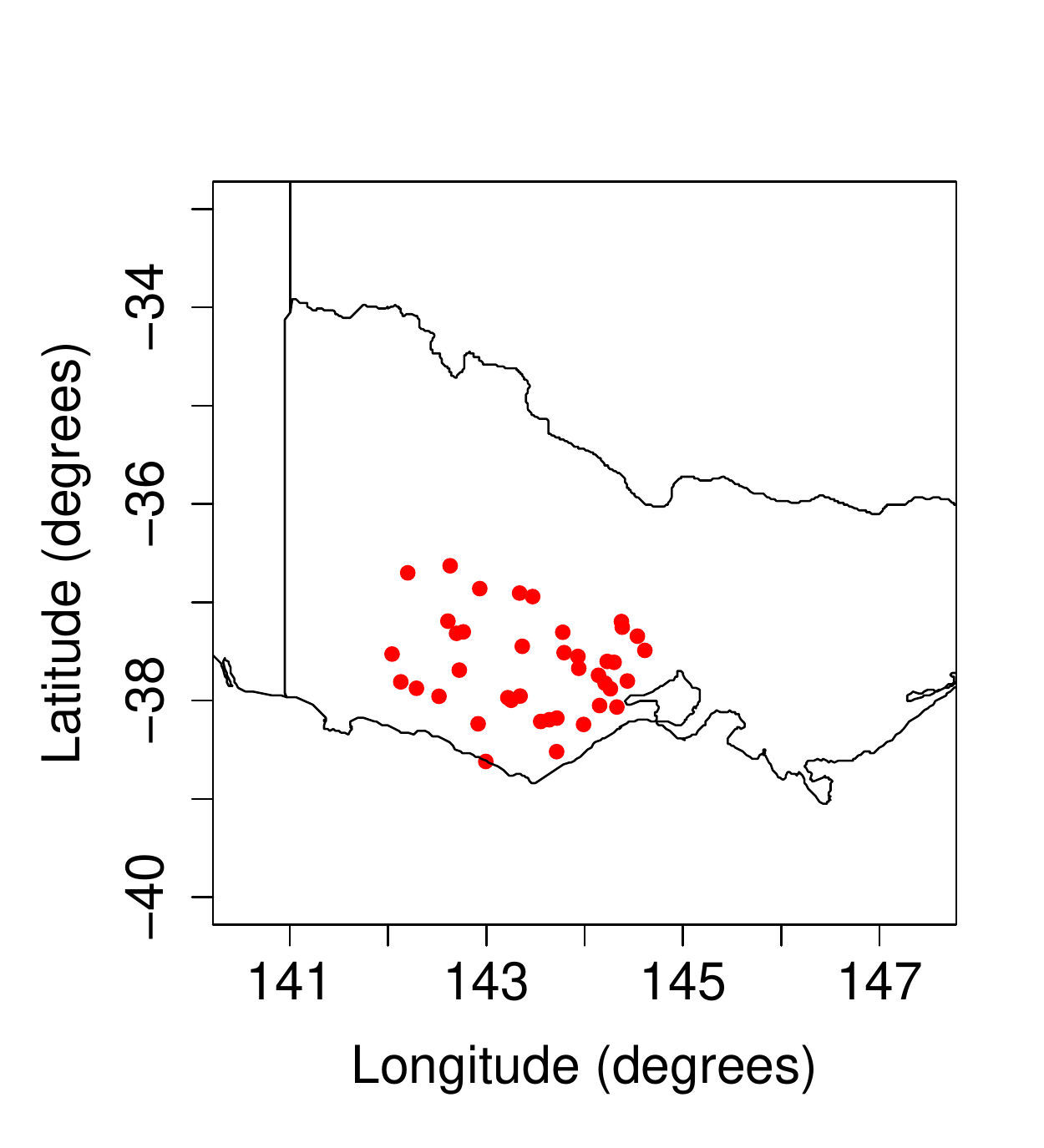}
\caption{Left: Estimated parameters $\hth_n^{(s)}$ against the distances $\Delta^{(s)}$. The black line represents the estimated curve $\theta(\cdot; 1.55, 2.24)$. Middle: Estimated curve $\theta(\cdot; \hat\alpha, \hat\beta)$ for the least squares estimator with different values of $m$. Right: The 40 sampled locations in the state of Victoria, southeastern Australia.}
\label{fig:data}
\end{figure}

\section*{Acknowledgments}
We are grateful to two reviewers for their valuable input. Their constructive comments and suggestions resulted in a substantial improvement of this paper. Micha\"el Lalancette was supported by the Fonds de recherche du Qu\'ebec -- Nature et technologies and by an Ontario Graduate Scholarship. Sebastian Engelke was supported by the Swiss National Science Foundation and the Fields Institute for Research in Mathematical Sciences. Stanislav Volgushev was supported in part by a discovery grant from NSERC of Canada and a Connaught new researcher award.


\setcounter{section}{0}
\renewcommand*{\theHsection}{chX.\the\value{section}}
\setcounter{lemm}{0}
\renewcommand*{\theHlemm}{chX.\the\value{lemm}}
\setcounter{rem}{0}
\renewcommand*{\theHrem}{chX.\the\value{rem}}
\setcounter{coro}{0}
\renewcommand*{\theHcoro}{chX.\the\value{coro}}
\setcounter{figure}{0}
\renewcommand*{\theHfigure}{chX.\the\value{figure}}
\setcounter{table}{0}
\renewcommand*{\theHtable}{chX.\the\value{table}}

\renewcommand{\thesection}{S\arabic{section}}
\renewcommand{\thelemm}{S\arabic{lemm}}
\renewcommand{\therem}{S\arabic{rem}}
\renewcommand{\thecoro}{S\arabic{coro}}
\renewcommand{\thefigure}{S\arabic{figure}}
\renewcommand{\thetable}{S\arabic{table}}

\begin{center}
SUPPLEMENTARY MATERIAL
\end{center}

This Supplementary Material is divided in six sections. \Cref{proofs} contains the proofs of all main results, with a number of necessary technical results deferred to \cref{tech}. \Cref{proof IMS,proof rsc} present proofs of several claims from different Examples. A brief discussion of computational complexity in spatial estimation is given in \cref{sec:cc} and additional simulation results appear in \cref{simulations-add}.

		\section{Proofs of main results} \label{proofs}

		In this section are collected the proofs of \cref{big lemma AI,big lemma AD,big thm,thm spatial np,thm spatial}. A number of more technical results, which are instrumental in the following, are collected in \cref{tech}.

		\subsection{Bivariate estimation}
		\label{proofs bivariate}

		For the proofs concerning the bivariate estimators, we assume the framework of \cref{estimation-nonpar,estimation-par}, we define the transformed random variables $U = 1 - F_1(X)$, $V = 1 - F_2(Y)$ and note that $Q$ is the distribution function of the random vector $(U, V)$. Define the transformed observations $U_i = 1 - F_1(X_i)$, $V_i = 1 - F_2(Y_i)$ and denote by $U_{n, 1}, \dots, U_{n, n}$ and $V_{n, 1}, \dots, V_{n, n}$ the ordered versions thereof.  {Additionally define $U_{n,0} = V_{n,0} = 0$.} For an intermediate sequence $k$, define the random functions $u_n$ and $v_n$ by
		\[
		u_n(x) = \frac{n}{k} U_{n, \lfloor kx \rfloor} \quad \text{and} \quad v_n(y) = \frac{n}{k} V_{n, \lfloor ky \rfloor},
		\]
		for $(x, y) \in [0, T]^2$. Recalling that $m = nq(k/n)$, it allows us to write
		\[
		\hc_n(x, y) = \frac{n}{m} Q_n\left( \frac{k}{n} u_n(x), \frac{k}{n} v_n(y) \right)
		\]
		 {where 
		\[
		Q_n(x,y) := \frac{1}{n} \sum_{i=1}^n \Ind{U_i \leq x, V_i \leq y} 
		\]
		denotes the empirical distribution function of $(U_1,V_1),\ldots,(U_n,V_n)$. We begin by discussing technical results that will be used in the proof of both \cref{big lemma AI} and \cref{big lemma AD}. Consider the decomposition
		\begin{align*}
		W_n(x, y) =& \sqrt{m} \Big( \frac{n}{m} Q_n\Big( \frac{k}{n} u_n(x), \frac{k}{n} v_n(y) \Big) - \frac{n}{m}Q\Big(\frac{k}{n} u_n(x),  \frac{k}{n} v_n(y) \Big) \Big)
		\\
		& + \sqrt{m}\Big( \frac{n}{m}Q\Big(\frac{k}{n} u_n(x),  \frac{k}{n} v_n(y) \Big) - c(u_n(x), v_n(y)) \Big)
		\\
		& + \sqrt{m} \big( c(u_n(x), v_n(y)) - c(x, y) \big).
		\end{align*}
		For the second term in the above decomposition, note that 
		\begin{align*}
		\sqrt{m} \Big( \frac{n}{m} Q\Big(\frac{k}{n} x,  \frac{k}{n} y \Big) - c(x,y) \Big)
		= O\left( \sqrt{m} q_1\left( \frac{k}{n} \right) \right) = o(1)
		\end{align*}
		uniformly over all $(x, y) \in [0, 2T]^2$; here the last equation follows from \cref{con-proc}(ii). By \cref{coro app} we have $\PP(u_n(T) \vee v_n(T) \leq 2T ) \to 1$, and thus 
		\[
		\sup_{x,y \in [0,T]} \sqrt{m}\Big| \frac{n}{m}Q\Big(\frac{k}{n} u_n(x),  \frac{k}{n} v_n(y) \Big) - c(u_n(x), v_n(y)) \Big| = o_P(1).
		\]
		Next define for all $x,y \in [0,2T]$
		\begin{equation}\label{eq:defHn}
		H_n(x,y) := \sqrt{m} \Big( \frac{n}{m} Q_n\Big( \frac{k}{n} x, \frac{k}{n} y \Big) - \frac{n}{m} Q\Big( \frac{k}{n} x, \frac{k}{n} y \Big) \Big).
		\end{equation}
		By \cref{limit Gn} this process converges, in $\ell^\infty([0, 2T]^2)$, to the process $W$ from \cref{big lemma AI} and by \cref{coro app} $u_n$ and $v_n$ converge uniformly in probability to the identity function $I: [0,2T] \to [0,2T]$. Therefore, the triple $(H_n, u_n, v_n)$ converges jointly in distribution to $(W, I, I)$. This implies
		\begin{equation}\label{eq:Hnun-Hn}
		\sup_{x,y \in [0,T]} \Big|H_n(u_n(x),v_n(y)) - H_n(x,y) \Big| = o_P(1).
		\end{equation}
		Indeed, consider the map
		\[
		f: \left\{
		\begin{array}{c}
		\ell^{\infty}([0,2T]^2)\times\cV[0,T]\times\cV[0,T] \to \R
		\\
		(a,b_1,b_2) \mapsto \sup_{x,y \in [0,T]} |a(b_1(x),b_2(y)) - a(x,y)|
		\end{array} 
		\right.
		\] 
		where $\cV[0,T] := \{g \in\ell^{\infty}[0,T]: g([0,T]) \subset [0,2T]\}$ and assume that the product space is equipped with the norm $\|a\|_\infty + \|b_1\|_\infty + \|b_2\|_\infty$. Observe that $f$ is continuous at points $(a,b_1,b_2)$ where $a$ is a continuous function and that the sample paths of $W$ are almost surely continuous. Thus, by the continuous mapping theorem, with probability converging to~1,
		\[
		\sup_{x,y \in [0,T]} \Big|H_n(u_n(x),v_n(y)) - H_n(x,y) \Big| = f(H_n,u_n,v_n) \wc f(W,I,I) = 0.
		\]
		Since the limit is constant a.s. \cref{eq:Hnun-Hn} follows. Combining the equations above, we find
		\begin{equation} \label{eq:reprWn}
		W_n(x, y) = H_n(x, y) + \sqrt{m} \big( c(u_n(x), v_n(y)) - c(x, y) \big) + o_P(1),
		\end{equation}
		where the term $o_P(1)$ is uniform on $[0,T]^2$, and we recall that $H_n \wc W$ in $\ell^\infty([0, 2T]^2)$.
		}

		\subsubsection{Proof of \cref{big lemma AI}}

		Define 
		\[
		S_n(x, y) := \sqrt{m} \big( c(u_n(x), v_n(y)) + c(x, y) \big).
		\]
		In light of \cref{eq:reprWn} it suffices to prove that $S_n \stackrel{P}{\to} 0$ uniformly on $[0,T]^2$. From here on it is more convenient to study component-wise increments. That is, we write
		\begin{align*}
		S_n(x, y) &= \sqrt{m} ( c(u_n(x), y) - c(x, y) ) + \sqrt{m} ( c(u_n(x), v_n(y)) - c(u_n(x), y) )
		\\
		&=: S_n^{(a)}(x, y) + S_n^{(b)}(x, y)
		\end{align*}
		and we will show that both $S_n^{(a)}$ and $S_n^{(b)}$ converge to 0 in probability, starting with $S_n^{(a)}$.

		By assumption, since with probability converging to 1 we have $u_n(x) \in [0, 2T]$ for every $x \leq T$, we can write
		\begin{align}
		S_n^{(a)}(x, y) &= \sqrt{m} \left( c(u_n(x), y) + c(x, y) \right) \label{c rep 1} \\
		&= \sqrt{m} \left\{ \frac{n}{m} Q \left( \frac{k}{n} u_n(x), \frac{k}{n} y \right) - \frac{n}{m} Q \left( \frac{k}{n} x, \frac{k}{n} y \right) + \Op{ q_1 \left( \frac{k}{n} \right) } \right\} \notag \\
		&= \frac{n}{\sqrt{m}} \left( Q \left( \frac{k}{n} u_n(x), \frac{k}{n} y \right) - Q \left( \frac{k}{n} x, \frac{k}{n} y \right) \right) + \op{1} \label{c rep 2}
		\end{align}
		uniformly on $[0, T]^2$, since the sequence $m$ was chosen so that $\sqrt{m} q_1(k/n) \to 0$. We will use both \cref{c rep 1,c rep 2} as representations of $S_n^{(a)}$ throughout the proof.

		Let $\beta_n = (m/k)/(\log(k/m))$. From there, partition $[0, T]^2$ in $\Theta_n^{(1)} = [0, 1/k) \times [0, T]$, $\Theta_n^{(2)} = [1/k, \beta_n) \times [0, T]$ and $\Theta_n^{(3)} = [\beta_n, T] \times [0, T]$ (if $\beta_n < 1/k$, $\Theta_n^{(2)}$ is empty). These sets represent the ``small", ``intermediate" and ``large" values of $x$, respectively. We will prove that the suprema of $S_n^{(a)}$ on $\Theta_n^{(1)}$, $\Theta_n^{(2)}$ and $\Theta_n^{(3)}$ all converge to 0 in probability. \Cref{c rep 2} yields
		\begin{align*}
		\sup_{(x, y) \in \Theta_n^{(1)}} |S_n^{(a)}(x, y)| &= \frac{n}{\sqrt{m}} \sup_{0 \leq x < 1/k} \left| Q \left( \frac{k}{n} u_n(x), \frac{k}{n} y \right) - Q \left( \frac{k}{n} x, \frac{k}{n} y \right) \right| + o_P(1) \\
		&= \frac{n}{\sqrt{m}} \sup_{0 \leq x < 1/k} Q \left( \frac{k}{n} x, \frac{k}{n} y \right) + o_P(1) \\
		&\leq \frac{n}{\sqrt{m}} \frac{1}{n} + o_P(1) \\
		&= \frac{1}{\sqrt{m}} + o_P(1),
		\end{align*}
		where we have once again used the facts that $u_n(x) = 0$ whenever $x < 1/k$ and that $Q(0, \cdot) = Q(\cdot, 0) = 0$, in addition to the fact that $Q(u, v) \leq u$. This proves that $\sup_{\Theta_n^{(1)}} |S_n^{(a)}| \to 0$ in probability.

		Using \cref{c rep 2} again, the supremum of $S_n^{(a)}$ on $\Theta_n^{(2)}$ can be expressed as
		\begin{align*}
		\sup_{1/k \leq x < \beta_n} \left| S_n^{(a)}(x, y) \right| &= \sup_{1/k \leq x < \beta_n} \frac{n}{\sqrt{m}} \left| Q \left( \frac{k}{n} u_n(x), \frac{k}{n} y \right) - Q \left( \frac{k}{n} x, \frac{k}{n} y \right) \right| + \op{1} \\
		&\leq \sup_{1/k \leq x < \beta_n} \frac{n}{\sqrt{m}} \left| \frac{k}{n} u_n(x) - \frac{k}{n} x \right| + \op{1} \\
		&= \sup_{1/k \leq x < \beta_n} \frac{k}{\sqrt{m}} |u_n(x) - x| + \op{1} \\
		&= \Op{ \sup_{1/k \leq x < \beta_n} \sqrt{\frac{k}{m}} \varphi(x) } + \op{1},
		\end{align*}
		where we have used Lipschitz continuity of $Q$ and \cref{Csorgo}. The last bound holds for any function $\varphi$ that satisfies the conditions in \cref{Csorgo}, but from now on we use $\varphi(x) := \sqrt{x \log \log(1/x)}$ on $(0, B]$ and $\varphi(x) := \sqrt{x}$ on $(B, T]$, where $B>0$ is chosen small enough so that $\varphi$ is well defined and non-decreasing. By monotonicity, the supremum is attained at $x = \beta_n$. We then have
		$$
		\sup_{1/k \leq x < \beta_n} \left| S_n^{(a)}(x, y) \right| = \Op{\sqrt{ \frac{k}{m} \beta_n \log \log(1/\beta_n) }} + \op{1}
		$$
		because since $\beta_n \to 0$, eventually $\beta_n \leq B$, so eventually $\varphi(\beta_n) = \sqrt{\beta_n \log\log(1/\beta_n)}$. The last display converges in probability to 0 since
		$$
		\frac{k}{m} \beta_n \log \log(1/\beta_n) = \frac{\log\log \left( \frac{k}{m} \log(k/m) \right)}{\log(k/m)} \too 0
		$$
		as $k/m \to \infty$, which proves that $\sup_{\Theta_n^{(2)}} |S_n^{(a)}| \to 0$ in probability.

		Finally, when considering large values of $x$, \cref{Csorgo} and a combination of \cref{bound_c,bound increments c} imply that
		\begin{align}
		\sup_{\beta_n \leq x \leq T} \left| S_n^{(a)}(x, y) \right| &= \sup_{\beta_n \leq x \leq T} \sqrt{m} |c(u_n(x), y) - c(x, y)| \notag \\
		&\lesssim \sqrt{m} \sup_{\beta_n \leq x \leq T} |u_n(x) - x| r(x \vee u_n(x)) \notag \\
		&= \Op{ \sqrt{\frac{m}{k}} \sup_{\beta_n \leq x \leq T} \varphi(x) r(x \vee u_n(x)) }, \notag
		\end{align}
		where $r(x) = (x \log(1/x))^{-1}$. By monotonicity of $\varphi$, the inside of the $O_P$ can be upper bounded by
		$$
		\sqrt{\frac{m}{k}} \sup_{\beta_n \leq x \leq T} \varphi(x \vee u_n(x)) r(x \vee u_n(x))
		$$
		and since with probability converging to 1, for every $x \leq T$, $u_n(x) \leq 2T$, this can in turn be upper bounded (with probability converging to 1) by
		$$
		\sqrt{\frac{m}{k}} \sup_{\beta_n \leq x \leq 2T} \varphi(x) r(x).
		$$

		It can easily be checked (e.g. by differentiation) that the function $\varphi \times r$ is decreasing. Thus, the above supremum is attained at $\beta_n$. Finally, elementary computations yield
		$$
		\sqrt{\frac{m}{k}} \varphi(\beta_n) r(\beta_n) \lesssim \sqrt{ \frac{\log\log((k/m)^2)}{\log(k/m)} } \too 0.
		$$

		Overall, we have shown that $S_n^{(a)} \stackrel{P}{\to} 0$ uniformly over $[0, T]^2$. Note that all the bounds we derived are uniform over all values of $y \in [0, T]$, although it was removed from the notation for parsimony. In order to deal with $S_n^{(b)}$, we recall once again that with probability converging to 1, we have $u_n(x) \leq 2T$ for every $x \leq T$. Therefore, with probability converging to 1,
		\begin{align*}
		\sup_{(x, y) \in [0, T]^2} \left| S_n^{(b)}(x, y) \right| &= \sup_{(x, y) \in [0, T]^2} \sqrt{m} | c(u_n(x), v_n(y)) - c(u_n(x), y) | \\
		&\leq \sup_{x \in [0, 2T], y \in [0, T]} \sqrt{m} | c(x, v_n(y)) - c(x, y) |.
		\end{align*}

		This can be shown to converge in probability to 0 using the exact same proof as for $S_n^{(a)}$. We finally conclude that $S_n \stackrel{P}{\to} 0$ in $\ell^\infty([0, T]^2)$, and the proof for deterministic $k=k_n$ is complete.  {It remains to show that the result continues to hold if we replace the deterministic sequence $k = k_n$ by data-dependent $\hat k$ as outlined in \cref{rem:hatk}. This is established in \cref{sec:proofhatk}.}
		\hfill $\square$

		\subsubsection{Proof of \cref{big lemma AD}}

		In view of \cref{eq:reprWn}, we require the joint asymptotic behavior of $H_n$, $u_n$ and $v_n$. Define, for $(x, y) \in [0, \infty)^2$,
		\[
		L_n^{(1)}(x) = \frac{1}{k} \sum_{i=1}^n \Ind{U_i \leq \frac{k}{n} x} \quad \text{and} \quad L_n^{(2)}(y) = \frac{1}{k} \sum_{i=1}^n \Ind{V_i \leq \frac{k}{n} y},
		\]
		a rescaled version of the marginal empirical distribution functions of $U$ and $V$. We now show that the $\DD$-valued process
		\begin{equation} \label{eq:preGn}
		(x, y) \mapsto \left( H_n(x, y), \sqrt{m} \left( L_n^{(1)}(x) - x \right), \sqrt{m} \left( L_n^{(2)}(y) - y \right) \right)
		\end{equation}
		converges in distribution to the Gaussian process $(W, W^{(1)}, W^{(2)})$ defined in \cref{sec:asy-non} with covariance matrix $\Lambda$ from \cref{cov matrix}, where $\DD := \big( \ell^\infty([0, 2T]^2) \big)^3$.

		Again, let $I$ denote the identity map on $\R$. The three processes $H_n$, $\sqrt{m} (L_n^{(1)} - I)$ and $\sqrt{m} (L_n^{(2)} - I)$ are individually tight (see \cref{limit Gn}) and hence it suffices to prove convergence of the marginal distributions. This in turn follows from convergence of the covariance function, by the multivariate Lindeberg-Feller theorem (see \cite{V2000}, Theorem 2.27); verification of the Lindeberg condition is similar to condition (B) in the proof of \cref{limit Gn}. The convergence of $\E{}{H_n(x, y) H_n(x', y')}$ to $c(x \wedge x', y \wedge y')$ is already shown in \cref{limit Gn}. Using similar arguments and recalling that $m/k \to \chi > 0$, one easily deals with the other covariance terms and concludes that the processes in \cref{eq:preGn} weakly converge to $(W, W^{(1)}, W^{(2)})$ in $\DD$.

		Note that the random functions $u_n$ and $v_n$ are the generalized inverses of $L_n^{(1)} + 1/k$ and $L_n^{(2)} + 1/k$, respectively. Because $\sqrt{m}/k \to 0$, the term $1/k$ is negligible. Upon applying Vervaat's lemma (\cite{V1972}), which states that the generalized inverse mapping is Hadamard differentiable around the identity function, we deduce that the processes $G_n$, defined by
		\[
		G_n(x, y) = (H_n(x, y), \sqrt{m} (u_n(x) - x), \sqrt{m} (v_n(y) - y)),
		\]
		weakly converge to $(W, -W^{(1)}, -W^{(2)})$ in $\DD$. For $t>0$, define the sets
		\begin{equation} \label{eq:defcV(t)}
		\cV(t) := \{b \in \ell^\infty([0, 2T]) : \forall x \in [0, T], x + tb(x) \in [0, 2T]\}.
		\end{equation}
		Let $\DD_n \subset \DD$ be the subset of functions $a = (a^{(0)}, a^{(1)}, a^{(2)})$ such that $a^{(1)}(x, y)$ is constant in $y$, $a^{(2)}(x, y)$ is constant in $x$ and the functions $x \mapsto a^{(1)}(x, y)$ and $y \mapsto a^{(2)}(x, y)$ are elements of $\cV(1/\sqrt{m})$. Let $\EE$ be the space of equivalence classes $L^\infty([0, T]^2)$ equipped with the topology of hypi-convergence. Define the functionals $f_n: \DD_n \to \EE$ by
		\[
		f_n(a)(x, y) := a^{(0)}(x, y) + \sqrt{m} \left( c\left( x + \frac{a^{(1)}(x, y)}{\sqrt{m}}, y + \frac{a^{(2)}(x, y)}{\sqrt{m}} \right) - c(x, y) \right).
		\]

		\Cref{eq:reprWn} can be rephrased as $W_n = f_n(G_n) + \op{1}$, assuming that $G_n \in \DD_n$, which is true with probability
		\[
		\Prob{}{u_n(T) \leq 2T, v_n(T) \leq 2T} \too 1.
		\]

		Let $\DD_0 \subset \DD$ be the subset of continuous functions $a$ such that $a(0) = 0$. As soon as $a_n \in \DD_n$ converges uniformly to $a \in \DD_0$, by \cref{lemma:Hadamard}, $f_n(a_n)$ hypi-converges to $f(a)$, where $f:\DD_0 \to \EE$ satisfies
		\[
		f(a) := a^{(0)} + \dot c_1 a^{(1)} + \dot c_2 a^{(2)}.
		\]

		Note that $(W, -W^{(1)}, -W^{(2)})$ concentrates on $\DD_0$. Therefore, by the extended continuous mapping theorem \citep[Theorem 1.11.1]{VW1996},
		\[
		W_n = f_n(G_n) + \op{1} \wc f((W, -W^{(1)}, -W^{(2)})) = W - \dot c_1 W^{(1)} - \dot c_2 W^{(2)}
		\]
		in $\EE$.  {It remains to show that the result continues to hold if we replace the deterministic sequence $k = k_n$ by data-dependent $\hat k$ as outlined in \cref{rem:hatk}. This is established in \cref{sec:proofhatk}.}
		\hfill $\square$

		\subsubsection{Proof that \cref{big lemma AI,big lemma AD} continue to hold with \texorpdfstring{$\hat k$}{k hat}} \label{sec:proofhatk}

		 {
		Let $\hc_{n, \hat k}$ be the estimator $\hc_n$ computed with the random quantity $\hat k$ instead of $k$. We shall prove that $\sqrt{m} |\hc_{n, \hat k} - \hc_n| \to 0$ in probability uniformly over $[0, T]^2$ (under asymptotic independence) or in the hypi semimetric (under asymptotic dependence).
		}

		Note that the definition of $\hat k$ implies that $\hc_n(\hat k/k, \hat k/k) = 1$. By assumption, $\hc_n$ converges to $c$ in probability uniformly in a neighborhood of $(1, 1)$. Jointly with the fact that $c(x, x) = x^{1/\eta}$, this readily implies that $\hat k/k \to 1$ in probability. Further note that
		\[
		\hc_{n, \hat k}(x, y) = \frac{q(k/n)}{q(\hat k/n)} \hc_n(\hat kx/k, \hat ky/k).
		\]

		We first discuss the case of asymptotic independence. By \cref{big lemma AI} and by Skorokhod's almost sure representation, we may assume that almost surely, $\hc_n = c + W/\sqrt{m} +  {o(1/\sqrt{m})}$ and $\hat k/k \to 1$. The object of interest is then equal, with probability one, to
		\begin{align}
		&\frac{q(k/n)}{q(\hat k/n)} \sqrt{m} \left( \hc_n(\hat kx/k, \hat ky/k) - \frac{q(\hat k/n)}{q(k/n)} \hc_n(x, y) \right) \notag \\
		&\quad =  {\frac{q(k/n)}{q(\hat k/n)}} \Big\{\sqrt{m} \Big( c(\hat kx/k, \hat ky/k) - \frac{q(\hat k/n)}{q(k/n)} c(x, y) \Big) + W(\hat kx/k, \hat ky/k) - W(x, y)\Big\} + o(1) \notag \\
		&\quad = -\sqrt{m} c(x, y) \left( \frac{q(\hat k/n)}{q(k/n)} - \bigg( \frac{\hat k/n}{k/n} \bigg)^{1/\eta} \right) {\frac{q(k/n)}{q(\hat k/n)}} + o(1), \label{eq:rk}
		\end{align}
		where we have used homogeneity of $c$, regular variation of $q$ and the fact that almost surely, the sample paths of $W$ are continuous, hence uniformly continuous on compact sets. The terms $o(1)$ are uniform over $[0, T]^2$. Finally, it is shown in \cref{RV1} that uniformly over $a$ in a neighborhood of 1, $q(at)/q(t) - a^{1/\eta} = O(q_1(t))$. Recalling that $\hat k/k \to 1$ almost surely, the first term in \cref{eq:rk} is then uniformly of the order of $\sqrt{m} q_1(k/n)$, which vanishes by \cref{con-proc}(ii).

		In the case of asymptotic dependence, \cref{big lemma AD} ensures that $\hc_n = c + B/\sqrt{m} + o(1/\sqrt{m})$ in the hypi semimetric. We may apply the reasoning above except that, from the definition of the process $B$, we get the additional term
		\begin{multline}
		-\sum_{j=1}^2 \Big( \dot c_j(\hat kx/k, \hat ky/k) W^{(j)}(\hat kx/k, \hat ky/k) - \dot c_j(x, y) W^{(j)}(x, y) \Big) \\
		= -\sum_{j=1}^2 \dot c_j(x, y) \Big( W^{(j)}(\hat kx/k, \hat ky/k) - W^{(j)}(x, y) \Big);
		\end{multline}
		this follows from the fact that under asymptotic dependence, $c$ is homogeneous of order 1 and the directional partial derivatives of such a function, when they exist, are constant along rays from the origin. The above term vanishes uniformly since $\dot c_j$ has to be locally bounded (only under asymptotic dependence) and since the sample paths of $W^{(j)}$ are almost surely continuous. We therefore obtain \cref{eq:rk}, except that this time the term $o(1)$ is understood in the hypi semimetric. From here on the proof is completed in the same way as under asymptotic independence.
		\hfill $\square$

		\subsubsection{Proof of \cref{big thm}}

		Recall the definition of $\Psi_n$ from \cref{estimation-par}. Letting $\hz_n = \frac{n}{m} \hat\zeta_n$, the assumption that $(\hth_n, \hat\zeta_n)$ minimizes the norm of $\Psi_n^*$ becomes equivalent to $(\hth_n, \hz_n)$ minimizing the norm of $\Psi_n$. The key is to note that for any $\theta, \sigma$,
		\begin{equation} \label{psi diff}
		\Psi(\theta, \sigma) - \Psi_n(\theta, \sigma) = \int g (\hc_n - c) \d\leb = \frac{1}{\sqrt{m}} \int gW_n \d\leb,
		\end{equation}
		with $W_n$ defined as in \cref{big lemma AI,big lemma AD}. By the dominated convergence theorem, and because $g$ is integrable, one easily sees that the functional $f \mapsto \int gf \d\leb$ is continuous in $\ell^\infty([0, T]^2)$. By \cref{prop hypi}, this is also true in the topology of hypi-convergence on $\ell^\infty([0, T]^2)$ at points $f$ that are continuous {Lebesgue-almost everywhere on $[0, T]^2$. It is the case of both limiting Gaussian processes appearing in \cref{big lemma AI,big lemma AD}: $W$, $W^{(1)}$ and $W^{(2)}$ have almost surely continuous sample paths and under asymptotic dependence, the directional derivatives $\dot c_j$ are almost everywhere continuous.} Those two results and the continuous mapping theorem then imply that
		\[
		\int g W_n \d\leb \wc N(0, A).
		\]
		We may therefore apply \cref{M} with $\phi = \Psi$, $x_0 = (\theta_0, 1)$, $Y_n = \frac{1}{\sqrt{m}} \int gW_n \d\leb$ and $a_n = 1/\sqrt{m}$, and as required we obtain
		\[
		\sqrt{m} ((\hth_n, \hz_n) - (\theta_0, 1)) = (J^\top J)^{-1} J^\top \int g W_n \d\leb + \op{1} \wc N(0, \Sigma).
		\]
		\hfill $\square$

		\subsection{Spatial estimation}

		For the proofs in the spatial setting, we assume the framework of \cref{estimation spatial}, we define the transformed random variables $U^{(j)} = 1 - F^{(j)}(X^{(j)})$ and for a pair $s$, let $Q^{(s)}$ be the distribution function of the random vector $(U^{(s_1)}, U^{(s_2)})$. Define the transformed observations $U_i^{(j)} = 1 - F^{(j)}(X_i^{(j)})$ and denote by $U_{n, 1}^{(j)}, \dots, U_{n, n}^{(j)}$ the ordered versions thereof  {and define $U_{n, 0}^{(j)} := 0$}. For intermediate sequences $k^{(s)}$, we define the (weighted) empirical tail quantile functions $u_n^{(s, j)}$, $s \in \cP, j \in \{1, 2\}$, by
		\[
		u_n^{(s, j)}(x) = \frac{n}{k^{(s)}} U_{n, \lfloor k^{(s)} x \rfloor}^{(s_j)}, \quad x \geq 0.
		\]
		Recalling that $m^{(s)} = nq^{(s)}(k^{(s)}/n)$, it allows us to write
		\[
		\hc_n^{(s)}(x, y) = \frac{n}{m^{(s)}} Q_n^{(s)}\left( \frac{k^{(s)}}{n} u_n^{(s, 1)}(x), \frac{k^{(s)}}{n} u_n^{(s, 2)}(y) \right).
		\]
		where  { $Q_n^{(s)}$ denotes the empirical distribution function of $(U_1^{(s_1)},U_1^{(s_2)}),\ldots,(U_n^{(s_1)},U_n^{(s_2)})$. Following the discussion before the proof of \cref{big lemma AI}, we may define
		\begin{multline*}
		H_n^{(s)}(x, y) := \sqrt{m^{(s)}} \Big\{ \frac{1}{m^{(s)}} \sum_{i=1}^n \mathbb{I} \Big\{ U_i^{(s_1)} \leq \frac{k^{(s)}}{n} x, U_i^{(s_2)} \leq \frac{k^{(s)}}{n} y\Big\}
		\\
		- \frac{n}{m^{(s)}} \mathbb{P}\Big({U^{(s_1)} \leq \frac{k^{(s)}}{n} x, U^{(s_2)} \leq \frac{k^{(s)}}{n} y}\Big) \Big\}.
		\end{multline*}
		and similarly obtain
		\begin{equation} \label{eq:reprWn-spatial}
		W_n^{(s)}(x, y) = H_n^{(s)}(x, y) + \sqrt{m^{(s)}} \left( c^{(s)} \left( u_n^{(s, 1)}(x), u_n^{(s, 2)}(y) \right) - c^{(s)}(x, y) \right) + \op{1},
		\end{equation}
		where $W_n^{(s)}$ is defined as in \cref{thm spatial np} and the term $\op{1}$ is uniform over compact sets.
		}

		\subsubsection{Proof of \cref{thm spatial np}}

		For asymptotically independent pairs, the second term of \cref{eq:reprWn-spatial} vanishes uniformly, by the proof of \cref{big lemma AI}. Define the $\DD$-valued processes $G_n$ by
		\[
		G_n(x, y) := \left( \left( H_n^{(s)}(x, y) \right)_{s \in \cP}, \left( \sqrt{m^{(s)}} \left( u_n^{(s, 1)}(x) - x \right), \sqrt{m^{(s)}} \left( u_n^{(s, 2)}(y) - y \right) \right)_{s \in \cP_D} \right),
		\]
		where $\DD = \left( \ell^\infty([0, 2T]^2) \right)^{|\cP| + 2|\cP_D|}$. The proof now proceeds similarly to that of \cref{big lemma AD}; we show that $G_n$ converges in distribution, that the processes of interest $W_n^{(s)}$ can be approximately represented as a transformation of $G_n$, and we conclude by applying a continuous mapping theorem.

		For $s \in \cP$, $j \in \{1, 2\}$, let
		\[
		L_n^{(s, j)}(x) = \frac{1}{k^{(s)}} \sum_{i=1}^n \Ind{U^{(s_j)} \leq \frac{k^{(s)}}{n} x}, \quad x \geq 0.
		\]
		Recall that $I$ denotes the identity mapping on $\R$. By standard arguments (see, e.g., the proofs of \cref{big lemma AI,big lemma AD}), we see that each of the processes $H_n^{(s)}$ and $\sqrt{m^{(s)}} \left( L_n^{(s, j)} - I \right)$ converge in distribution in $\ell^\infty([0, 2T]^2)$, hence they are tight random elements in that space. It follows that the sequence of processes
		\begin{equation} \label{eq:preGn spatial}
		(x, y) \mapsto \left( \left( H_n^{(s)}(x, y) \right)_{s \in \cP}, \left( \sqrt{m^{(s)}} \left( L_n^{(s, 1)}(x) - x \right), \sqrt{m^{(s)}} \left( L_n^{(s, 2)}(y) - y \right) \right)_{s \in \cP_D} \right)
		\end{equation}
		is tight in the product space $\DD$. A Lindeberg-type condition \citep[Theorem 2.27]{V2000} can easily be checked, so weak convergence of the process in \cref{eq:preGn spatial} follows from convergence of $\E{}{G_n(x, y) G_n(x', y')^\top}$ to a suitable covariance matrix. This is simply a consequence of \cref{con-spatial-new}; indeed, for suitable pairs $s, s' \in \cP$, $j, j' \in \{1, 2\}$ and $(x, y), (x', y') \in [0, \infty)^2$, this condition implies that
		\begin{align*}
		\lim_{n \to \infty} \E{}{H_n^{(s)}(x, y) H_n^{(s')}(x', y')} &= \Gamma^{(s, s')}((x, y), (x', y')), \\
		\lim_{n \to \infty} \E{}{H_n^{(s)}(x, y) \sqrt{m^{(s')}} \left( L_n^{(s', j)}(x') - x' \right)} &= \Gamma^{(s, s', j)}((x, y), (x', y')), \\
		\lim_{n \to \infty} \E{}{\sqrt{m^{(s)}} \left( L_n^{(s, j)}(x) - x \right) \sqrt{m^{(s')}} \left( L_n^{(s', j')}(x') - x' \right)} &= \Gamma^{(s, j, s', j')}((x, y), (x', y')).
		\end{align*}
		We deduce that in $\DD$, the processes in \cref{eq:preGn spatial} weakly converge to the Gaussian process
		\[
		\left( (W^{(s)})_{s \in \cP}, (W^{(s, j)})_{s \in \cP_D, j \in \{1, 2\}} \right)
		\]
		as defined in \cref{sec:asy-spatial}. Noting that $u_n^{(s, j)}$ is the generalized inverse function of $L_n^{(s, j)} + 1/k^{(s)}$ and that $\sqrt{m^{(s)}}/k^{(s)} \to 0$, we apply Vervaat's lemma \citep{V1972} to obtain that
		\begin{equation} \label{eq:limit Gn}
		G_n \wc G := \left( (W^{(s)})_{s \in \cP}, (-W^{(s, j)})_{s \in \cP_D, j \in \{1, 2\}} \right)
		\end{equation}
		in $\DD$.

		Recall the definition of the sets $\cV(t)$ in \cref{eq:defcV(t)} and let $\DD_n \subset \DD$ be the subset of functions $a$ of the form $\big( (a^{(s)})_{s \in \cP}, (a^{(s, j)})_{s \in \cP_D, j \in \{1, 2\}} \big)$ such that $a^{(s, 1)}(x, y)$ is constant in $y$, $a^{(s, 2)}(x, y)$ is constant in $x$ and such that the functions $x \mapsto a^{(s, 1)}(x, y)$ and $y \mapsto a^{(s, 2)}(x, y)$ are elements of $\cV\big(1/\sqrt{m^{(s)}}\big)$. 

		Defining $\EE$ as the product space $\left( L^\infty([0, T]^2) \right)^{|\cP|}$, with $L^\infty([0, T]^2)$ equipped with the topology of hypi-convergence, consider the following functionals $f_n: \DD_n \to \EE$. For an element $a = \big( (a^{(s)})_{s \in \cP}, (a^{(s, j)})_{s \in \cP_D, j \in \{1, 2\}} \big) \in \DD_n$, $f_n(a) = (f_n(a)^{(s)})_{s \in \cP}$ is a function such that $f_n(a)^{(s)} = a^{(s)}$ if $s \in \cP_I$, and
		\[
		f_n(a)^{(s)}(x, y) = a^{(s)}(x, y) + \sqrt{m^{(s)}} \left( c^{(s)}\left( x + \frac{a^{(s, 1)}(x, y)}{\sqrt{m^{(s)}}}, y + \frac{a^{(s, 2)}(x, y)}{\sqrt{m^{(s)}}} \right) - c^{(s)}(x, y) \right)
		\]
		if $s \in \cP_D$. Referring to \cref{eq:reprWn-spatial} and recalling that the second term thereof vanishes if $s \in \cP_I$, we notice that for every pair $s$, $W_n^{(s)} = f_n(G_n)^{(s)} + \op{1}$. This representation, of course, holds only if $G_n \in \DD_n$; this is satisfied with probability at least
		\[
		\Prob{}{ \forall s \in \cP_D, j \in \{1, 2\}, u_n^{(s, j)}(T) \leq 2T } \too 1
		\]
		where the last convergence follows by \cref{coro app} applied for each $s \in \cP$. Define $f:\DD_0 \to \EE$, where $\DD_0 \subset \DD$ is the subset of continuous functions $a$ such that $a(0) = 0$, as
		\[
		f(a)^{(s)} = \begin{cases}
		a^{(s)}, &\quad s \in \cP_I \\
		a^{(s)} + \dot c_1 a^{(s, 1)} + \dot c_2 a^{(s, 2)}, &\quad s \in \cP_D
		\end{cases}.
		\]
		For a sequence $a_n \in \DD_n$ that converges uniformly to a function $a \in \DD_0$, $f_n(a_n) \to f(a)$ in $\EE$. This can be seen by considering each pair separately; the result is obvious for asymptotically independent pairs, and for asymptotically dependent ones it follows from \cref{lemma:Hadamard}.

		Finally, notice that the process $G$ concentrates on $\DD_0$. Therefore, by \cref{eq:limit Gn} and the extended continuous mapping theorem \citep[Theorem 1.11.1]{VW1996},
		\[
		\Big( W_n^{(s)} \Big)_{s \in \cP} = f_n(G_n) + \op{1} \wc f(G) = \Big( B^{(s)} \Big)_{s \in \cP}
		\]
		in $\EE$.
		\hfill $\square$

		\subsubsection{Proof of \cref{thm spatial ls}}

		Similarly to the bivariate case, let 
		\[
		\Psi_n^{(s)}(\theta, \sigma) := (n/m) \Psi_n^{*(s)}(\theta, m \sigma/n).
		\]
		As in the proof of \cref{big thm}, we may deduce that for every pair $s$, $\theta \in \tilde\Theta$ and $\sigma>0$,
		\[
		\Psi^{(s)}(\theta, \sigma) - \Psi_n^{(s)}(\theta, \sigma) = \int g\left( \hc_n^{(s)} - c^{(s)} \right) \d\leb = \frac{1}{\sqrt{m}} \int gW_n^{(s)} \d\leb,
		\]
		with $W_n^{(s)}$ as defined in \cref{thm spatial np}. By a similar argument to the bivariate case (involving the dominated convergence theorem and \cref{prop hypi} to establish continuity of the mapping $f \mapsto \int gf \d\leb$,  {see the proof of \cref{big thm} for the applicability of \cref{prop hypi}}), \cref{thm spatial np} and the continuous mapping theorem yield
		\begin{equation} \label{eq:spatial integral}
		\left( \int gW_n^{(s)} \d\leb \right)_{s \in \cP} \wc \left( \int gB^{(s)} \d\leb \right)_{s \in \cP}.
		\end{equation}
		The remaining proof consists of a number of successive applications of \cref{M}. We deal with each of the two estimators separately.

		\begin{itemize}

		\item[(i)] For each pair $s$, applying \cref{M} with $\phi = \Psi^{(s)}$, $x_0 = (h^{(s)}(\vartheta_0), 1)$, $a_n = 1/\sqrt{m}$ and $Y_n = \frac{1}{\sqrt{m}} \int gW_n^{(s)} \d\leb$ yields
		\begin{equation} \label{eq:pairwise}
		\hth_n^{(s)} - h^{(s)}(\vartheta_0) = \frac{1}{\sqrt{m}} \cD^{(s)} \int gW_n^{(s)} \d\leb + \op{\frac{1}{\sqrt{m}}},
		\end{equation}
		where $\cD^{(s)}$ is the block corresponding to the pair $s$ in the matrix $\cD$ defined in \cref{eq:defcDs}; its existence, as well as the required smoothness of $\phi$, are guaranteed by \cref{con-spatial-par}. Now redefining $\phi$ as $\phi(\vartheta) = \big( h^{(s)}(\vartheta) - h^{(s)}(\vartheta_0) \big)_{s \in \cP}$,
		we see that $\hat\vartheta_n$ is in fact a minimizer of the norm of $\phi(\vartheta) - Y_n$, where $Y_n$ is redefined as $\big( \hth_n^{(s)} - h^{(s)}(\vartheta_0) \big)_{s \in \cP}$. Applying \cref{M} again with $\phi$ and $Y_n$ as above, $x_0 = \vartheta_0$ and $a_n = 1/\sqrt{m}$, we obtain
		\begin{align*}
		\hat\vartheta_n - \vartheta_0 &= (J_1^\top J_1)^{-1} J_1^\top Y_n + \op{\frac{1}{\sqrt{m}}} \\
		&= \frac{1}{\sqrt{m}} (J_1^\top J_1)^{-1} J_1^\top \left( \cD^{(s)} \int gW_n^{(s)} \d\leb \right)_{s \in \cP} + \op{\frac{1}{\sqrt{m}}},
		\end{align*}
		where the last equality follows from \cref{eq:pairwise} and $J_1$ is defined as in \cref{sec:asy-spatial} in the paragraph below \cref{eq:defcDs}. The conclusion that $\sqrt{m} (\hat\vartheta_n - \vartheta_0) \wc N(0,\Sigma_1)$ follows from this and \cref{eq:spatial integral}.

		\item[(ii)] Let $\tilde\sigma_n = \frac{n}{m} \tilde\zeta_n \in \R_+^{|\cP|}$. Once more, we redefine
		\[
		Y_n = \frac{1}{\sqrt{m}} \left( \int gW_n^{(s)} \d\leb \right)_{s \in \cP} \quad \text{and} \quad \phi(\vartheta, \sigma) = \left( \Psi^{(s)}(h^{(s)}(\vartheta), \sigma^{(s)}) \right)_{s \in \cP}.
		\]
		The estimator $(\tilde\vartheta_n, \tilde\sigma_n)$ can be seen to minimize the norm of $\phi - Y_n$. Therefore, applying \cref{M} with $a_n = 1/\sqrt{m}$ and $x_0 = (\vartheta_0, 1, \dots, 1)$, we obtain
		\[
		(\tilde\vartheta_n, \tilde\sigma_n) - (\vartheta_0, 1, \cdots, 1) = \frac{1}{\sqrt{m}} (J_2^\top J_2)^{-1} J_2^\top \left( \int gW_n^{(s)} \d\leb \right)_{s \in \cP} + \op{\frac{1}{\sqrt{m}}},
		\]
		which, combined with \cref{eq:spatial integral}, implies $\sqrt{m} ((\tilde\vartheta_n, \tilde\sigma_n) - (\vartheta_0, 1, \cdots, 1)) \wc N(0, \Sigma_2)$.

		\end{itemize}
		\hfill $\square$

		\section{Technical results used in \cref{proofs}}
		\label{tech}

		Throughout the paper, particularly the proof of \cref{RV1} below, we use (without reference when obvious) the following results on regularly varying functions at 0.

		\begin{lemm} \label{RV}

		Suppose the functions $f_1$ and $f_2$ are regularly varying at 0 with indices $\rho_1$ and $\rho_2$, respectively.

		\begin{itemize}
			
			\item[(i)] If $\rho_1 > 0$ (respectively $\rho_1 < 0$), $\lim_{t \to 0} f_1(t) = 0$ (respectively $\infty$).
			
			\item[(ii)] For any $\alpha \in \R$, $f_1^\alpha$ is $(\alpha \rho_1)$--RV at 0.
			
			\item[(iii)] The product $f_1 f_2$ is $(\rho_1 + \rho_2)$--RV at 0.
			
			\item[(iv)] If $\lim_{t \to 0} f_2(t) = 0$, then $f_1 \circ f_2$ is $(\rho_1 \rho_2)$--RV at 0.
			
			\item[(v)] If $\rho_1 > 0$, then $f_1^{-1}$ is $(1/\rho_1)$--RV at 0, where we define the generalized inverse of $f_1$ as
			$$
			f_1^{-1}(t) = \inf \{u > 0: f_1(u) \geq t\}.
			$$
			
		\end{itemize}

		\end{lemm}

		\begin{proof}

		The assertions (ii) and (iii) are trivial consequences of the definition of regular variation. As for (i), (iv) and (v), analogue versions for regularly varying functions at $\infty$ are proved in Proposition 0.8 of \cite{R1987}. The proof can readily be adapted, using the fact that $f$ is $\rho$--RV at 0 if and only if $u \mapsto 1/f(1/u)$ is $\rho$--RV at $\infty$.

		\end{proof}

		\begin{lemm} \label{RV1}
		\begin{itemize}
			
			\item[(i)] Assume \cref{Q}. Then there exists $\eta \in (0, 1]$ such that $q$ is a regularly varying (RV) function at 0 with index $1/\eta$ and $c$ is $1/\eta$-homogeneous.
			
			\item[(ii)] Assume \cref{con-proc}(i) and suppose that $q_1$ is non-decreasing and that there exists $b > 1$ such that $q_1(bt) = O(q_1(t))$ as $t \to 0$. Then \cref{Q} holds locally uniformly on $[0, \infty)^2$.
			
		\end{itemize}
		\end{lemm}

		\begin{rem}

		In part (ii) of the previous result, the monotonicity condition on $q_1$ is artificial; it can be removed at the cost of replacing $q_1(t)$ by the non-decreasing function $\bar{q_1}(t) := \sup_{0 < s \leq t}q_1(s)$. Indeed, if \cref{con-proc} is satisfied with $q_1$, it is trivially satisfied with $\bar{q_1}$. Moreover, if $q_1(bt) = O(q_1(t))$, $\bar{q_1}$ also satisfies the same property.

		Because $q_1$ is positive non-decreasing, that required property implies that $q_1(bt) = O(q_1(t))$ holds for every $b \geq 1$ \citep[Corollary 2.0.6]{BGT1987}. The function $q_1$ is then said to be $O$-regularly varying at 0.

		\end{rem}

		\begin{proof} 

		\begin{itemize}
			\item[(i)] Recall that we assume $c(1, 1) = 1$. For any $x>0$, \cref{Q} implies that $Q(tx, tx) = q(t)(c(x, x) + o(1))$ and $Q(tx, tx) = q(tx)(1 + o(1))$. This can be manipulated into
			$$
			\frac{q(tx)}{q(t)} = \frac{c(x, x) + o(1)}{1 + o(1)} \too c(x, x).
			$$
			
			By Karamata's characterization theorem \citep[Theorem 1.4.1]{BGT1987}, $q$ has to be $\rho$--RV and $c(x, x) = x^\rho$, for some $\rho \in \R$. However, since $q(t) \leq t$, we must have $\rho \geq 1$. Moreover, for any $a, x, y > 0$,
			$$
			c(ax, ay) = \lim_{t \to 0} \frac{Q(atx, aty)}{q(t)} = \lim_{t \to 0} \frac{Q(tx, ty)}{q(t/a)} = \lim_{t \to 0} \frac{Q(tx, ty)}{q(t)} \frac{q(t)}{q(t/a)} = a^\rho c(x, y).
			$$
			Defining $\eta = 1/\rho$, this proves (i).

			\item[(ii)] For arbitrary $(x, y) \in [0, \infty)^2$, we write $(x, y) = a(u, v)$. We will prove that \cref{Q} holds uniformly over all $(u, v) \in \cS^+$ and over $a \in (0, b]$, for an arbitrary $b \in [1, \infty)$.
			
			We have
			\begin{equation} \label{blabla}
			\frac{Q(tx, ty)}{q(t)} = \frac{Q(atu, atv)}{q(t)} = \frac{q(at)}{q(t)} \frac{Q(atu, atv)}{q(at)}.
			\end{equation}
			
			First, the term $Q(atu, atv) / q(at)$ is equal to $c(u, v) + O(q_1(at))$ uniformly in $(u, v) \in \cS^+$. In order to control the term $q(at)/q(t)$, we note that since $q$ is $1/\eta$-RV, there exists a slowly varying function $L$ such that for any $a>0$,
			\begin{align*}
			\frac{L(at)}{L(t)} - 1 &= a^{-1/\eta} \left( \frac{q(at)}{q(t)} - c(a, a) \right) \\
			&= a^{-1/\eta} \left( \frac{ Q(at, at) (1 + O(q_1(at))) }{q(t)} - c(a, a) \right) \\
			&= a^{-1/\eta} \left( \frac{Q(at, at)}{q(t)} - c(a, a) + O(q_1(at)) \right) \\
			&= O(q_1(t) + q_1(at)) = O(q_1(bt)) = O(q_1(t)),
			\end{align*}
			where we have used the fact that $Q(at, at) = q(at)(1 + O(q_1(at)))$, which can be reversed into $q(at) = Q(at, at)(1 + O(q_1(at)))$. The function $L$ is thus \emph{slowly varying with remainder} \citep[Section 3.12]{BGT1987}. By theorem 3.12.1 of that book, the previous relation holds uniformly over all $a \in (1/2, b]$, so we henceforth focus on values $a \in (0, 1/2]$. Using Theorem 3.12.2 of the same book (which we adapt for slow variation at 0), we obtain that for some constants $C \in \R, T_0 \in (0, \infty)$ and for $t$ small enough,
			$$
			L(t) = \exp\left\{ C + \delta_1(t) + \int_t^{T_0} \frac{\delta_2(s)}{s} \d s \right\},
			$$
			where the functions $\delta_j$ are real-valued, measurable and satisfy $|\delta_j(t)| \leq K q_1(t)$ for some constant $K \in (0, \infty)$. The ratio $L(at)/L(t)$ becomes
			$$
			\frac{L(at)}{L(t)} = \exp\left\{ \delta_1(at) - \delta_1(t) + \int_{at}^t \frac{\delta_2(s)}{s} \d s \right\}.
			$$
			
			As $t \to 0$, we can use the monotonicity of $q_1$ to control the integral in the previous display:
			$$
			\left| \int_{at}^t \frac{\delta_2(s)}{s} \d s \right| \leq K \int_{at}^t \frac{q_1(s)}{s} \d s \leq K q_1(t) \int_{at}^t \frac{\d s}{s} = K q_1(t) \log \left( \frac{1}{a} \right).
			$$
			
			Because $a \leq 1/2$, $\log(1/a)$ is lower bounded, so $K$ can be chosen large enough so that $K q_1(t) \log(1/a)$ also upper bounds the absolute value of $\delta_1(at) - \delta_1(t) + \int_{at}^t \frac{\delta_2(s)}{s} \d s$. Therefore, using the fact that for every $h \in \R$, $|e^h - 1| \leq e^{|h|} - 1$, we obtain
			$$
			\left| \frac{L(at)}{L(t)} - 1 \right| \leq \exp\left\{ K q_1(t) \log \left( \frac{1}{a} \right) \right\} - 1 = a^{-K q_1(t)} - 1.
			$$
			
			What we are interested in is bounding $q(at)/q(t) - a^{1/\eta}$. This can be done by recalling that
			\begin{equation} \label{bound tau}
			\left| \frac{q(at)}{q(t)} - a^{1/\eta} \right| = a^{1/\eta} \left| \frac{L(at)}{L(t)} - 1 \right| \leq a^{1/\eta} \left( a^{- K q_1(t)} - 1 \right) =: \tau(a, t).
			\end{equation}
			
			By simple differentiation, it is straightforward to see that for a fixed value of $t$ small enough so that $K q_1(t) < 1/\eta$, the function $\tau$ is differentiable in its first argument and that
			$$
			\frac{\partial}{\partial a} \tau(a, t) = a^{1/\eta - 1} \left( (1/\eta - K q_1(t)) a^{-K q_1(t)} - 1/\eta \right).
			$$
			
			This suggests that the function attains its unique maximum at the point $a_{\max}(t) := (1 - \eta K q_1(t))^{1/(K q_1(t))}$. Considering \cref{bound tau}, we obtain that for all $a \in (0, 1/2]$,
			\begin{align*}
			\left| \frac{q(at)}{q(t)} - a^{1/\eta} \right| &\leq \tau(a_{\max}(t), t) \\
			&= (1 - \eta K q_1(t))^{1/(\eta K q_1(t))} \left( \frac{1}{1 - \eta K q_1(t)} - 1 \right) \\
			&= O(q_1(t))
			\end{align*}
			as $t \to 0$, since $(1 - \eta K q_1(t))^{1/(\eta K q_1(t))} \to e^{-1}$ and since the function $x \mapsto 1/(1-x)$ is continuously differentiable at 0. Finally, this allows us to rewrite \cref{blabla} as
			$$
			\frac{Q(tx, ty)}{q(t)} = \left( a^{1/\eta} + O(q_1(t)) \right) ( c(u, v) + O(q_1(at)) ) = a^{1/\eta} c(u, v) + O(q_1(t)),
			$$
			and the last equation holds uniformly over $a \in (0, b]$ and $(u, v) \in \cS^+$. The proof is over since $a^{1/\eta} c(u, v) = c(x, y)$.
			
		\end{itemize}

		\end{proof}

		\begin{lemm} \label{Csorgo}

		Let $\varphi: (0, T] \to (0, \infty)$ be a non-decreasing function such that $\varphi(t)/\sqrt{t} \to \infty$ as $t \to 0$ and assume there exists $c>0$ such that
		$$
		\int_0^T \frac{1}{x} \exp \left\{ -c \frac{\varphi^2(x)}{x} \right\} \d x < \infty.
		$$

		Then under the assumptions of \cref{big lemma AI}, for every $\lambda \in (0, 1)$ we have
		$$
		\sup_{\lambda/k \leq x \leq T} \frac{\sqrt{k}}{\varphi(x)} |u_n(x) - x| = \Op{1},
		$$
		where $u_n$ is defined as in \cref{proofs bivariate}. In particular, note that $\varphi(x) := 1$, as well as any function that satisfies $\varphi(x) := \sqrt{x \log\log(1/x)}$ in a neighborhood of 0, are valid choices.

		\end{lemm}

		\begin{proof}

		This is essentially proved in \cite{CH1987}, up to a slight difference between their definition of the quantiles and ours. We prove here that this difference does not change the result. More precisely, their Theorem 2.6 (ii) states that
		\begin{equation} \label{2.6}
		\sup_{\lambda/k \leq x \leq T} \frac{|w_n(x)|}{\varphi(x)} = \Op{1},
		\end{equation}
		where we denote $w_n$ what they call $v_n$ (to avoid confusion with our definitions). From their definitions, one easily sees that
		$$
		w_n(x) = \frac{n}{\sqrt{k}} \left( \frac{k}{n} x - U_{n, \lceil kx \rceil} \right) = \sqrt{k} \left( x - \frac{n}{k} U_{n, \lceil kx \rceil} \right).
		$$

		Then, by the reverse triangle inequality,
		\begin{align*}
		| \sqrt{k}|u_n(x) - x| - |w_n(x)| | &\leq |\sqrt{k}(u_n(x) - x) + w_n(x)| \\
		&= \sqrt{k} \left| u_n(x) - \frac{n}{k} U_{n, \lceil kx \rceil} \right| = \frac{n}{\sqrt{k}} \left( U_{n, \lceil kx \rceil} - U_{n, \lfloor kx \rfloor} \right).
		\end{align*}

		Using this and the inequality $\lfloor x \rfloor \geq \lceil x \rceil - 1$, we have
		\begin{align}
		&\left| \sup_{\lambda/k \leq x \leq T} \frac{\sqrt{k}}{\varphi(x)} |u_n(x) - x| - \sup_{\lambda/k \leq x \leq T} \frac{|w_n(x)|}{\varphi(x)} \right| \notag \\
		&\quad \leq \frac{n}{\sqrt{k}} \sup_{\lambda/k \leq x \leq T} \frac{1}{\varphi(x)} \left( U_{n, \lceil kx \rceil} - U_{n, \lfloor kx \rfloor} \right) \notag \\
		&\quad \leq \frac{n}{\sqrt{k}} \sup_{\lambda/k \leq x \leq T} \frac{1}{\varphi(x)} \left( U_{n, \lceil kx \rceil} - U_{n, \lceil kx \rceil - 1} \right) \notag \\
		&\quad \leq \frac{n}{\sqrt{k}} \sup_{\lambda/k \leq x \leq (1+\lambda)/k} \frac{1}{\varphi(x)} \left( U_{n, \lceil kx \rceil} - U_{n, \lceil kx \rceil - 1} \right) \notag \\
		&\quad\quad + \frac{n}{\sqrt{k}} \sup_{(1+\lambda)/k \leq x \leq T} \frac{1}{\varphi(x)} \left( U_{n, \lceil kx \rceil} - U_{n, \lceil kx \rceil - 1} \right). \label{2 sup}
		\end{align}

		In the first term, since $\lambda/k \leq x \leq (1+\lambda)/k$ and $\lambda \in (0, 1)$, we must have $\lceil kx \rceil \in \{1, 2\}$. Therefore, we end up studying $U_{n, i} - U_{n, i-1}$, for some $i \in \{1, 2\}$. It is a well known fact that those differences, regardless of the value of $i$, have a Beta distribution with parameters 1 and $n$. In particular, they are both $\Op{1/n}$. It follows that the first supremum on the right hand side of \cref{2 sup} is asymptotically bounded in probability by
		$$
		\frac{1}{\sqrt{k}} \sup_{\lambda/k \leq x \leq (1+\lambda)/k} \frac{1}{\varphi(x)} = \frac{1}{ \sqrt{k} \varphi(\lambda/k) } \too 0
		$$
		by assumption on $\varphi$. As for the second term in \cref{2 sup}, it is equal to
		\begin{align*}
		\frac{n}{\sqrt{k}} &\sup_{(1+\lambda)/k \leq x \leq T} \frac{1}{\varphi(x)} \left( U_{n, \lceil kx \rceil} - U_{n, \lceil k (x - 1/k) \rceil} \right) \\
		&= \frac{n}{\sqrt{k}} \sup_{\lambda/k \leq x \leq T - 1/k} \frac{1}{\varphi(x + 1/k)} \left( U_{n, \lceil k (x+1/k) \rceil} - U_{n, \lceil kx \rceil} \right)
		\end{align*}
		after shifting $x$ to the right by $1/k$. Using \cref{2.6}, this is in turn equal to
		\begin{align*}
		\frac{n}{\sqrt{k}} &\sup_{\lambda/k \leq x \leq T - 1/k} \frac{1}{\varphi(x + 1/k)} \left( \frac{k}{n} \left( x + \frac{1}{k} \right) - \frac{k}{n} x \right) + \frac{n}{\sqrt{k}} \Op{\frac{\sqrt{k}}{n}} \\
		&= \frac{1}{\sqrt{k}} \sup_{\lambda/k \leq x \leq T - 1/k} \frac{1}{\varphi(x + 1/k)} + \Op{1} \\
		&= \frac{1}{\sqrt{k} \varphi((1+\lambda)/k)} + \Op{1} \\
		&= \Op{1}
		\end{align*}
		once again by the properties of $\varphi$. We have shown that the difference between the quantity we are interested in and the term appearing in \cref{2.6} is $\Op{1}$. We may thus conclude, by \cref{2.6}, that
		$$
		\sup_{\lambda/k \leq x \leq T} \frac{\sqrt{k}}{\varphi(x)} |u_n(x) - x| = \sup_{\lambda/k \leq x \leq T} \frac{|w_n(x)|}{\varphi(x)} + \Op{1} = \Op{1}.
		$$

		\end{proof}

		\begin{coro} \label{coro app}

		Define the random functions $u_n$ and $v_n$ as in \cref{proofs bivariate}. Then, as $n \to \infty$,
		$$
		\sup_{0 \leq x \leq 2T} |u_n(x) - x| \quad \text{and} \quad \sup_{0 \leq y \leq 2T} |v_n(y) - y|
		$$
		are both $\Op{1/\sqrt{k}}$.

		\end{coro}

		\begin{proof}

		Note that by definition, $u_n(z) = v_n(z) = 0$ whenever $z<1/k$. It follows that
		\begin{align*}
		\sup_{0 \leq x \leq 2T} |u_n(x) - x| &\leq \sup_{0 \leq x < 1/k} |u_n(x) - x| + \sup_{1/k \leq x \leq 2T} |u_n(x) - x| \\
		&= \sup_{0 \leq x < 1/k} x + \sup_{1/k \leq x \leq 2T} |u_n(x) - x| \\
		&= \frac{1}{k} + \sup_{1/k \leq x \leq 2T} |u_n(x) - x|.
		\end{align*}

		This is $\Op{1/\sqrt{k}}$ by the preceding \cref{Csorgo} with the function $\varphi(x) = 1$. The same proof holds with $u_n$ replaced by $v_n$.

		\end{proof}

		\begin{lemm} \label{limit Gn}

		 {Under \cref{con-proc} the process $H_n$ as defined in \cref{eq:defHn} converges to the process $W$ from \cref{big lemma AI} in $\ell^\infty([0,2T]^2)$.}

		\end{lemm}

		\begin{proof}

		Denoting $f_{n, (x, y)}(u, v) := \sqrt{\frac{n}{m}} \Ind{u \leq \frac{k}{n} x, v \leq \frac{k}{n} y}$, we see that $H_n$ can be written as
		$$
		H_n(x, y) = \sqrt{n} \left( \frac{1}{n} \sum_{i=1}^n f_{n, (x, y)}(U_i, V_i) - \E{}{f_{n, (x, y)}(U, V)} \right).
		$$

		Therefore, convergence of the process $H_n$ to a Gaussian process in $\ell^\infty([0, 2T]^2)$ is equivalent to checking that the sequence of function classes
		$$
		\cF_n = \{f_{n, (x, y)} : (x, y) \in [0, 2T]^2\}
		$$
		are Donsker classes for the distribution of $(U, V)$. This is guaranteed by Theorem 11.20 of \cite{K2008}, provided that we can check the six conditions. Note that $\cF_n$ admits the envelope function $F_n = f_{n, (2T, 2T)}$.

		\begin{itemize}
			
			\item[(0)] First, the AMS condition is trivially satisfied; by right continuity of indicator functions, for any $n \in \N$, $(x, y) \in [0, 2T]^2$ and $(u, v) \in [0, 1]^2$,
			$$
			\inf_{(x', y') \in \Q^2} |f_{n, (x', y')}(u, v) - f_{n, (x, y)}(u, v)| = 0.
			$$
			It follows that Equation (11.7) of \cite{K2008} is satisfied with $T_n = \Q^2$, which is countable. Hence the classes $\cF_n$ are AMS.
			
			\item[(A)] For every $n$, it is easily checked that $\cF_n$ is a VC class with VC-index 2. Therefore, condition (A) is a direct consequence of Lemma 11.21 of \cite{K2008}.
			
			\item[(B)] For $(x, y), (x', y') \in [0, 2T]^2$ arbitrary, it follows from the definition of $H_n$ that
			\begin{align*}
			\E{}{H_n(x, y) H_n(x', y')} &= \E{}{f_{n, (x, y)}(U, V) f_{n, (x', y')}(U, V)} \\
			&\phantom{=} - \E{}{f_{n, (x, y)}(U, V)} \E{}{f_{n, (x', y')}(U, V)} \\
			&= \frac{n}{m} \Prob{}{U \leq \frac{k}{n} (x \wedge x'), V \leq \frac{k}{n} (y \wedge y')} \\
			&\phantom{=} - \frac{n}{m} \Prob{}{U \leq \frac{k}{n} x, V \leq \frac{k}{n} y} \Prob{}{U \leq \frac{k}{n} x', V \leq \frac{k}{n} y'}.
			\end{align*}
			
			Recall that $n/m = 1/q(k/n)$. Therefore, the first term of the last display converges to $c(x \wedge x', y \wedge y')$. The second term vanishes since both probabilities are of the order of $m/n$. The convariance functions of $H_n$ thus converge pointwise to the covariance function of $W$.
			
			\item[(C)] By definition of the envelope functions and by assumption, we have
			$$
			\mathop{\lim\sup}_{n \to \infty} \E{}{F_n^2(U, V)} = \mathop{\lim\sup}_{n \to \infty} \frac{n}{m} \Prob{}{U \leq \frac{k}{n} 2T, V \leq \frac{k}{n} 2T} = c(2T, 2T) < \infty.
			$$
			
			\item[(D)] For every $\eps>0$,
			$$
			\E{}{F_n^2(U, V) \Ind{F_n(U, V) > \eps \sqrt{n}}} \leq \frac{n}{m} \Ind{\sqrt{\frac{n}{m}} > \eps \sqrt{n}},
			$$
			which is equal to 0 as soon as $m \geq \eps^{-2}$.
			
			\item[(E)] We first recall that for arbitrary events $A, B$,
			$$
			\Prob{}{\ind_A \neq \ind_B} = \Prob{}{A \backslash B} + \Prob{}{B \backslash A} = \Prob{}{A} + \Prob{}{B} - 2 \Prob{}{A \cap B}.
			$$
			
			A direct application of this fact yields
			\begin{align*}
			\rho_n^2((x, y), (x', y')) :&= \E{}{(f_{n, (x, y)}(U, V) - f_{n, (x', y')}(U, V))^2} \\
			&= \frac{n}{m} \Prob{}{ \Ind{U \leq \frac{k}{n} x, V \leq \frac{k}{n} y} \neq \Ind{U \leq \frac{k}{n} x', V \leq \frac{k}{n} y'} } \\
			&= \frac{n}{m} \Prob{}{U \leq \frac{k}{n} x, V \leq \frac{k}{n} y} + \frac{n}{m} \Prob{}{U \leq \frac{k}{n} x', V \leq \frac{k}{n} y'} \\
			&\phantom{=} - 2 \frac{n}{m} \Prob{}{U \leq \frac{k}{n} (x \wedge x'), V \leq \frac{k}{n} (y \wedge y')} \\
			&\longrightarrow c(x, y) + c(x', y') - 2 c(x \wedge x', y \wedge y') \\
			&=: \rho^2((x, y), (x', y')).
			\end{align*}
			
			Moreover, by \cref{RV1}(ii), this convergence is uniform over $[0, 2T]^4$. This means that for any sequences $x_n, y_n, x_n', y_n'$ in $[0, 2T]$ such that $\rho((x_n, y_n), (x_n', y_n')) \to 0$, $\rho_n((x_n, y_n), (x_n', y_n'))$ is equal to
			\begin{align*}
			&\{ \rho_n((x_n, y_n), (x_n', y_n')) - \rho((x_n, y_n), (x_n', y_n')) \} + \rho((x_n, y_n), (x_n', y_n')) \\
			&\quad \leq \sup_{(x, y, x', y') \in [0, 2T]^4} | \rho_n((x, y), (x', y')) - \rho((x, y), (x', y')) | \\
			&\quad\quad + \rho((x_n, y_n), (x_n', y_n')) \\
			&\quad \longrightarrow 0.
			\end{align*}
			
		\end{itemize}

		Finally, the theorem implies that $H_n \wc W$ in $\ell^\infty([0, 2T]^2)$.

		\end{proof}

		\begin{lemm} \label{nu}

		Let $Q$ be a bivariate copula. If there exists a positive function $q$ and a finite function $c$ that is not everywhere 0 such that for every $(x, y) \in [0, \infty)^2$, as $n \to \infty$,
		$$
		\frac{Q(x/n, y/n)}{q(1/n)} \too c(x, y),
		$$
		then there exists a measure $\nu$ such that for every $(x, y) \in [0, \infty)^2$, $c(x, y) = \nu((0, x] \times (0, y])$. Note that \cref{Q} satisfies this setting.

		\end{lemm}

		\begin{proof}

		Define the measures $\nu_n$ by
		$$
		\nu_n((0, x] \times (0, y]) = \frac{Q(x/n, y/n)}{q(1/n)}
		$$
		and fix $a \in (0, \infty)$. Note that since $c$ is not everywhere 0, $c(a, a)$ is eventually positive, so for $n$ and $a$ large enough, $\nu_n((0, a]^2) > 0$. Then clearly
		$$
		P_{n, a} := \left( \nu_n((0, a]^2) \right)^{-1} \nu_n
		$$
		is a probability measure on $[0, a]^2$. Since it is supported on the same compact set for every $n$, the sequence $\{P_{n, a}: n \in \N\}$ is tight. Thus, by Helly's selection theorem there exists a probability measure $P_a$ also supported on $[0, a]^2$ and a subsequence $\{n_j: j \in \N\}$ such that $P_{n_j, a} \wc P_a$. However, by definition of $\nu_n$, we have for every $(x, y) \in [0, a]^2$
		$$
		P_{n_j, a}((0, x] \times (0, y]) \too \frac{c(x, y)}{c(a, a)}.
		$$

		Therefore, we must have $P_a((0, x] \times (0, y]) = c(x, y)/c(a, a)$, so choosing $\nu_a = c(a, a) P_a$, the result holds for every $(x, y) \in [0, a]^2$. However, the value of $\nu_a((0, x] \times (0, y])$ is independent of $a$ (as long as $x \vee y \leq a$), so $\nu_a$ can be uniquely extended to a measure $\nu$ on the bounded Borel sets of $[0, \infty)^2$.

		\end{proof}

		\begin{lemm}[similar to Theorem 1 in \cite{RL2009}] \label{int rep}

		Define the function $c$ as in \cref{Q}. Then there exists a finite measure $H$ on $[0, 1]$ such that, for every $(x, y) \in [0, \infty)^2$,
		$$
		c(x, y) = \int_{[0, 1]} \left( \frac{x}{1-w} \wedge \frac{y}{w} \right)^{1/\eta} H(\d w).
		$$

		It is also useful to note that this integral is equal to
		$$
		\int_{\left[ 0, \frac{y}{x+y} \right]} \left( \frac{x}{1-w} \right)^{1/\eta} H(\d w) + \int_{\left( \frac{y}{x+y}, 1 \right]} \left( \frac{y}{w} \right)^{1/\eta} H(\d w).
		$$

		\end{lemm}

		\begin{proof}

		By \cref{nu}, we can write
		\begin{equation} \label{c}
		c(x, y) = \nu((0, x] \times (0, y]) = \int_{[0, \infty)^2} \mathds{1}_{(0, x] \times (0, y]} \d \nu = \int_{[0, \infty)^2 \backslash \{0\}} \mathds{1}_{[0, x] \times [0, y]} \d \nu.
		\end{equation}

		In the last equality, nothing changed since $\nu((0, x] \times \{0\} \cup \{0\} \times (0, y]) \leq c(x, 0) + c(0, y) = 0$. Then, through the mapping $f:[0, \infty)^2 \backslash \{0\} \to (0, \infty) \times [0, 1]$ defined by $f(x, y) = (x+y, \frac{y}{x+y})$, define the push-forward measure $\mu = \nu \circ f^{-1}$. By homogeneity of $\nu$, we see that $\mu$ is a product measure:
		$$
		\mu((0, r] \times (0, w]) = r^{1/\eta} \mu((0, 1] \times (0, w]) =: G((0, r]) H((0, w]),
		$$
		where $G$ is a measure on $(0, \infty)$ and $H$ is a measure on $[0, 1]$. Finally, for any $(x, y)$, define the function $g: (0, \infty) \times [0, 1] \to \R$ as
		$$
		g(r, w) = \Ind{r \leq \frac{x}{1-w} \wedge \frac{y}{w}},
		$$
		so that $g \circ f = \mathds{1}_{[0, x] \times [0, y]}$. Using \cref{c} and Theorem 9.15 from \cite{T1998}, we have
		\begin{align*}
		c(x, y) &= \int_{[0, \infty)^2 \backslash \{0\}} g \circ f \d \nu \\
		&= \int_{(0, \infty) \times [0, 1]} g \d \mu \\
		&= \int_{[0, 1]} \int_{(0, \infty)} \mathds{1}_{\left( 0, \frac{x}{1-w} \wedge \frac{y}{w} \right]} (r) G(\d r) H(\d w) \\
		&= \int_{[0, 1]} \left( \frac{x}{1-w} \wedge \frac{y}{w} \right)^{1/\eta} H(\d w),
		\end{align*}
		where we used Fubini's theorem to write the integral with respect to the product measure $\mu$ as a double integral. Moreover, note that $H$ is finite since
		$$
		H([0, 1]) = \mu((0, 1] \times [0, 1]) = \nu \left(\left\{ (x, y) \in [0, \infty)^2: x+y \leq 1 \right\}\right) \leq c(1, 1) = 1.
		$$

		\end{proof}

		\begin{lemm} \label{bound increments c}

		Define the function $c$ as in \cref{Q}. Then for every $(x, y) \in [0, T]^2$ and $h>0$,
		$$
		c(x+h, y) - c(x, y) \leq \frac{1}{\eta} h \frac{c(x+h, y)}{x+h}.
		$$

		\end{lemm}

		\begin{proof}

		By \cref{int rep}, write
		$$
		c(x, y) = \int_{[0, 1]} \left( \frac{x}{1-w} \wedge \frac{y}{w} \right)^{1/\eta} H(\d w) =: \int_{[0, 1]} f(x, y, w) H(\d w).
		$$

		Clearly, it is sufficient to prove that for every $x, y, h, w$,
		\begin{equation} \label{increments f}
		f(x+h, y, w) - f(x, y, w) \leq \frac{1}{\eta} h \frac{f(x+h, y, w)}{x+h},
		\end{equation}
		because then the result follows by integrating both sides. To prove \cref{increments f}, first note that for any $y, w$,
		$$
		f(x, y, w) = \begin{cases} \left( \frac{x}{1-w} \right)^{1/\eta}&, \quad x \leq \frac{1-w}{w} y \\
		\left( \frac{y}{w} \right)^{1/\eta}&, \quad x \geq \frac{1-w}{w} y
		\end{cases}.
		$$

		As a function of $x$, this is continuously differentiable everywhere on $(0, T]$ except at the change point $x = \frac{1-w}{w} y$ and its derivative with respect to $x$, $f'$, is equal to $f(x, y, h)/(\eta x)$ on the first part and 0 on the second. From here we consider three different cases, depending on the position of the change point with respect to $x$ and $x+h$.

		First, if $x+h \leq \frac{1-w}{w} y$,
		$$
		f(x+h, y, w) - f(x, y, w) = h f'(\xi, y, w) = h \frac{f(\xi, y, w)}{\eta \xi},
		$$
		for some $\xi \in [x, x+h]$, by Taylor's theorem. By monotonicity, this is upper bounded by
		$$
		\frac{1}{\eta} h \frac{f(x+h, y, w)}{x+h}.
		$$

		Next, if $\frac{1-w}{w} y \leq x$, $f(x+h, y, w) - f(x, y, w) = 0$ so the result is trivial.

		Finally, if $x < \frac{1-w}{w} y < x+h$,
		$$
		f(x+h, y, w) - f(x, y, w) = f \left( \frac{1-w}{w} y, y, w \right) - f(x, y, w) = \left( \frac{1-w}{w} y - x \right) \frac{f(\xi, y, w)}{\eta \xi},
		$$
		for $\xi$ between $x$ and $\frac{1-w}{w} y$, once again by Taylor's theorem. By monotonicity, we have
		$$
		\frac{f(\xi, y, w)}{\eta \xi} \leq \frac{1}{\eta \frac{1-w}{w} y} \left( \frac{y}{w} \right)^{1/\eta} = \frac{1}{\eta \frac{1-w}{w} y} f(x+h, y, w).
		$$
		Moreover,
		$$
		\frac{ \frac{1-w}{w} y - x }{ \frac{1-w}{w} y } \leq \frac{(x+h) - x}{(x+h)} = \frac{h}{x+h},
		$$
		because the function $t \mapsto (t-x)/t$ is non-decreasing. Piecing everything together, we have
		$$
		f(x+h, y, w) - f(x, y, w) \leq \frac{1}{\eta} h \frac{f(x+h, y, w)}{x+h}.
		$$

		We have proved that \cref{increments f} holds for every $(x, y) \in [0, T]^2$, $h>0$ and $w \in [0, 1]$.

		\end{proof}

		\begin{lemm} \label{bound_c}

		Define the function $c$ as in \cref{Q} and assume \cref{con-proc}(i). Then there exists a constant $K := K_T < \infty$ such that for every $(x, y) \in [0, T]^2$,
		$$
		c(x, y) \leq \frac{K}{\log(1/x)}.
		$$

		\end{lemm}

		\begin{proof}

		We will prove that as $x \to 0$,
		$$
		c(x, y) \lesssim \frac{1}{\log(1/x)}
		$$
		uniformly for all $y \in [0, T]$. Since $c$ is locally bounded, the result will follow.

		Since \cref{con-proc}(i) is satisfied, we may assume it is satisfied with the function $q_1(t) = 1/\log(1/t)$. Recall that as $t \downarrow 0$, by \cref{RV1},
		$$
		Q(tx, ty) = q(t) c(x, y) + O(q(t) q_1(t))
		$$
		uniformly over all $(x, y) \in [0, T]^2$. That is,
		\begin{equation} \label{bound c}
		c(x, y) = \frac{Q(tx, ty)}{q(t)} + O(q_1(t)) \leq \frac{tx}{q(t)} + O(q_1(t))
		\end{equation}
		uniformly, by Lipschitz continuity of the copula $Q$. The previous relation holds whenever $t \to 0$, and in particular it holds when $t$ and $x$ are related and both tend to 0.

		Define $g(t) = q(t) q_1(t) / t \to 0$ as $t \to 0$. We argue, in the following, that for any $x$ small enough, there exists $t(x) > 0$ such that $x \leq g(t(x)) \leq 2^{1/\eta} x$. Plugging $t(x)$ into \cref{bound c}, we find that as $x \to 0$,
		\begin{equation} \label{bound c 2}
		c(x, y) \leq \frac{t(x) x}{q(t(x))} + O(q_1(t(x))) = O(q_1(t(x))),
		\end{equation}
		because, since we assume $x \leq g(t(x))$,
		$$
		\frac{t(x) x}{q(t(x))} \leq \frac{t(x)}{q(t(x))} g(t(x)) = q_1(t(x)).
		$$

		Moreover, since the function $g$ is $\rho$-RV at 0, $\rho := 1/\eta-1$, for small enough $t$ we have $g(t) \geq t^\alpha$, as long as $\alpha > \rho$. This means that
		$$
		q_1(t(x)) = \frac{1}{\log(1/t(x))} = \frac{\alpha}{\log(1/t(x)^\alpha)} \lesssim \frac{1}{\log(1/g(t(x)))}.
		$$

		Finally, by the assumption that $g(t(x)) \leq 2^{1/\eta} x$, we get
		$$
		q_1(t(x)) \lesssim \frac{1}{\log(1/g(t(x)))} \lesssim \frac{1}{\log(1/x)}
		$$
		which, in conjunction with \cref{bound c 2}, yields the desired bound for $c(x, y)$ as $x \to 0$, uniformly over bounded $y$.

		The only thing left is to prove the existence of a point $t(x)$ such that $g(t(x)) \in [x, 2^{1/\eta} x]$ for every small enough $x$. This can be done by using the fact that the function $g$ is $\rho$-RV at 0. Applying Theorem 1.5.6(iii) in \cite{BGT1987} (adapted to functions of regular variation at 0) with any $\delta \in (0, 1)$ and $A = 2^{1-\delta}$, we find that there exists $T_0 \in (0, \infty)$ such that for every $t \leq T_0$,
		$$
		\frac{g(t)}{g(t/2)} \leq 2^{1-\delta} 2^{\rho + \delta} = 2^{1/\eta}.
		$$

		We now construct a non-increasing sequence the follwing way: take $t_0 = T_0$ and for $n \in \N$, define $t_n = t_{n-1}/2$ if $g(t_{n-1}/2) \leq g(t_{n-1})$. Otherwise, $t_n = t_{n-1}/4$ if $g(t_{n-1}/4) \leq g(t_{n-1})$. Otherwise, we try $t_{n-1}/8$, etc. In general
		$$
		t_n = \max \left\{ \frac{t_{n-1}}{2^k} : k \in \N, g \left( \frac{t_{n-1}}{2^k} \right) \leq g(t_{n-1}) \right\}.
		$$

		Therefore, the sequence satisfies, for every natural $n$,
		\begin{equation} \label{ratio g}
		1 \leq \frac{g(t_n)}{g(t_{n+1})} \leq 2^{1/\eta}.
		\end{equation}

		Now choose any $x \in (0, T_0/2]$ and let $t = \min_{n \in \N} \{t_n : g(t_n) \geq x\}$. Clearly, $g(t) \geq x$, and $g(t)$ has to be $\leq 2^{1/\eta} x$. Indeed, suppose the opposite. Then by \cref{ratio g}, $g(t_{n+1}) \geq g(t)/2^{1/\eta} > x$, which contradicts the definition of $t$. We conclude that for every $x \in (0, T_0/2]$, the desired $t(x)$ exists.

		\end{proof}

		\begin{lemm} \label{lemma:Hadamard}

		Assume the setting of \cref{big lemma AD}. For arbitrary positive $t$ and $T$, let
		\[
		\cV(t) := \{b \in \ell^\infty([0, 2T]) : \forall x \in [0, T], x + tb(x) \in [0, 2T]\}.
		\]

		Let $t_n \downarrow 0$ and assume that $b_n := (b_n^{(1)}, b_n^{(2)}) \in \cV(t_n)^2$ converges uniformly to a continuous function $b = (b^{(1)}, b^{(2)})$ such that $b^{(1)}(0) = b^{(2)}(0) = 0$. Then, the functions $g_n:[0, T] \to \R$ defined by
		\[
		g_n(x, y) := \frac{c\left( x + t_n b_n^{(1)}(x), y + t_n b_n^{(2)}(y) \right) - c(x, y)}{t_n}
		\]
		hypi-converge to $\dot c_1(x, y) b^{(1)}(x) + \dot c_2(x, y) b^{(2)}(y)$, where $\dot c_1$ and $\dot c_2$ are defined as in \cref{sec:asy-non}.

		\end{lemm}

		\begin{proof}

		Let $\ell$ be the stable tail dependence function associated to the random vector $(X, Y)$. Because we assume asymptotic dependence, we know that $\chi := \lim_{t \downarrow 0} q(t)/t > 0$ and that $c(x, y) = (x+y - \ell(x, y)) / \chi$. Then,
		\[
		g_n(x, y) = \chi^{-1} \left( b_n^{(1)}(x) + b_n^{(2)}(y) - \frac{\ell\left( x + t_n b_n^{(1)}(x), y + t_n b_n^{(2)}(y) \right) - \ell(x, y)}{t_n} \right).
		\]

		By assumption, the sum of the first two terms converges uniformly to $b^{(1)}(x) + b^{(2)}(y)$. Let $\cS \subset [0, \infty)^2$ be the set of points where $\ell$ is differentiable. Since $\ell$ is convex, the complement of $\cS$ is Lebesgue-null and the gradient of $\ell$ is continuous on $\cS$ \citep[Theorem 25.5]{R1970}. By Lemma F.3 of \cite{BSV2014}, the last term hypi-converges to
		\[
		\cL_1(x, y) := \sup_{\eps>0} \inf \left\{ \dot\ell_1(x', y') b^{(1)}(x') + \dot\ell_2(x', y') b^{(2)}(y') : (x', y') \in \cS, \|(x, y) - (x', y')\| < \eps \right\},
		\]
		where $\dot\ell_j$ are defined like $\dot c_j$: $\dot\ell_1(x, y)$ is the first partial derivative at $(x, y)$ from the left, except if $x=0$ in which case it is from the right, and $\dot\ell_2$ is always the second partial derivative from the right. We argue below that the hypi-distance between the functions $\cL_1$ and $\cL_2$, defined by $\cL_2(x, y) = \dot\ell_1(x, y) b^{(1)}(x) + \dot\ell_2(x, y) b^{(2)}(y)$, is 0. That is, $\cL_1$ and $\cL_2$ belong to the same equivalence class in the space $L^\infty([0, 2T]^2)$ and hypi-convergence to $\cL_1$ is equivalent to hypi-convergence to $\cL_2$. It follows that $g_n(x, y)$ hypi-converges to
		\begin{equation} \label{derivative}
		\frac{b^{(1)}(x) + b^{(2)}(y) - \cL_2(x, y)}{\chi} = \dot c_1(x, y) b^{(1)}(x) + \dot c_2(x, y) b^{(2)}(y),
		\end{equation}
		where the last equality is a consequence of the relation $\dot\ell_j(x, y) = 1 - \chi \dot c_j(x, y)$, $j \in \{1, 2\}$.

		To prove the equivalence between $\cL_1$ and $\cL_2$, first note that by continuity of $b^{(1)}$ and $b^{(2)}$,
		\[
		\cL_1(x, y) := \sup_{\eps>0} \inf \left\{ \dot\ell_1(x', y') b^{(1)}(x) + \dot\ell_2(x', y') b^{(2)}(y) : (x', y') \in \cS, \|(x, y) - (x', y')\| < \eps \right\}.
		\]
		Let $\dot\ell_j^-$ and $\dot\ell_j^+$ denote the directional partial derivatives of $\ell$ from the left and right, respectively. The function $\cL_2$ can then be expressed the following way, and we analogously define $\cL_3$:
		\[
		\cL_2(x, y) = \dot\ell_1^-(x, y) b^{(1)}(x) + \dot\ell_2^+(x, y) b^{(2)}(y), \quad \cL_3(x, y) := \dot\ell_1^+(x, y) b^{(1)}(x) + \dot\ell_2^-(x, y) b^{(2)}(y).
		\]
		The main tool is the homogeneity property of $\ell$ ($\ell(ax, ay) = a\ell(x, y)$, $a\geq0$). It implies that the directional derivatives $\dot\ell_j^\pm$ are constant along rays of the form $\{az: a>0\}$, $z \in (0, \infty)^2$ and therefore that $\cS$ consists exactly of a dense union of such rays.

		Fix a point $(x, y) \in (0, \infty)^2$. For any sufficiently small $\eps>0$, the open $\eps$-ball $B(\eps)$ around $(x, y)$ can be partitioned into  {the two open ``half-balls"
		\[
		B_1(\eps) := \{(x', y') \in B(\eps): y'/x' > y/x\}, \quad B_2(\eps) := \{(x', y') \in B(\eps): y'/x' < y/x\}
		\]
		and the line $B_3(\eps) := \{(x', y') \in B(\eps): y'/x' = y/x\}$.
		}
		Provided that $\eps$ is sufficiently small, there exists a positive $\delta = \delta(\eps)$ such that $\delta(\eps) \to 0$ as $\eps \to 0$, such that each point in $B_1(\eps)$ is on the same ray as some $u \in (x-\delta, x] \times \{y\}$ and some $v \in \{x\} \times [y, y+\delta)$ and such that each point in $B_2(\eps)$ is on the same ray as some $u \in (x, x+\delta) \times \{y\}$ and some $v \in \{x\} \times (y-\delta, y)$. By \cite{R1970}, Theorem 24.1, we have
		\begin{align*}
		\lim_{\delta \downarrow 0} \dot\ell_1^\pm(x-\delta, y) = \dot\ell_1^-(x, y), &\quad \lim_{\delta \downarrow 0} \dot\ell_1^\pm(x+\delta, y) = \dot\ell_1^+(x, y), \\
		\lim_{\delta \downarrow 0} \dot\ell_2^\pm(x, y-\delta) = \dot\ell_2^-(x, y), &\quad \lim_{\delta \downarrow 0} \dot\ell_2^\pm(x, y+\delta) = \dot\ell_2^+(x, y).
		\end{align*}
		Then, as $\eps \to 0$, the vectors  {$(\dot\ell_1^\pm(x', y'), \dot\ell_2^\pm(x', y'))$} converge to $(\dot\ell_1^-(x, y), \dot\ell_2^+(x, y))$ for $(x', y') \in B_1(\eps)$ and to $(\dot\ell_1^+(x, y), \dot\ell_2^-(x, y))$ for $(x', y') \in B_2(\eps)$. It follows by continuity of $b$  { that for any sufficiently small $\eps >0$,
		\begin{align}	
		\lim_{(x',y')\to (x,y), (x',y')\in B_1(\eps)} \cL_2(x',y') 
		= \lim_{(x',y')\to (x,y), (x',y')\in B_1(\eps)} \cL_3(x',y') 
		= \cL_2(x,y) \label{eq:convL2L3B1}
		\\
		\lim_{(x',y')\to (x,y), (x',y')\in B_2(\eps)} \cL_2(x',y') 
		= \lim_{(x',y')\to (x,y), (x',y')\in B_2(\eps)} \cL_3(x',y') 
		= \cL_3(x,y) \label{eq:convL2L3B2}
		\end{align}
		In particular, since $\dot \ell_j^\pm$ are constant on $B_3(\eps)$,} the semicontinuous hulls of $\cL_2$ are
		\begin{align*}
		\cL_{2, \wedge}(x, y) &:= \sup_{\eps>0} \inf \left\{ \cL_2(x', y'): (x', y') \in B(\eps) \right\} = \cL_2(x, y) \wedge \cL_3(x, y), \\
		\cL_{2, \vee}(x, y) &:= \inf_{\eps>0} \sup \left\{ \cL_2(x', y'): (x', y') \in B(\eps) \right\} = \cL_2(x, y) \vee \cL_3(x, y),
		\end{align*}
		and since $B_1(\eps) \cap \cS$ and $B_2(\eps) \cap \cS$ are always nonempty, the preceding relations also hold if $B(\eps)$ is intersected with $\cS$, whence
		\[
		\cL_1(x, y) = \sup_{\eps>0} \inf \left\{ \cL_2(x', y'): (x', y') \in B(\eps) \cap \cS \right\} = \cL_{2, \wedge}(x, y).
		\]
		One easily argues that $\cL_1$ is lower semicontinuous, i.e. its lower semicontinuous hull is equal to $\cL_1$ itself, which is also equal to the lower semicontinuous hull of $\cL_2$.

		Next observe that
		\begin{align*}
		\cL_{1, \vee}(x, y) &= \inf_{\eps>0} \sup \left\{ \cL_2(x', y') \wedge \cL_3(x', y'): (x', y') \in B(\eps) \right\}
		\\
		&= \cL_2(x, y) \vee \cL_3(x, y) = \cL_{2, \vee}(x, y).
		\end{align*}
		 {where the first equality follows from the definition of $\cL_{1, \vee}$, the fact that $\cL_1 = \cL_{2,\wedge}$ as shown earlier and the representation for $\cL_{2,\wedge}$ derived above while the second equality follows from \cref{eq:convL2L3B1,eq:convL2L3B2}. }

		The previous argument assumes $(x, y) \in (0, \infty)^2$. It remains to show that the semicontinuous hulls of $\cL_1$ also correspond to those of $\cL_2$ on the axes. For this, assume now that $x>0, y=0$. The ball $B(\eps)$ around $(x, 0)$ now becomes a ``half-ball" (we intersect if with $[0, \infty)^2$). Let $(x', y')$ be a point in that ball. Then $(x', y')$ is on the same ray as $(x, \delta)$, for some $\delta \geq 0$ that can be made to converge to 0 as $\eps \to 0$. We have $\dot\ell_2^{\pm}(x', y') = \dot\ell_2^{\pm}(x, \delta) \to \dot\ell_2^+(x, 0)$ as $\eps \to 0$. For the first derivative, the known bounds $x \vee y \leq \ell(x, y) \leq x+y$ imply that $x \leq \ell(x, \delta) \leq x + \delta$. The convexity and homogeneity properties then imply that
		\[
		\dot\ell_1(x, 0) = 1 \geq \dot\ell_1^\pm(x, \delta) \geq \frac{\ell(x, \delta) - \ell(0, \delta)}{x} \geq \frac{x-\delta}{x} \too 1
		\]
		as $\eps \to 0$. By uniform boundedness of $\dot \ell_1^\pm, \dot \ell_2^\pm$ it follows easily that $\cL_1$ and $\cL_2$ are continuous at $(x, 0)$ and that $\cL_1(x, 0) = \cL_2(x, 0) = b^{(1)}(x)$, whence those two functions have the same semicontinuous hulls at that point.

		Because $\dot\ell_1(0, y)$ was defined as the partial derivative from the right, one deals with a point $(0, y)$ in the same way. 

		Finally, note that since $b^{(1)}(0) = b^{(2)}(0) = 0$, and by uniform boundedness of $\dot \ell_1^\pm, \dot \ell_2^\pm$ the functions $\cL_1$ and $\cL_2$ are both continuous and take the value 0 at $(0, 0)$. Their semicontinuous hulls are therefore also equal at that point.

		We have shown that everywhere on $[0, \infty)^2$, $\cL_{1, \wedge} = \cL_{2, \wedge}$ and $\cL_{1, \vee} = \cL_{2, \vee}$. By definition \citep[see][Proposition 2.1]{BSV2014}, this means that $d_\text{hypi}(\cL_1, \cL_2) = 0$.

		\end{proof}

		\begin{lemm}\label{prop hypi}
		Let $f: [0, T]^2 \to \R$ be continuous  {Lebesgue-almost everywhere}, $g := (g_1, \dots, g_q)^\top: [0, T]^2 \to \R^q$ be a vector of integrable functions  {and assume that $f_n$ are measurable and hypi-converge to $f$ on $[0, T]^2$. Then $\int gf_n \d\leb \to \int gf \d\leb$}, where $\leb$ denotes the Lebesgue measure on $[0, T]^2$.

		\end{lemm}

		\begin{proof}

		For every $j \in \{1, \dots, q\}$ and $M < \infty$, we have
		\begin{align*}
		\int |g_j f_n - g_j f| \d\leb &= \int |g_j| |f_n - f| \Ind{|g_j| \leq M} \d\leb + \int |g_j| |f_n - f| \Ind{|g_j| > M} \d\leb \\
		&\leq M \int |f_n - f| \d\leb + \sup_{(x, y) \in [0, T]^2} |f_n(x, y) - f(x, y)| \int |g_j| \Ind{|g_j| > M} \d\leb \\
		&\leq M \int |f_n - f| \d\leb \\
		&\phantom{=} + \left( \sup_{(x, y) \in [0, T]^2} |f_n(x, y)| + \sup_{(x, y) \in [0, T]^2} |f(x, y)| \right) \int |g_j| \Ind{|g_j| > M} \d\leb.
		\end{align*}

		The first term on the right hand side converges to 0 by Proposition 2.4 of \cite{BSV2014} and since $f$ is assumed continuous almost everywhere. By Proposition  {2.3} of that paper, $\sup_{(x, y) \in [0, T]^2} |f_n(x, y)| \to \sup_{(x, y) \in [0, T]^2} |f(x, y)|$. Therefore, we have
		\[
		\lim_{n \to \infty} \int |g_j f_n - g_j f| \d\leb \leq 2 \sup_{(x, y) \in [0, T]^2} |f(x, y)| \int |g_j| \Ind{|g_j| > M} \d\leb,
		\]
		which can be made arbitrarily small by choosing $M$ large enough, since $g_j$ is integrable. The claim follows.

		\end{proof}


		\begin{lemm} \label{M}

		Let $\phi: \R^p \to \R^q$, $p \leq q$, have a unique, well separated zero at a point $x_0 \in \R^p$ and be continuously differentiable at $x_0$ with Jacobian matrix $J := J_\phi(x_0)$ of full rank $p$. Let $Y_n$ be a random vector in $\R^q$ such that $a_n^{-1} Y_n$ weakly converges to a random vector $Y$, for some sequence $a_n \to 0$. Then if $X_n = \argmin_x \| \phi(x) - Y_n\|$, we have
		\[
		X_n - x_0 = (J^\top J)^{-1} J^\top Y_n + \op{a_n}.
		\]

		\end{lemm}

		\begin{proof}

		Let $h_n := a_n^{-1} (X_n - x_0 - (J^\top J)^{-1} J^\top Y_n)$. By definition of $X_n$, $h_n$ is a minimizer of the random function $M_n: \R^p \to \R_+$ defined as
		\[
		M_n(h) := a_n^{-1} \Big\| \phi \big( x_0 + (J^\top J)^{-1} J^\top Y_n + a_n h \big) - Y_n \Big\|.
		\]
		By differentiability of $\phi$, $M_n(h)$ is the norm of
		\[
		\big( J (J^\top J)^{-1} J^\top - I \big) a_n^{-1} Y_n + Jh + o(1)
		\]
		uniformly over bounded $h$, where $I$ is the $q \times q$ identity matrix. The above display, seen as a function of $h$, weakly converges to
		\[
		\big( J (J^\top J)^{-1} J^\top - I \big) Y + Jh
		\]
		in $(\ell^\infty(\cK))^q$, for any compact set $\cK$. The mapping $f \mapsto \{h \mapsto \|f(h)\|\}$ being continuous from $(\ell^\infty(\cK))^q$ onto $\ell^\infty(\cK)$, it follows that $M_n \wc M$ in $\ell^\infty(\cK)$, for
		\[
		M(h) := \Big\| \big( J (J^\top J)^{-1} J^\top - I \big) Y + Jh \Big\|.
		\]

		The function $M^2$ is strictly convex and has derivative $\partial(M^2(h))/\partial h = 2 J^\top J h$ which, since $J$ has full rank, has a unique zero at $h=0$. It follows that $M^2$, and thus $M$, has a unique minimizer at the point $0$. Therefore, if we can show that the sequence $\{h_n\}$ is uniformly tight, Corollary 5.58 of \cite{V2000} will ensure that $h_n$ converges in distribution (and hence in probability) to 0, which in turn implies the result.

		It is known by Prohorov's theorem that $\{a_n^{-1} Y_n\}$ is uniformly tight. Therefore, it is sufficient to establish tightness of $\{a_n^{-1} (X_n - x_0)\}$. First, define for $\delta>0$
		\[
		\eps(\delta) = \inf_{x \notin B(x_0, \delta)} \|\phi(x)\|,
		\]
		where $B(x_0, \delta)$ denotes an open $\delta$-ball around $x_0$. By assumption, $\eps(\delta)>0$ for every positive $\delta$. Choose $\delta_0 > 0$ small enough so that for every $x \in B(x_0, \delta_0)$,
		\[
		\|\phi(x) - J(x - x_0)\| < \frac{1}{2} \|J(x - x_0)\|,
		\]
		which is possible by differentiability of $\phi$ (recall that $J$ is the Jacobian at $x_0$). By the reverse triangle inequality, this implies that $|\|\phi(x)\| - \|J(x - x_0)\||$ has the same upper bound. Then, for $\delta \leq \delta_0$,
		\[
		\eps(\delta) > \frac{1}{2} \inf_{x \in B(x_0, \delta)} \|J(x - x_0)\| = \frac{\sigma_1(J)}{2} \delta,
		\]
		where $\sigma_1(J)$, the smallest singular value of $J$, is positive since $J$ has full rank.

		Now, fix an arbitrary $\eta>0$. Because the sequence $\{a_n^{-1} Y_n\}$ is uniformly tight, there exists a finite $K = K(\eta)$ such that for $\delta_n := K a_n$ and for $n$ large enough so that $\delta_n \leq \delta_0$,
		\[
		\Prob{}{\|Y_n\| \geq \frac{\eps(\delta_n)}{2}} \leq \Prob{}{\|Y_n\| \geq \frac{K\sigma_1(J)}{4} a_n} \leq \eta
		\]
		Hence with probability at least $1-\eta$, $\|Y_n\| < \eps(\delta_n)/2$. The last inequality implies two things. First, letting $\phi_n = \phi - Y_n$ and recalling that $\phi(x_0) = 0$, we have $\|\phi_n(x_0)\| = \|Y_n\| < \eps(\delta_n)/2$. Second, for any $x \notin B(x_0, \delta_n)$, we have $\|\phi(x)\| \geq \eps(\delta_n)$ so
		\[
		\|\phi_n(x)\| = \|\phi(x) - Y_n\| \geq |\|\phi(x)\| - \|Y_n\|| > \frac{\eps(\delta_n)}{2}.
		\]
		That is, with probability at least $1-\eta$, $X_n = \argmin_x \|\phi_n(x)\| \in B(x_0, \delta_n)$. Since $\delta_n = O(a_n)$ and $\eta$ was arbitrary, we conclude that $\{a_n^{-1} (X_n - x_0)\}$ is uniformly tight, and so is $\{h_n\}$.

		\end{proof}

		\section{Proof of the claims in \cref{ex:ims:cont,ex_spatial:ims,ex_spatial:mix}}
		\label{proof IMS}

		\subsection{\cref{ex:ims:cont}}

		Recall that the random vector $Z := (1-X, 1-Y)$ is assumed max-stable with uniform margin and stable tail dependence function $\ell$, hence its distribution function is given by \cref{max_stable}. Let $(x, y) \in (0, 1]^2$ (the result is trivial if $x$ or $y$ is zero).  {Note that we can without loss of generality focus on $(x, y) \in (0, 1]^2$ instead of general bounded sets since any bounded set can be rescaled to be contained in $[0,1]^2$ at the cost of absorbing the scaling into $t$.} The survival copula $Q$ of $(X, Y)$ satisfies
		\begin{align*}
		Q(tx, ty) &:= \Prob{}{X \geq 1 - tx, Y \geq 1 - ty} \\
		&= \Prob{}{1-X \leq tx, 1-Y \leq ty} \\
		&= \exp\{ -\ell(-\log (tx), -\log (ty)) \} \\
		&= \exp\left\{ \log(t) \ell \left( 1 + \frac{\log(x)}{\log(t)}, 1 + \frac{\log(y)}{\log(t)} \right) \right\},
		\end{align*}
		where we have used the homogeneity property of $\ell$ in the last line. By the assumed expansion of the function $\ell$,
		$$
		\ell \left( 1 + \frac{\log(x)}{\log(t)}, 1 + \frac{\log(y)}{\log(t)} \right) = \ell(1, 1) + \dot\ell_1(1, 1) \frac{\log(x)}{\log(t)} + \dot\ell_2(1, 1) \frac{\log(y)}{\log(t)} + \delta(t, x, y),
		$$
		where $\dot\ell_1$ and $\dot\ell_2$ are the right partial derivatives of $\ell$ with respect to its first and second argument, respectively, and
		$$
		\delta(t, x, y) \lesssim \left( \frac{\log(x)}{\log(t)} \right)^2 + \left( \frac{\log(y)}{\log(t)} \right)^2.
		$$

		This is a linear approximation of the function $\ell$; since that function is convex, it lies above its sub gradient, so the error term $\delta(t, x, y)$ is non-negative. Plugging this in our expression for $Q(tx, ty)$ yields
		$$
		Q(tx, ty) = t^{\ell(1, 1)} x^{\dot\ell_1(1, 1)} y^{\dot\ell_2(1, 1)} e^{\delta'(t, x, y)},
		$$
		where $\delta'(t, x, y) = \log(t) \delta(t, x, y)$ satisfies
		$$
		\frac{\log(x)^2 + \log(y)^2}{\log(t)} \lesssim \delta'(t, x, y) \leq 0.
		$$
		Letting $q(t) = t^{\ell(1, 1)}$ and $c(x, y) = x^{\dot\ell_1(1, 1)} y^{\dot\ell_2(1, 1)}$, we obtain
		\begin{align*}
		\left| \frac{Q(tx, ty)}{q(t)} - c(x, y) \right| &= x^{\dot\ell_1(1, 1)} y^{\dot\ell_2(1, 1)} \left( 1 - e^{\delta'(t, x, y)} \right) \\
		&\leq x^{\dot\ell_1(1, 1)} y^{\dot\ell_2(1, 1)} |\delta'(t, x, y)|
		\\
		&\lesssim \frac{ x^{\dot\ell_1(1, 1)} y^{\dot\ell_2(1, 1)} (\log(x)^2 + \log(y)^2)}{\log(1/t)},
		\end{align*}
		where we used the fact that $0 \leq 1 - e^x \leq |x|$ for all $x \leq 0$. Since $\dot\ell_1(1, 1)$ and $\dot\ell_2(1, 1)$ are positive it follows that this upper bound is of order $1/\log(1/t)$ uniformly over $x,y$ in bounded sets. The claim in \cref{ex:ims:cont} is proved.
		\hfill $\square$

		\subsection{\cref{ex_spatial:ims}}

		Now, recall the setting of \cref{ex_spatial:ims}. The expression for $\Gamma^{(s, s)}$ is trivial. We shall treat the case where $s$ and $s'$ are two pairs that share an element, i.e. $s = (s_1, s_2)$ and $s' = (s_1, s_3)$. One similarly deals with different combinations of $s, s'$, including the case where they are disjoint.

		Let $\ell$ be the stable tail dependence function of the max-stable, trivariate random vector $(1-X^{(s_1)}, 1-X^{(s_2)}, 1-X^{(s_3)})$. By assumption and by the calculations above for the bivariate case, the pairs $(X^{(s_1)}, X^{(s_2)})$ and $(X^{(s_1)}, X^{(s_3)})$ satisfy \cref{con-proc}(i) with scaling functions $q^{(s)}(t) = t^{\ell(1, 1, 0)}$ and $q^{(s')}(t) = t^{\ell(1, 0, 1)}$, respectively. Since those functions are invertible, we may choose any diverging sequence $m = o(\log(n)^2)$ and invert them, setting $k^{(s)}/n = (m/n)^{1/\ell(1, 1, 0)}$ and $k^{(s')}/n = (m/n)^{1/\ell(1, 0, 1)}$. In fact, we may do so with every pair and obtain, as claimed, a universal sequence $m$.

		Without loss of generality, let us assume that $\ell(1, 1, 0) \leq \ell(1, 0, 1)$ so that $k^{(s)} \leq k^{(s')}$. Let $t_n = k^{(s)}/n$ and $\alpha = \ell(1, 1, 0)/\ell(1, 0, 1) \in (0, 1]$; observe that $k^{(s')}/n = t_n^\alpha$. By definition, for fixed $x^1, x^2 \in (0, 1]^2$ (we can restrict our attention to this setting by similar arguments as in the bivariate case), we have 
		\begin{equation} \label{eq:limit_Gamma}
		\Gamma^{(s, s')}(x^1, x^2) = \lim_{n \to \infty} \frac{n}{m} \Prob{}{ 1 - X^{(s_1)} \leq t_n x, 1 - X^{(s_2)} \leq t_n y, 1 - X^{(s_3)} \leq t_n^\alpha z },
		\end{equation}
		where $x$ is equal to $x_1^1 \wedge x_2^1$ if $\alpha=1$ and to $x_1^1$ otherwise, $y = x_2^1$ and $z = x_2^2$. Using the same reasoning as in the bivariate case above (including the homogeneity property of $\ell$), the probability in \cref{eq:limit_Gamma} can be written as
		\begin{align*}
		&\exp\left\{ -\ell(-\log(t_n x), -\log(t_n y), -\log(t_n^\alpha z)) \right\} \\
		&\quad = \exp\left\{ \log(t_n) \ell\left( 1 + \frac{\log(x)}{\log(t_n)}, 1 + \frac{\log(y)}{\log(t_n)}, \alpha + \frac{\log(z)}{\log(t_n)} \right) \right\} \\
		&\quad = t_n^{\ell(1, 1, \alpha)} \exp\left\{ \log(t_n) \left[ \ell\left( 1 + \frac{\log(x)}{\log(t_n)}, 1 + \frac{\log(y)}{\log(t_n)}, \alpha + \frac{\log(z)}{\log(t_n)} \right) - \ell(1, 1, \alpha) \right] \right\}.
		\end{align*}

		Eventually, $\log(t_n)$ is negative, which makes the difference in the square brackets non-negative by monotonicity of $\ell$. This eventually upper bounds the exponential by 1 and the entire expression by $t_n^{\ell(1, 1, \alpha)}$, for any $x, y, z \in (0,1]$. Considering \cref{eq:limit_Gamma}, it follows that for every fixed $x^1, x^2 \in (0, 1]^2$,
		\[
		\Gamma^{(s, s')}(x^1, x^2) \leq \lim_{n \to \infty} \frac{n}{m} t_n^{\ell(1, 1, \alpha)} = \lim_{n \to \infty} \left( \frac{m}{n} \right)^{\frac{\ell(1, 1, \alpha)}{\ell(1, 1, 0)} - 1} = 0,
		\]
		since the assumption that $\ell$ is component-wise strictly increasing means that $\ell(1, 1, \alpha) > \ell(1, 1, 0)$.
		\hfill $\square$

		\subsection{\cref{ex_spatial:mix}}

		We present here the main ideas, as most of the precise calculations are similar to the preceding section. As before, let $X^{(j)} = Y(u_j)$, and write $Z^{(j)}$ and $Z'^{(j)}$ for $Z(u_j)$ and $Z'(u_j)$. Consider a pair $s := (s_1, s_2)$ and let $F$ be the distribution function of the unit Fr\'echet distribution. Recall that $X^{(j)} = \max\{aZ^{(j)}, (1-a)Z'^{(j)}\}$. We have for $t \downarrow 0$
		\begin{align}
			&\Prob{}{F(X^{(s_1)}) \geq 1 - tx, F(X^{(s_2)}) \geq 1 - ty } \notag
			\\
			&= \Prob{}{ F(Z^{(s_1)})^{1/a} \vee F(Z'^{(s_1)})^{1/(1-a)} \geq 1-tx, F(Z^{(s_2)})^{1/a} \vee F(Z'^{(s_2)})^{1/(1-a)} \geq 1-ty } \notag
			\\
			&= \Prob{}{F(Z^{(s_1)}) \geq (1-tx)^a, F(Z^{(s_2)}) \geq (1-ty)^a} \notag
			\\
			&\quad\quad + \Prob{}{F(Z'^{(s_1)}) \geq (1-tx)^{1-a}, F(Z'^{(s_2)}) \geq (1-ty)^{1-a}} + O(t^2), \label{eq:two probs}
		\end{align}
		where the term $O(t^2)$ is uniform over bounded $x, y$. Note that $(1 -tx)^a = 1 - t(ax + O(tx^2))$. The first term of \cref{eq:two probs} is equal to
		\[
		a \chi^{Z, (s)} t (x + y - \ell^{Z, (s)}(x, y)) + O(t^2)
		\]
		uniformly over bounded $x, y$, where $\chi^{Z, (s)}$ and $\ell^{Z, (s)}$ are the extremal dependence coefficient and stable tail dependence function, respectively, corresponding to the random vector $(Z^{(s_1)}, Z^{(s_2)})$. From previous calculations, the second term of \cref{eq:two probs} is equal to
		\[
		((1-a) t)^{\ell^{Z', (s)}(1, 1)} x^{\dot \ell_1^{Z', (s)}(1, 1)} y^{\dot \ell_2^{Z', (s)}(1, 1)} + O\Big( t^{\ell^{Z', (s)}(1, 1)} / \log(1/t) \Big),
		\]
		uniformly over bounded $x, y$, where $\ell^{Z', (s)}$ is the stable tail dependence function corresponding to the max-stable random vector $(1/Z'^{(s_1)}, 1/Z'^{(s_2)})$. It follows that \cref{con-proc}(i) is satisfied for every pair of locations; depending on whether $(Z^{(s_1)}, Z^{(s_2)})$ is dependent or independent, either the first of the second of the last two expressions dominates. This determines that $q^{(s)}(t)$ is proportional to $t$ for asymptotically dependent pairs and to $t^{1/\eta'^{(s)}}$ for asymptotically independent ones, where $\eta'^{(s)}$ is the coefficient of tail dependence of $(1/Z'^{(s_1)}, 1/Z'^{(s_2)})$, satisfying $1 < 1/\eta'^{(s)} < 2$ by assumption --- for any inverted max-stable distribution, its coefficient of tail dependence $\eta$ is always in $[1/2, 1)$, and can only be equal to $1/2$ under perfect independence. The coefficient of tail dependence $\eta^{(s)}$ of $(X^{(s_1)}, X^{(s_2)})$ is equal to 1 if $\chi^{Z, (s)} > 0$ and to $\eta'^{(s)}$ otherwise.

		We now show how to obtain an expression for the functions $\Gamma^{(s, s')}$. First, since the functions $q^{(s)}$ are proportional to simple powers, for a sufficiently slow intermediate sequence $m$, we let $k^{(s)}/n$ be proportional to $m/n$ if $s$ is an asymptotically dependent pair and to $(m/n)^{\eta^{(s)}}$ otherwise, so that all $m^{(s)}$ are equal to $m$.

		The case $s = s'$ follows trivially from the previous developments; $\Gamma^{(s, s)}$ can be derived from $c^{(s)}$. Next consider the case where $s,s'$ share one element, i.e. $s = (s_1, s_2)$ and $s' = (s_1, s_3)$. Letting $t_n = k^{(s)}/n$ and $t_n' = k^{(s')}/n$, assume without loss of generality that $t_n' \lesssim t_n$. The probability of interest is of the form
		\begin{align*}
			&\Prob{}{ F(X^{(s_1)}) \geq 1 - (t_n x \wedge t_n' x'), F(X^{(s_2)}) \geq 1 - t_n y, F(X^{(s_3)}) \geq 1 - t_n' z }
			\\
			&= \Prob{}{ F(Z^{(s_1)}) \geq (1 - (t_n x \wedge t_n' x'))^a, F(Z^{(s_2)}) \geq (1 - t_n y)^a, F(Z^{(s_3)}) \geq (1 - t_n' z)^a }
			\\
			&\quad + \PP \left( F(Z'^{(s_1)}) \geq (1 - (t_n x \wedge t_n' x'))^{1-a}, F(Z'^{(s_2)}) \geq (1 - t_n y)^{1-a}, \right.
			\\
			&\quad\quad\quad \left. F(Z'^{(s_3)}) \geq (1 - t_n' z)^{1-a} \right) + O(t_n^2).
		\end{align*}
		Indeed, the third term above is the probability of a certain event that requires at least one of the $Z$ and one of the $Z'$ to be large, which has probability at most $O(t_n^2)$ since $Z$ and $Z'$ are assumed independent (recall that we assumed $t_n' = O(t_n)$). We note that the term in front of this probability in the definition of $\Gamma^{(s, s')}$ is equal to $q^{(s)}(t_n)^{-1} = t_n^{-1/\eta^{(s)}}$. However $t_n^2 = o(t_n^{1/\eta^{(s)}})$ since $\eta^{(s)} > 1/2$, and the second probability above is also $o(t_n^{1/\eta^{(s)}})$, following the calculations for \cref{ex_spatial:ims}. Therefore, in this case, $\Gamma^{(s, s')}((x, y), (x', z))$ is equal to the limit
		\begin{align*}
			\lim_{n \to \infty} t_n^{-1/\eta^{(s)}} &\PP\left( F(Z^{(s_1)}) \geq (1 - (t_n x \wedge t_n' x'))^a, \right.
			\\
			&\quad \left. F(Z^{(s_2)}) \geq (1 - t_n y)^a, F(Z^{(s_3)}) \geq (1 - t_n' z)^a \right)
			\\
			= \lim_{n \to \infty} t_n^{-1/\eta^{(s)}} &\PP\left( F(Z^{(s_1)}) \geq 1 - a(t_n x \wedge t_n' x'), \right.
			\\
			&\quad \left. F(Z^{(s_2)}) \geq 1 - at_n y, F(Z^{(s_3)}) \geq 1 - at_n' z \right)
		\end{align*}
		which is non-zero if and only if $(Z^{(s_1)}, Z^{(s_2)}, Z^{(s_3)})$ is fully dependent (i.e., it contains no pairwise independence).

		For the case where the pairs $s = (s_1, s_2)$ and $s' = (s_3, s_4)$ are disjoint, let $t_n = k^{(s)}/n$ and $t_n' = k^{(s')}/n$ and assume as before that $t_n' \lesssim t_n$. By similar arguments as above, one obtains that $\Gamma^{(s, s')}((x, y), (x', y'))$ is equal to the limit
		\begin{align*}
			\lim_{n \to \infty} t_n^{-1/\eta^{(s)}} &\PP\left( F(Z^{(s_1)}) \geq 1 - at_n x, F(Z^{(s_2)}) \geq 1 - at_n y, \right.
			\\
			&\quad \left. F(Z^{(s_3)}) \geq 1 - at_n' x', F(Z^{(s_4)}) \geq 1 - at_n' y' \right),
		\end{align*}
		which is non-zero if and only if $(Z^{(s_1)}, Z^{(s_2)}, Z^{(s_3)}, Z^{(s_4)})$ has no independent pairs.

		Using the same ideas and after straightforward computations, one may calculate the limits $\Gamma^{(s, s', j)}$, for $s' \in \cP_D$. First, consider the case where $s = (s_1, s_2)$ and $s_j' = s_1$, that is the element $s_j'$ is in the pair $s$. Defining $t_n$ and $t_n'$ as above, we still have $t_n' \lesssim t_n$ since $s'$ is an asymptotically dependent pair. Then $\Gamma^{(s, s', j)}((x, y), (x', y'))$ is equal to
		\[
			\chi^{Z, (s')} \lim_{n \to \infty} t_n^{-1/\eta^{(s)}} \Prob{}{ F(Z^{(s_1)}) \geq 1 - a(t_n x \wedge t_n' x'), F(Z^{(s_2)}) \geq 1 - at_n y },
		\]
		which is non-zero if and only if $(Z^{(s_1)}, Z^{(s_2)})$ is dependent. Now if $s_3 := s_j'$ is not an element of $s$, $\Gamma^{(s, s', j)}((x, y), (x', y'))$ becomes
		\[
			\chi^{Z, (s')} \lim_{n \to \infty} t_n^{-1/\eta^{(s)}} \Prob{}{ F(Z^{(s_1)}) \geq 1 - at_n x, F(Z^{(s_2)}) \geq 1 - at_n y, F(Z^{(s_3)}) \geq 1 - at_n' x' },
		\]
		which is non-zero if and only if $(Z^{(s_1)}, Z^{(s_2)}, Z^{(s_3)})$ is fully dependent.

		Finally, for $s, s' \in \cP_D$, again letting $t_n = k^{(s)}/n$ and $t_n' = k^{(s')}/n$, note that this time $t_n'/t_n$ is constant. Without loss of generality, let $j=j'=1$. Then $\Gamma^{(s, j, s', j')}((x, y), (x', y'))$ is equal to
		\[
			\chi^{Z, (s)} \chi^{Z, (s')} \lim_{n \to \infty} t_n^{-1} \Prob{}{ F(Z^{(s_1)}) \geq 1 - t_n x, F(Z^{(s_1')}) \geq 1 - t_n' y' },
		\]
		which is non-zero if and only if $(Z^{(s_1)}, Z^{(s_1')})$ is dependent.
		\hfill $\square$

		\section{Proof of the claims in \cref{ex2:cont}}
		\label{proof rsc}

		\newcommand{\aR}{{\alpha_R}}
		\newcommand{\aW}{{\alpha_W}}
		\newcommand{\amin}{{\alpha_\wedge}}
		\newcommand{\amax}{{\alpha_\vee}}

		The multiplicative constant appearing in the scaling function $q$, as a function of $\lambda$, is given by
		\begin{equation} \label{eq:K_lambda}
		K_\lambda = \begin{cases}
		2 \frac{1-\lambda}{2-\lambda}, \quad &\lambda \in (0, 1) \\
		2, &\lambda = 1 \\
		\left( 1 - \frac{1}{\lambda} \right)^{\lambda-1} \frac{2(\lambda-1)}{\lambda(2-\lambda)}, &\lambda \in (1, 2) \\
		\frac{1}{2}, &\lambda = 2 \\
		\frac{\left( 1 - \frac{1}{\lambda} \right)^2}{1 - \frac{2}{\lambda}}, &\lambda \in (2, \infty)
		\end{cases};
		\end{equation}
		it can be deduced from the proof.

		The argument must be separated in two cases depending on whether $\lambda=1$.

		\subsection{The case \texorpdfstring{$\lambda \neq 1$}{lambda not equal 1}}

		For now, assume that $\alpha_R \neq \alpha_W$. Let $\bar F_R$ denote the survival function of $R$. Then $\bar F_R(x) = x^{-\alpha_R}$ for $x>1$, and $\bar F_R(x) = 1$ for $x \leq 1$. The first step in calculationg $Q$ is to find an expression for the survival function $\bar F$ of $X$ (and equivalently of $Y$) and its inverse. We have, for $x \geq 1$,
		\begin{align*}
		\bar F(x) &= \Prob{}{RW_1 > x} \\
		&= \Prob{}{R > \frac{x}{W_1}} \\
		&= \E{}{\bar F_R \left( \frac{x}{W_1} \right)} \\
		&= \Prob{}{W_1 > x} + \int_1^x \left( \frac{w}{x} \right)^{\alpha_R} \frac{\alpha_W}{w^{\alpha_W+1}} \d w \\
		&= x^{-\alpha_W} + \alpha_W x^{-\alpha_R} \frac{x^{\alpha_R - \alpha_W} - 1}{\alpha_R - \alpha_W} \\
		&= \frac{\alpha_R}{\alpha_R - \alpha_W} x^{-\alpha_W} - \frac{\alpha_W}{\alpha_R - \alpha_W} x^{-\alpha_R} \\
		&= \frac{\amax}{\amax - \amin} x^{-\amin} \left( 1 - \frac{\amin}{\amax - \amin} x^{\amin - \amax} \right),
		\end{align*}
		where $\amin$ and $\amax$ denote the smallest and the largest of the two $\alpha$'s, respectively. Although not easily invertible, this function is close to $\frac{\amax}{\amax - \amin} x^{-\amin}$, which has an analytical inverse. We now argue that this inverse is close to that of $\bar F$. First, for any $X \in (1, \infty)$, we have for $x \in [X, \infty)$
		$$
		\underbrace{\frac{\amax}{\amax - \amin} x^{-\amin} \left( 1 - \frac{\amin}{\amax - \amin} X^{\amin - \amax} \right)}_{f_1(x)} \leq \bar F(x) \leq \underbrace{\frac{\amax}{\amax - \amin} x^{-\amin}}_{f_2(x)}.
		$$

		Now note that for two decreasing, invertible functions $g_1$ and $g_2$, $g_1 \leq g_2$ is equivalent to $g_1^{-1} \leq g_2^{-1}$. This means that as soon as $y \leq f_1(X)$, $f_1^{-1}(y) \leq \bar F^{-1}(y) \leq f_2^{-1}(y)$. In other words, for such $y$,
		$$
		\left( 1 - \frac{\amin}{\amax - \amin} X^{\amin - \amax} \right)^{1/\amin} \left( \frac{\amax}{\amax - \amin} \right)^{1/\amin} y^{-1/\amin} \leq \bar F^{-1}(y) \leq \left( \frac{\amax}{\amax - \amin} \right)^{1/\amin} y^{-1/\amin}.
		$$

		Because these inequalities are true as soon as $y \leq f_1(X)$, they are true if $y = f_1(X)$. If $y$ is small enough, choosing $X = \left( \frac{1}{2} \frac{\amax}{\amax - \amin} \right)^{1/\amin} y^{-1/\amin}$ is sufficient to have $y \leq f_1(X)$. Therefore, if $y$ is small enough, the first inequality in the last display becomes
		$$
		\bar F^{-1}(y) \geq \left( 1 - O \left( y^{\frac{\amax}{\amin} - 1} \right) \right) \left( \frac{\amax}{\amax - \amin} \right)^{1/\amin} y^{-1/\amin}.
		$$

		Combining this with the upper bound (the second inequality) yields
		\begin{equation} \label{quantile function}
		\bar F^{-1}(y) = \left( 1 + O \left( y^\tau \right) \right) \left( \frac{\amax}{\amax - \amin} \right)^{1/\amin} y^{-1/\amin},
		\end{equation}
		where $\tau = \frac{\amax}{\amin} - 1$.

		The copula $Q$ can now be expressed as
		\begin{align*}
		Q(tx, ty) &= \Prob{}{X \geq \bar F^{-1}(tx), Y \geq \bar F^{-1}(ty)} \\
		&= \Prob{}{RW_1 \geq \bar F^{-1}(tx), RW_2 \geq \bar F^{-1}(ty)} = \Prob{}{R \geq Z} = \E{}{\bar F_R(Z)},
		\end{align*}
		where
		$$
		Z := Z(tx, ty) = \frac{ \bar F^{-1}(tx)}{W_1} \vee \frac {\bar F^{-1}(ty)}{W_2}.
		$$

		Recalling the definition of $\bar F_R$, we have
		\begin{align*}
		Q(tx, ty) &= \Prob{}{Z \leq 1} + \E{}{Z^{-\alpha_R} ; Z>1} \\
		&= \Prob{}{Z \leq 1} + \int_0^\infty \Prob{}{Z^{-\alpha_R} > a, Z > 1} \d a \\
		&= \Prob{}{Z \leq 1} + \int_0^\infty \Prob{}{1 < Z \leq a^{-1/\aR}} \d a \\
		&= \Prob{}{Z \leq 1} + \int_0^1 \Prob{}{1 < Z \leq a^{-1/\aR}} \d a \\
		&= \Prob{}{Z \leq 1} + \int_0^1 \left( \Prob{}{Z \leq a^{-1/\aR}} - \Prob{}{Z \leq 1} \right) \d a \\
		&= \int_0^1 \Prob{}{Z \leq a^{-1/\aR}} \d a.
		\end{align*}

		In order to compute the previous integral, we need to derive the CDF of $Z$. From the definition of $Z$ and by independence of $W_1$ and $W_2$, it is clear that, for any $z > 0$,
		$$
		\Prob{}{Z \leq z} = \Prob{}{W_1 \geq \frac{\bar F^{-1}(tx)}{z}} \Prob{}{W_2 \geq \frac{\bar F^{-1}(ty)}{z}}.
		$$

		From now on, assume without loss of generality that $x \geq y$ since $c(x, y) = c(y, x)$ (because the random variables $X$ and $Y$ are exchangeable). Then $\bar F^{-1}(tx) \leq \bar F^{-1}(ty)$. The previous probability can take 3 different forms:
		$$
		\Prob{}{Z \leq z} = \begin{cases}
		\left( \bar F^{-1}(tx) \bar F^{-1}(ty) \right)^{-\aW} z^{2\aW}, \quad &\text{if } z \leq \bar F^{-1}(tx) \\
		\left( \bar F^{-1}(ty) \right)^{-\aW} z^\aW, &\text{if } \bar F^{-1}(tx) < z \leq \bar F^{-1}(ty) \\
		1, &\text{if } z > \bar F^{-1}(ty)
		\end{cases}.
		$$

		When substituting $z = a^{-1/\aR}$, for $a \in (0, 1)$, notice that we are in the three preceding cases, respectively, when
		$$
		\begin{cases}
		a \geq \left( \bar F^{-1}(tx) \right)^{-\aR} \\
		\left( \bar F^{-1}(ty) \right)^{-\aR} \leq a < \left( \bar F^{-1}(tx) \right)^{-\aR} \\
		a < \left( \bar F^{-1}(ty) \right)^{-\aR}
		\end{cases}.
		$$

		This allows us to write
		\begin{align}
		Q(tx, ty) = \int_0^{\left( \bar F^{-1}(ty) \right)^{-\aR}} \d a &+ \left( \bar F^{-1}(ty) \right)^{-\aW} \int_{\left( \bar F^{-1}(ty) \right)^{-\aR}}^{\left( \bar F^{-1}(tx) \right)^{-\aR}} a^{-\frac{\aW}{\aR}} \d a \notag \\
		&+ \left( \bar F^{-1}(tx) \bar F^{-1}(ty) \right)^{-\aW} \int_{\left( \bar F^{-1}(tx) \right)^{-\aR}}^1 a^{-2 \frac{\aW}{\aR}} \d a. \label{3 int}
		\end{align}

		Since we only need \cref{Q} to hold uniformly over a sphere, we may assume that $y \leq x \leq 1$. Then, \cref{quantile function} yields
		$$
		\bar F^{-1}(tx) = (1 + O(t^\tau)) \left( \frac{\amax}{\amax - \amin} \right)^{1/\amin} (tx)^{-1/\amin}
		$$
		and the same for $\bar F^{-1}(ty)$. Moreover, the term $O(t^\tau)$ is uniform over all $(x, y) \in [0, 1]^2$. The first term in \cref{3 int} is then equal to
		$$
		\left( \bar F^{-1}(ty) \right)^{-\aR} = (1 + O(t^\tau)) \left( 1 - \frac{\amin}{\amax} \right)^{\frac{\aR}{\amin}} t^{\frac{\aR}{\amin}} y^{\frac{\aR}{\amin}} =: Q^{(1)}(tx, ty),
		$$
		the second one is equal to
		\begin{align*}
		&\left( \bar F^{-1}(ty) \right)^{-\aW} \left. \frac{a^{1 - \frac{\aW}{\aR}}}{1 - \frac{\aW}{\aR}} \right|_{a = \left( \bar F^{-1}(ty) \right)^{-\aR}}^{\left( \bar F^{-1}(tx) \right)^{-\aR}} \\
		&\quad = \frac{1}{1 - \frac{\aW}{\aR}} \left( \bar F^{-1}(ty) \right)^{-\aW} \left( \bar F^{-1}(tx)^{-(\aR - \aW)} - \bar F^{-1}(ty)^{-(\aR - \aW)} \right) \\
		&\quad = (1 + O(t^\tau)) \frac{ \left( 1 - \frac{\amin}{\amax} \right)^{\frac{\aR}{\amin}} }{1 - \frac{\aW}{\aR}} t^{\frac{\aR}{\amin}} y^{\frac{\aW}{\amin}} \left( x^{\frac{\aR - \aW}{\amin}} - y^{\frac{\aR - \aW}{\amin}} \right) \\
		&\quad =: Q^{(2)}(tx, ty)
		\end{align*}
		and finally the third one is equal to
		\begin{align*}
		&\left( \bar F^{-1}(tx) \bar F^{-1}(ty) \right)^{-\aW} \left. \frac{a^{1 - 2\frac{\aW}{\aR}}}{1 - 2\frac{\aW}{\aR}} \right|_{a = \left( \bar F^{-1}(tx) \right)^{-\aR}}^1 \\
		&\quad = \frac{1}{1 - 2\frac{\aW}{\aR}} \left( \bar F^{-1}(tx) \bar F^{-1}(ty) \right)^{-\aW} \left( 1 - \left( \bar F^{-1}(tx) \right)^{2\aW - \aR} \right) \\
		&\quad = (1 + O(t^\tau)) \frac{ \left( 1 - \frac{\amin}{\amax} \right)^{2\frac{\aW}{\amin}} }{1 - 2\frac{\aW}{\aR}} t^{2\frac{\aW}{\amin}} (xy)^{\frac{\aW}{\amin}} \left( 1 - \left( 1 - \frac{\amin}{\amax} \right)^{\frac{\aR - 2\aW}{\amin}} (tx)^{\frac{\aR - 2\aW}{\amin}} \right) \\
		&\quad =: Q^{(3a)}(tx, ty)
		\end{align*}
		in the case where $\aR \neq 2\aW$, and if $\aR = 2\aW$, it is equal to
		\begin{align*}
		&-\left( \bar F^{-1}(tx) \bar F^{-1}(ty) \right)^{-\aW} \log\left( \left( \bar F^{-1}(tx) \right)^{-\aR} \right) \\
		&\quad = (1 + O(t^\tau)) \left( 1 - \frac{\amin}{\amax} \right)^{2\frac{\aW}{\amin}} t^{2\frac{\aW}{\amin}} (xy)^{\frac{\aW}{\amin}} \\
		&\quad\quad \times \left( -\log \left( (1 + O(t^\tau)) \left( 1 - \frac{\amin}{\amax} \right)^{\frac{\aR}{\amin}} \right) + \frac{\aR}{\amin} \left( \log(1/x) + \log(1/t) \right) \right) \\
		&\quad = \frac{1}{2} t^2 \log(1/t) xy + O(t^2) \\
		&\quad =: Q^{(3b)}(tx, ty),
		\end{align*}
		where the term $O(t^2)$ is uniform over $(x, y) \in [0, 1]^2$. We now  divide the possible values of $\lambda = \aR/\aW$ in four ranges and determine which of the three terms $Q^{(1)}$, $Q^{(2)}$ or $Q^{(3)}$ dominates.

		\subsubsection{\texorpdfstring{$\lambda \in (0, 1)$}{lambda in (0, 1)}}

		This is the case where we obtain asymptotic dependence. All three terms are of the order of $t$, so they all matter. In this case, $\amin = \aR$, $\amax = \aW$ and $\tau = 1/\lambda - 1$. Therefore,
		\begin{align*}
		Q^{(1)}(tx, ty) &= (1 + O(t^\tau)) \left( 1 - \frac{\aR}{\aW} \right) ty = (1-\lambda) ty + O \left( t^{1 + \tau} \right), \\
		Q^{(2)}(tx, ty) &= (1 + O(t^\tau)) \frac{1 - \frac{\aR}{\aW}}{1 - \frac{\aW}{\aR}} t y^{\frac{\aW}{\aR}} \left( x^{1 - \frac{\aW}{\aR}} - y^{1 - \frac{\aW}{\aR}} \right) \\
		&= (1 + O(t^\tau)) \frac{\aR}{\aW}  t \left( y - x^{1 - \frac{\aW}{\aR}} y^{\frac{\aW}{\aR}} \right) \\
		&= \lambda t \left( y - x^{1 - 1/\lambda} y^{1/\lambda} \right) + O \left( t^{1 + \tau} \right), \\
		Q^{(3a)}(tx, ty) &= (1 + O(t^\tau)) \frac{ \left( 1 - \frac{\aR}{\aW} \right)^{2\frac{\aW}{\aR}} }{2\frac{\aW}{\aR} - 1} t^{2\frac{\aW}{\aR}} (xy)^{\frac{\aW}{\aR}} \left( \left( 1 - \frac{\aR}{\aW} \right)^{1 - 2\frac{\aW}{\aR}} (tx)^{1 - 2\frac{\aW}{\aR}} - 1 \right) \\
		&= (1 + O(t^\tau)) \frac{1 - \frac{\aR}{\aW}}{2\frac{\aW}{\aR} - 1} t x^{1 - \frac{\aW}{\aR}} y^{\frac{\aW}{\aR}} + O \left( t^{2\frac{\aW}{\aR}} \right) \\
		&= \lambda \frac{1-\lambda}{2 - \lambda} t x^{1 - 1/\lambda} y^{1/\lambda} + O \left( t^{1+\tau} + t^{2/\lambda} \right) = \lambda \frac{1-\lambda}{2 - \lambda} t x^{1 - 1/\lambda} y^{1/\lambda} + O \left( t^{1+\tau} \right),
		\end{align*}
		where in the last line we have used $1+\tau = \amax/\amin = 1/\lambda < 2/\lambda$. Therefore in this case we get
		\begin{align*}
		Q(tx, ty) &= Q^{(1)}(tx, ty) + Q^{(2)}(tx, ty) + Q^{(3a)}(tx, ty) \\
		&= (1-\lambda) ty + \lambda t \left( y - x^{1 - 1/\lambda} y^{1/\lambda} \right) + \lambda \frac{1-\lambda}{2 - \lambda} t x^{1 - 1/\lambda} y^{1/\lambda} + O \left( t^{1+\tau} \right) \\
		&= t \left( y + \left( -\lambda + \lambda \frac{1-\lambda}{2 - \lambda} \right) x^{1 - 1/\lambda} y^{1/\lambda} \right) + O \left( t^{1+\tau} \right) \\
		&= t \left( y - \frac{\lambda}{2-\lambda} x^{1 - 1/\lambda} y^{1/\lambda} \right) + O \left( t^{1+\tau} \right).
		\end{align*}

		\subsubsection{\texorpdfstring{$\lambda \in (1, 2)$}{lambda in (1, 2)}}

		Here again, all three terms are of the order of $t^\lambda$ so they all matter. Note that here and in the next two cases, $\amin = \aW$, $\amax = \aR$ and $\tau = \lambda-1$. Through similar calculations as before, we obtain this time
		\begin{align*}
		Q^{(1)}(tx, ty) &= (1 + O(t^\tau)) \left( 1 - \frac{\aW}{\aR} \right)^{\frac{\aR}{\aW}} t^{\frac{\aR}{\aW}} y^{\frac{\aR}{\aW}} = \left( 1 - \frac{1}{\lambda} \right)^\lambda t^\lambda y^\lambda + O \left( t^{\lambda+\tau} \right), \\
		Q^{(2)}(tx, ty) &= (1 + O(t^\tau)) \frac{\left( 1 - \frac{\aW}{\aR} \right)^{\frac{\aR}{\aW}}}{1 - \frac{\aW}{\aR}} t^{\frac{\aR}{\aW}} y \left( x^{\frac{\aR}{\aW} - 1} - y^{\frac{\aR}{\aW} - 1} \right) \\
		&= \left( 1 - \frac{1}{\lambda} \right)^{\lambda-1} t^\lambda \left( x^{\lambda - 1} y - y^\lambda \right) + O \left( t^{\lambda+\tau} \right), \\
		Q^{(3a)}(tx, ty) &= (1 + O(t^\tau)) \frac{ \left( 1 - \frac{\aW}{\aR} \right)^2 }{2\frac{\aW}{\aR} - 1} t^2 xy \left( \left( 1 - \frac{\aW}{\aR} \right)^{\frac{\aR}{\aW} - 2} (tx)^{\frac{\aR}{\aW} - 2} - 1 \right) \\
		&= (1 + O(t^\tau)) \frac{ \left( 1 - \frac{1}{\lambda} \right)^2 }{\frac{2}{\lambda} - 1} t^2 xy \left( \left( 1 - \frac{1}{\lambda} \right)^{\lambda - 2} (tx)^{\lambda - 2} - 1 \right) \\
		&= (1 + O(t^\tau)) \frac{ \left( 1 - \frac{1}{\lambda} \right)^\lambda }{\frac{2}{\lambda} - 1} t^\lambda x^{\lambda-1} y + O \left( t^2 \right) \\
		&= \lambda \frac{ \left( 1 - \frac{1}{\lambda} \right)^\lambda }{2 - \lambda} t^\lambda x^{\lambda-1} y + O \left( t^{\lambda+\tau} + t^2 \right) \\
		&= \lambda \frac{ \left( 1 - \frac{1}{\lambda} \right)^\lambda }{2 - \lambda} t^\lambda x^{\lambda-1} y + O \left( t^{(2\lambda - 1) \wedge 2} \right).
		\end{align*}

		Therefore, $Q$ can be calculated as
		\begin{align*}
		Q(tx, ty) &= Q^{(1)}(tx, ty) + Q^{(2)}(tx, ty) + Q^{(3a)}(tx, ty) \\
		&= \left( 1 - \frac{1}{\lambda} \right)^{\lambda-1} t^\lambda \left( \left( 1 - \frac{1}{\lambda} \right) y^\lambda + x^{\lambda - 1} y - y^\lambda + \lambda \frac{1 - \frac{1}{\lambda}}{2-\lambda} x^{\lambda - 1} y \right) + O \left( t^{(2\lambda - 1) \wedge 2} \right) \\
		&= \left( 1 - \frac{1}{\lambda} \right)^{\lambda-1} t^\lambda \left( -\frac{1}{\lambda} y^\lambda + \left( 1 + \lambda \frac{1 - \frac{1}{\lambda}}{2-\lambda} \right) x^{\lambda - 1} y \right) + O \left( t^{(2\lambda - 1) \wedge 2} \right) \\
		&= \left( 1 - \frac{1}{\lambda} \right)^{\lambda-1} t^\lambda \left( \frac{1}{2-\lambda} x^{\lambda-1} y - \frac{1}{\lambda} y^\lambda \right) + O \left( t^{(2\lambda - 1) \wedge 2} \right).
		\end{align*}

		\subsubsection{\texorpdfstring{$\lambda = 2$}{lambda = 2}}

		In this case, $\aR/\amin = 2$, so we easily see that both $Q^{(1)}(tx, ty)$ and $Q^{(2)}(tx, ty)$ are $O(t^2)$. Because the term $Q^{(3b)}$ is of the order of $t^2 \log(1/t)$, it dominates the preceding two by a factor of $\log(1/t)$. Therefore,
		$$
		Q(tx, ty) = Q^{(3b)}(tx, ty) + O \left( t^2 \right) = \frac{1}{2} t^2 \log(1/t) xy + O \left( t^2 \right).
		$$

		\subsubsection{\texorpdfstring{$\lambda \in (2, \infty)$}{lambda in (2, inf)}}

		Once again, the terms $Q^{(1)}$ and $Q^{(2)}$ are dominated by the third term; they are both of the order of $t^\lambda$, whereas the third term is of the order of $t^2$. Therefore,
		\begin{align*}
		Q(tx, ty) &= Q^{(3a)}(tx, ty) + O \left( t^{\frac{\aR}{\aW}} \right) \\
		&= (1 + O(t^\tau)) \frac{ \left( 1 - \frac{\aW}{\aR} \right)^2 }{1 - 2\frac{\aW}{\aR}} t^2 xy \left( 1 - \left( 1 - \frac{\aW}{\aR} \right)^{\frac{\aR}{\aW} - 2} (tx)^{\frac{\aR}{\aW} - 2} \right) + O \left( t^{\frac{\aR}{\aW}} \right) \\
		&= (1 + O(t^\tau)) \frac{ \left( 1 - \frac{1}{\lambda} \right)^2 }{1 - \frac{2}{\lambda}} t^2 xy + O \left( t^\lambda \right) \\
		&= \frac{ \left( 1 - \frac{1}{\lambda} \right)^2 }{1 - \frac{2}{\lambda}} t^2 xy + O \left( t^{(2+\tau) \wedge \lambda} \right) \\
		&= \frac{ \left( 1 - \frac{1}{\lambda} \right)^2 }{1 - \frac{2}{\lambda}} t^2 xy + O \left( t^\lambda \right),
		\end{align*}
		because, in the last line, $2+\tau = \lambda+1 > \lambda$.

		\subsection{The case \texorpdfstring{$\lambda = 1$}{lambda = 1}}

		\newcommand{\W}{W_{-1}}

		From now on, we assume that $\aR = \aW = \alpha$. That is, $R, W_1, W_2$ are iid with a $\Par{\alpha}$ distribution. Like before, we denote by $\bar F_R$ and $\bar F$ the survival functions of $R$ and of $X$ (and equivalently $Y$), respectively. As before, we first find an expression for $\bar F$. For any $x \geq 1$,
		\begin{align*}
		\bar F(x) &= \Prob{}{RW_1 > x} \\
		&= \Prob{}{R > \frac{x}{W_1}} \\
		&= \E{}{\bar F_R \left( \frac{x}{W_1} \right)} \\
		&= \Prob{}{W_1 > x} + \int_1^x \left( \frac{w}{x} \right)^{\alpha} \frac{\alpha}{w^{\alpha+1}} \d w \\
		&= x^{-\alpha} + \alpha x^{-\alpha} \int_1^x \frac{\d w}{w} \\
		&= x^{-\alpha} \left( 1 + \alpha\log(x) \right).
		\end{align*}

		The inverse of this function is given by
		$$
		\bar F^{-1}(y) = \left( \frac{-\W(-y/e)}{y} \right)^{1/\alpha},
		$$
		where $\W$ denotes the lower branch of the Lambert $W$ function; for $y \in [-e^{-1}, 0)$, $\W(y)$ denotes the only solution in $x \in (-\infty, -1]$ of the equation $y = xe^x$. Indeed, it can be seen by a simple plug-in argument that for any $y \in (0, 1]$,
		$$
		\bar F \left( \left( \frac{-\W(-y/e)}{y} \right)^{1/\alpha} \right) = y.
		$$

		Repeating the steps leading to \cref{3 int}, we obtain the following similar integral representation for $Q$:
		\begin{align*}
		Q(tx, ty) &= \int_0^{\left( \bar F^{-1}(ty) \right)^{-\alpha}} \d a + \left( \bar F^{-1}(ty) \right)^{-\alpha} \int_{\left( \bar F^{-1}(ty) \right)^{-\alpha}}^{\left( \bar F^{-1}(tx) \right)^{-\alpha}} a^{-1} \d a \\
		&\quad + \left( \bar F^{-1}(tx) \bar F^{-1}(ty) \right)^{-\alpha} \int_{\left( \bar F^{-1}(tx) \right)^{-\alpha}}^1 a^{-2} \d a \\
		&= \left( \bar F^{-1}(ty) \right)^{-\alpha} + \left( \bar F^{-1}(ty) \right)^{-\alpha} \log \left( \frac{\left( \bar F^{-1}(tx) \right)^{-\alpha}}{\left( \bar F^{-1}(ty) \right)^{-\alpha}} \right) \\
		&\quad + \left( \bar F^{-1}(tx) \bar F^{-1}(ty) \right)^{-\alpha} \left( \left( \bar F^{-1}(tx) \right)^\alpha - 1 \right) \\
		&= \left( \bar F^{-1}(ty) \right)^{-\alpha} \left( 2 + \log \left( \frac{\left( \bar F^{-1}(tx) \right)^{-\alpha}}{\left( \bar F^{-1}(ty) \right)^{-\alpha}} \right) \right) - \left( \bar F^{-1}(tx) \bar F^{-1}(ty) \right)^{-\alpha}.
		\end{align*}

		The last term in this expression is negligible, compared to the first one, by a factor of at least $\left( \bar F^{-1}(ty) \right)^{-\alpha}$, which (we shall see) is small enough to be absorbed by the term $O(q_1(t))$.

		Now, by \cite{corless1996lambertw}, Section 4, we may obtain the following expansion of $\left( \bar F^{-1}(t) \right)^{-\alpha}$ as $t \to 0$:
		\begin{align*}
		\left( \bar F^{-1}(t) \right)^{-\alpha} &= \frac{t}{-\W(-t/e)} \\
		&= \frac{t}{\log(e/t) + \log\log(e/t) + o(1)} \\
		&= \frac{t}{\log(1/t) + \log\log(1/t) + O(1)} \\
		&= \left( 1 + O \left( \frac{1}{\log(1/t)} \right) \right) \frac{t}{\log(1/t) + \log\log(1/t)}.
		\end{align*}

		Note that, since we are only interested in $(x, y) \in (0, 1]^2$ and since we assume $y \leq x$, $1/\log(1/ty) \leq 1/\log(1/tx) \leq 1/\log(1/t)$. Plugging the expansion in our expression for $Q$ yields
		\begin{align}
		Q(tx, ty) &= \left\{ 1 + O \left( \frac{1}{\log(1/t)} \right) \right\} \frac{ty}{\log(1/ty) + \log\log(1/ty)} \notag \\
		&\quad\quad \times \left( 2 + \log \left( \frac{ \left\{ 1 + O \left( \frac{1}{\log(1/t)} \right) \right\} \frac{tx}{\log(1/tx) + \log\log(1/tx)} }{ \left\{ 1 + O \left( \frac{1}{\log(1/t)} \right) \right\} \frac{ty}{\log(1/ty) + \log\log(1/ty)} } \right) \right) \notag \\
		&\quad + O \left( \left( \frac{t}{\log(1/t) + \log\log(1/t)} \right)^2 \right) \notag \\
		&= \left\{ 1 + O \left( \frac{1}{\log(1/t)} \right) \right\} \frac{ty}{\log(1/t) + \log\log(1/t) + O(\log(1/y))} \notag \\
		&\quad\quad \times \left( 2 + \log \left( \left\{ 1 + O \left( \frac{1}{\log(1/t)} \right) \right\} \frac{x}{y} \frac{\log(1/t) + \log\log(1/t) + O(\log(1/y))}{\log(1/t) + \log\log(1/t) + O(\log(1/x))} \right) \right) \label{Q expansion}
		\\
		&\quad + O \left( \left( \frac{t}{\log(1/t) + \log\log(1/t)} \right)^2 \right) \notag
		\end{align}

		Note that the first term thereof can be written as
		\begin{align*}
		\frac{ty}{\log(1/t) + \log\log(1/t) + O(\log(1/y))} &= \frac{ty}{\log(1/t) + \log\log(1/t)} \left\{ 1 + O \left( \frac{\log(1/y)}{\log(1/t)} \right) \right\} \\
		&= \frac{ty}{\log(1/t) + \log\log(1/t)} \left\{ 1 + O \left( \frac{1}{\log(1/t)} \right) \right\}
		\end{align*}
		because as $y$ approaches 0, the term $\log(1/y)$ gets absorbed by the term $y$ on the numerator. Furthermore,
		\begin{align*}
		\frac{x}{y} \frac{\log(1/t) + \log\log(1/t) + O(\log(1/y))}{\log(1/t) + \log\log(1/t) + O(\log(1/x))} &= \frac{x}{y} \left\{ 1 + O \left( \frac{\log(1/x) + \log(1/y)}{\log(1/t)} \right) \right\} \\
		&= \frac{x}{y} \left\{ 1 + O \left( \frac{\log(1/y)}{\log(1/t)} \right) \right\}.
		\end{align*}

		Thus the $\log$ term in \cref{Q expansion} equals
		\begin{align*}
		\log \left( \frac{x}{y} + O \left( \frac{x}{y} \frac{\log(1/y)}{\log(1/t)} \right) \right) = \log \left( \frac{x}{y} \right) + O \left( \frac{ \frac{x}{y} \frac{\log(1/y)}{\log(1/t)} }{x/y} \right) = \log \left( \frac{x}{y} \right) + O \left( \frac{\log(1/y)}{\log(1/t)} \right),
		\end{align*}
		where we have used the fact that, for any $a \geq 1$ and $b \geq 0$, $\log(a + b) \leq \log(a) + b/a$ (recall that $x/y \geq 1$). Piecing everything together, \cref{Q expansion} may be rewritten as
		\begin{align*}
		Q(tx, ty) &= \frac{ty}{\log(1/t) + \log\log(1/t)} \left( 2 + \log \left( \frac{x}{y} \right) + O \left( \frac{\log(1/y)}{\log(1/t)} \right) \right) \left\{ 1 + O \left( \frac{1}{\log(1/t)} \right) \right\} \\
		&= \frac{ty}{\log(1/t) + \log\log(1/t)} \left( 2 + \log \left( \frac{x}{y} \right) \right) \left\{ 1 + O \left( \frac{1}{\log(1/t)} \right) \right\},
		\end{align*}
		once again because the term $\log(1/y)$ is absorbed by $y$ as $y$ approaches 0. Recalling that we assumed $y \leq x$, the claim follows.
		\hfill $\square$

		\section{A few words on the computational complexity of the method in spatial problems}
		\label{sec:cc}

		Both estimators we propose in the spatial setting (defined in \cref{eq:defthetahatls,eq:defthetahat}) essentially rely on the evaluation of bivariate functions and as such are much faster than methods based on full likelihood (especially if the number of locations is large). A comparison with pairwise likelihood depends on the cost of likelihood evaluations in the particular model under consideration and the type of weight functions that we choose. For the sake of brevity we will focus on the estimator $\hat \vartheta$ from \cref{eq:defthetahatls}; similar arguments apply to $\tilde \vartheta$ from \cref{eq:defthetahat} with obvious modifications.

		Typically, we expect that $\hat \vartheta$ can be computed faster than a pairwise likelihood-based estimator. The main computational burden arises when computing the pairwise empirical integrals $\int g(x,y) \hat Q^{(s)}(kx/n, ky/n) dxdy$ and the corresponding estimators $\hat\theta_n^{(s)}$. In computing those estimators, when finding the minimizer of
		\[
		\Big\| \int g(x, y) \hat Q^{(s)}(kx/n, ky/n) dxdy - \zeta\int g(x, y) c_\theta(x, y) dxdy \Big\|
		\] 
		through numerical optimization, only population level integrals $\int g(x,y) c_{\theta}(x,y) dxdy$ need to be re-computed for each optimization step. For specific models (such as the inverted Brown--Resnick process considered in our application) those integrals have simple analytic expressions, which additionally speeds up the computation. In comparison, the likelihood of a bivariate extreme value model may be substantially more costly to compute, and it needs to be evaluated at every optimization step.

		The above procedure only needs to be completed once and can easily be parallelized by considering pairs independently. Once the estimators $\hat\theta_n^{(s)}$ are available, the objective function in \cref{eq:defthetahatls} only depends on evaluating the low-dimensional functions $h^{(s)}$. Again, in our example those are very simple analytic functions.

		To give a rough idea of the computation times for the proposed methods in a specific example, we report below average computation times for the spatial simulation study in \cref{sim spatial}, with $d=40$ locations (corresponding to 780 pairs), $n=5000$ and a few different values of $m$. All computation times are for computing both spatial estimators simultaneously (but the time to compute only one is not so different since most of the ``pairwise" steps leading to each estimator are the same). The values given are averaged based on 100 repetitions and the values in parenthesis are standard deviations. All computations were executed on a personal laptop with a 2.5GHz Intel Core i5-7200U processor without utilizing parallel computation.

		\begin{center}
		\begin{tabular}{c|c|c|c|c|c}
		$m$ & 25 & 100 & 250 & 500 & 1000 \\
		\hline
		time (seconds) & 9.6 (0.6) & 9.5 (0.3) & 9.6 (0.4) & 9.8 (0.3) & 9.8 (0.3)
		\end{tabular}
		\end{center}

		\section{Additional simulation results}
		\label{simulations-add}

		This section contains additional simulation results not included in \cref{simulations}.

		\subsection{Bivariate distributions}

		The following scatter plots represent data from each of the three bivariate models M1--M3 found in \cref{sim bivariate}. For illustration purposes, there is no additive noise and the marginals are transformed to unit exponential.

		\begin{figure}[H]
			\centering
			\includegraphics[scale = 0.35, trim = 0 50 0 60]{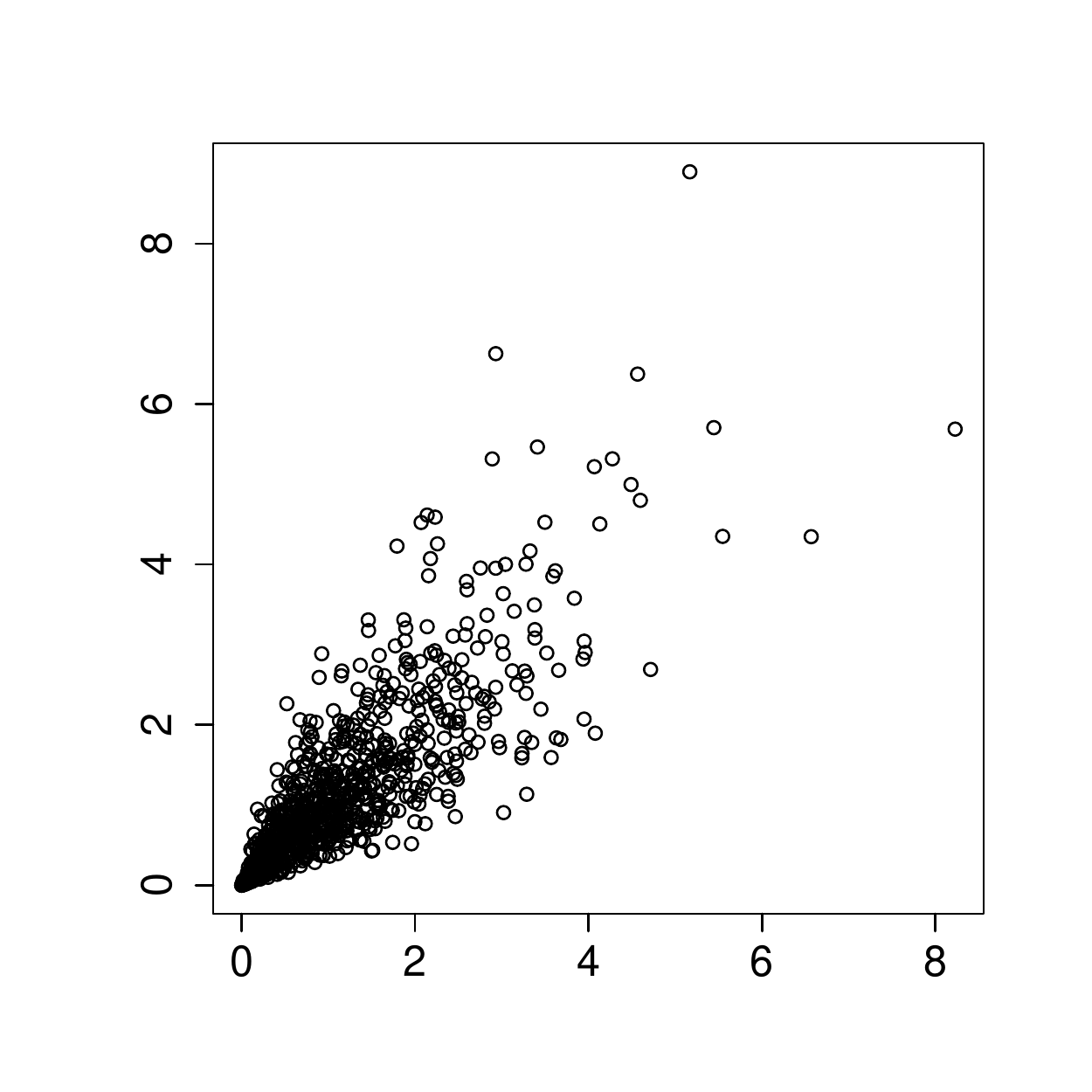}
			\includegraphics[scale = 0.35, trim = 0 50 0 60]{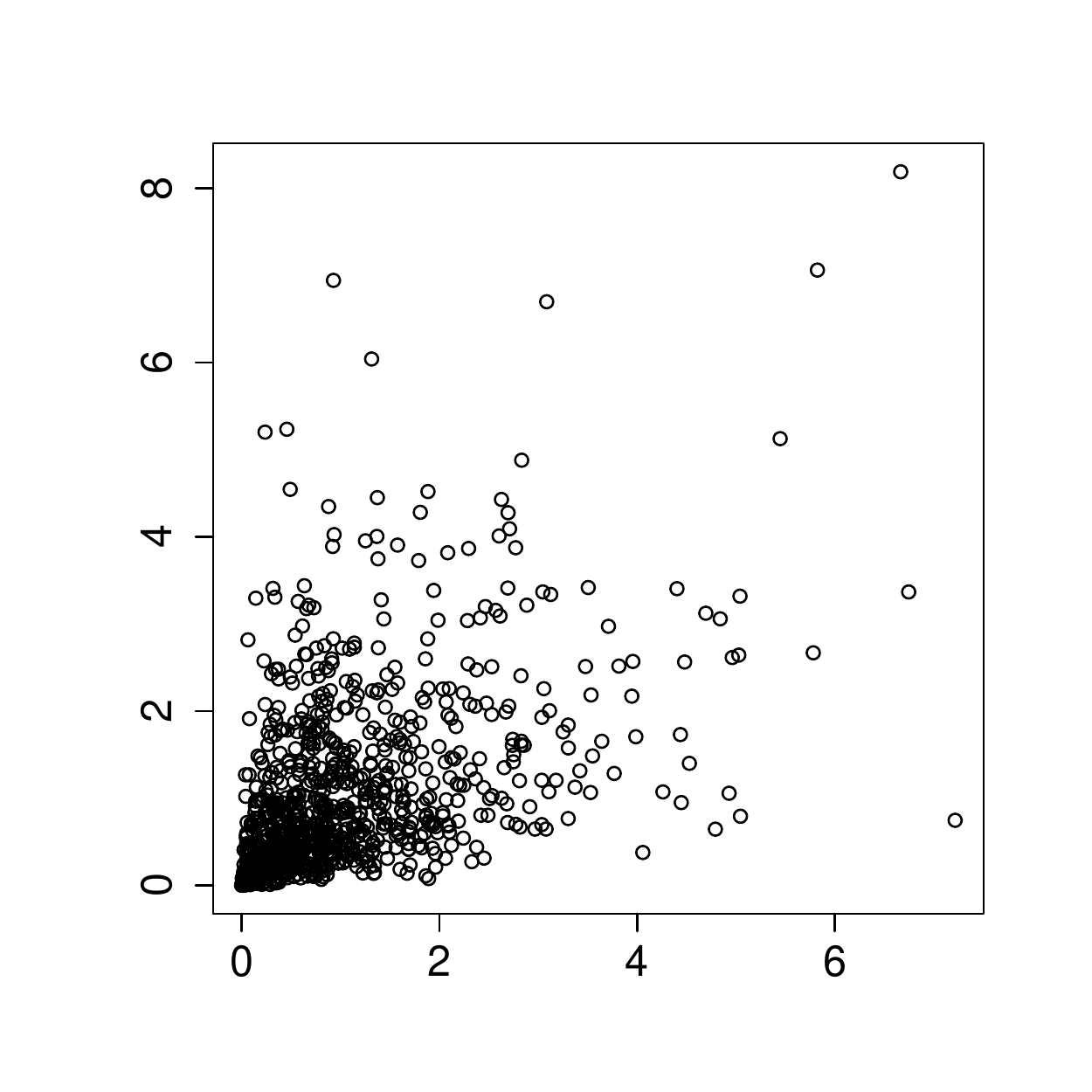}
			\includegraphics[scale = 0.35, trim = 0 50 0 60]{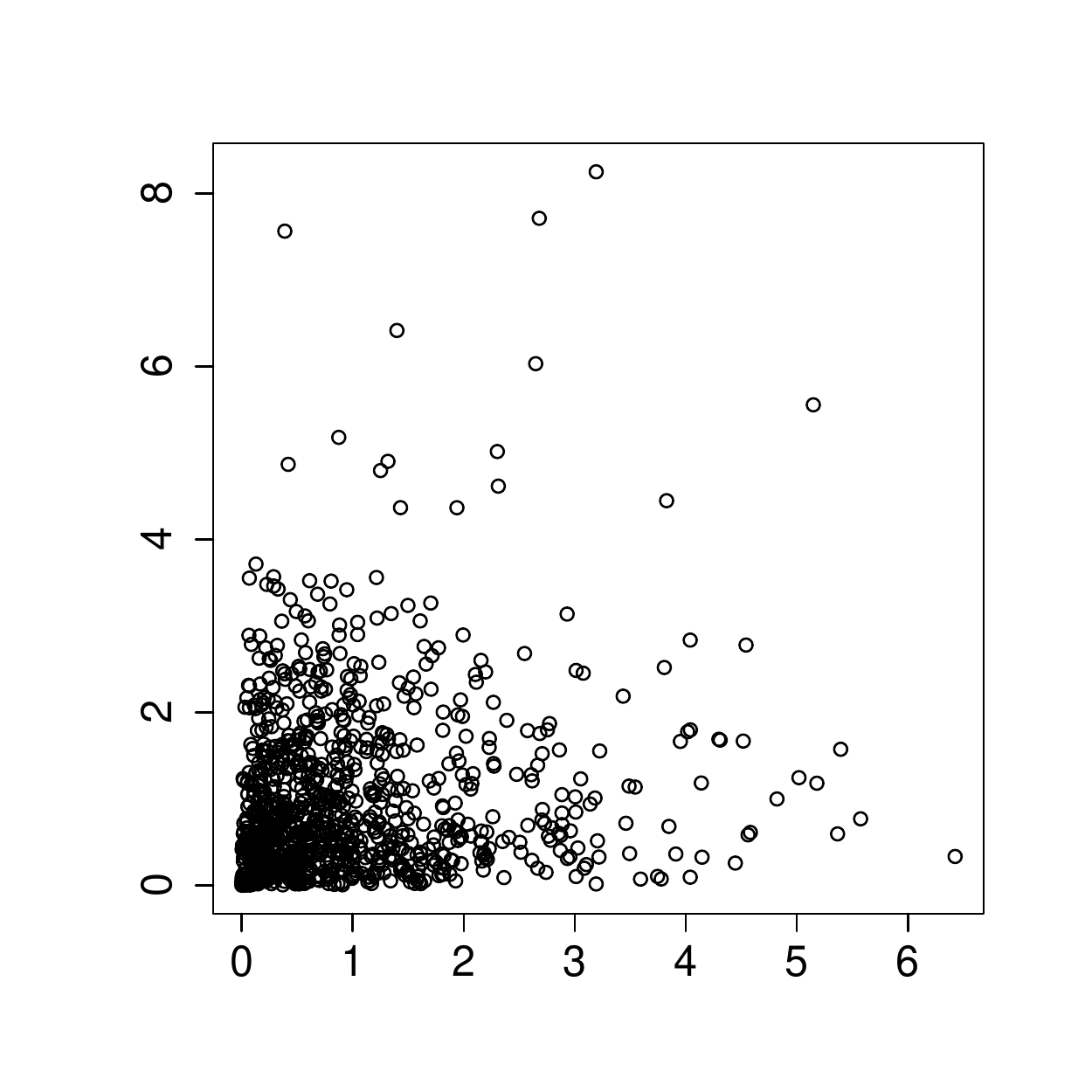}
			\caption{Samples of 1\,000 data points from the inverted H\"usler--Reiss distribution with parameter $\theta$ equal to 0.6, 0.75 and 0.9, from left to right. The marginal distributions are scaled to unit exponential.
			\label{fig:IHR-rea}}
		\end{figure}

		\begin{figure}[H]
			\centering
			\includegraphics[scale = 0.35, trim = 0 50 0 60]{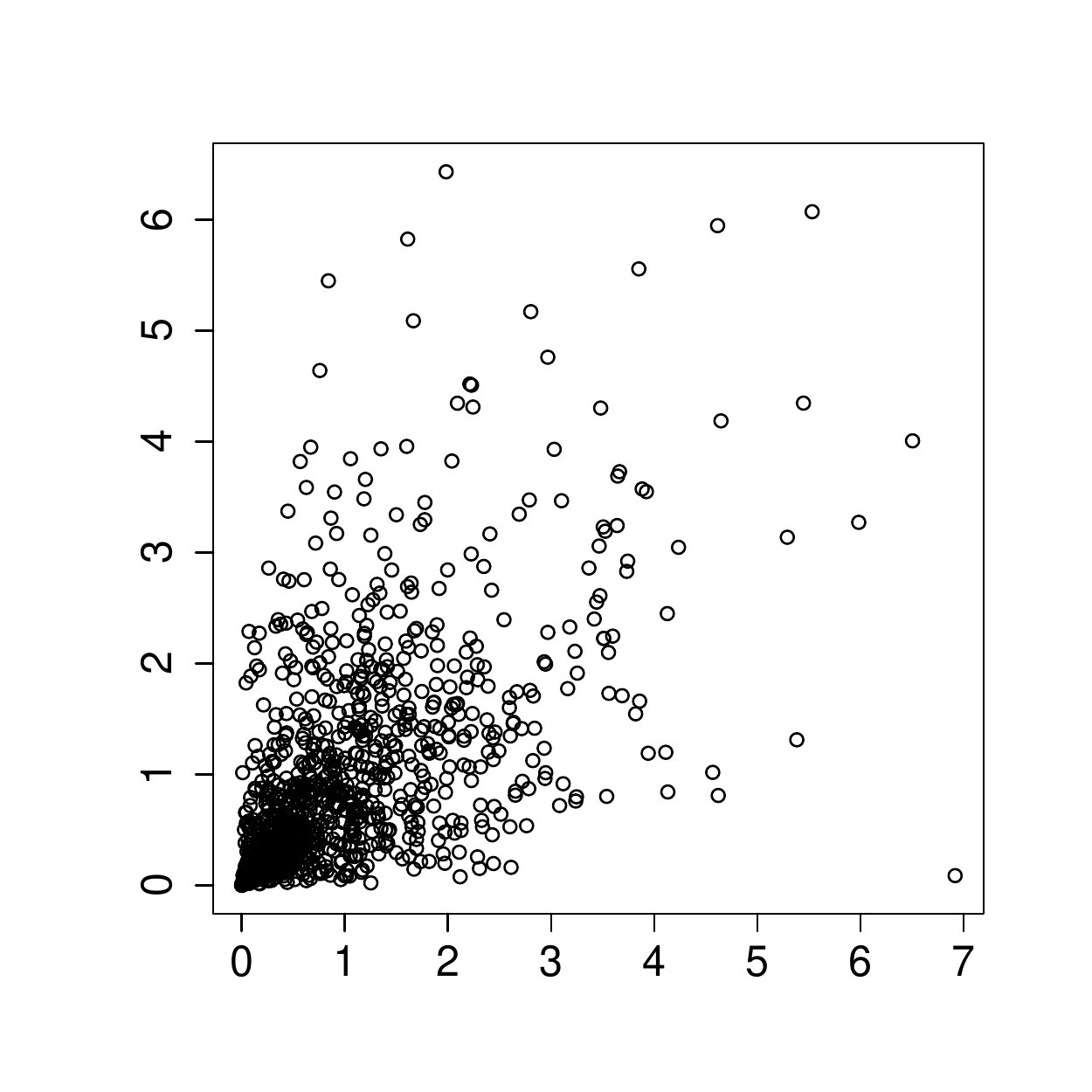}
			\includegraphics[scale = 0.35, trim = 0 50 0 60]{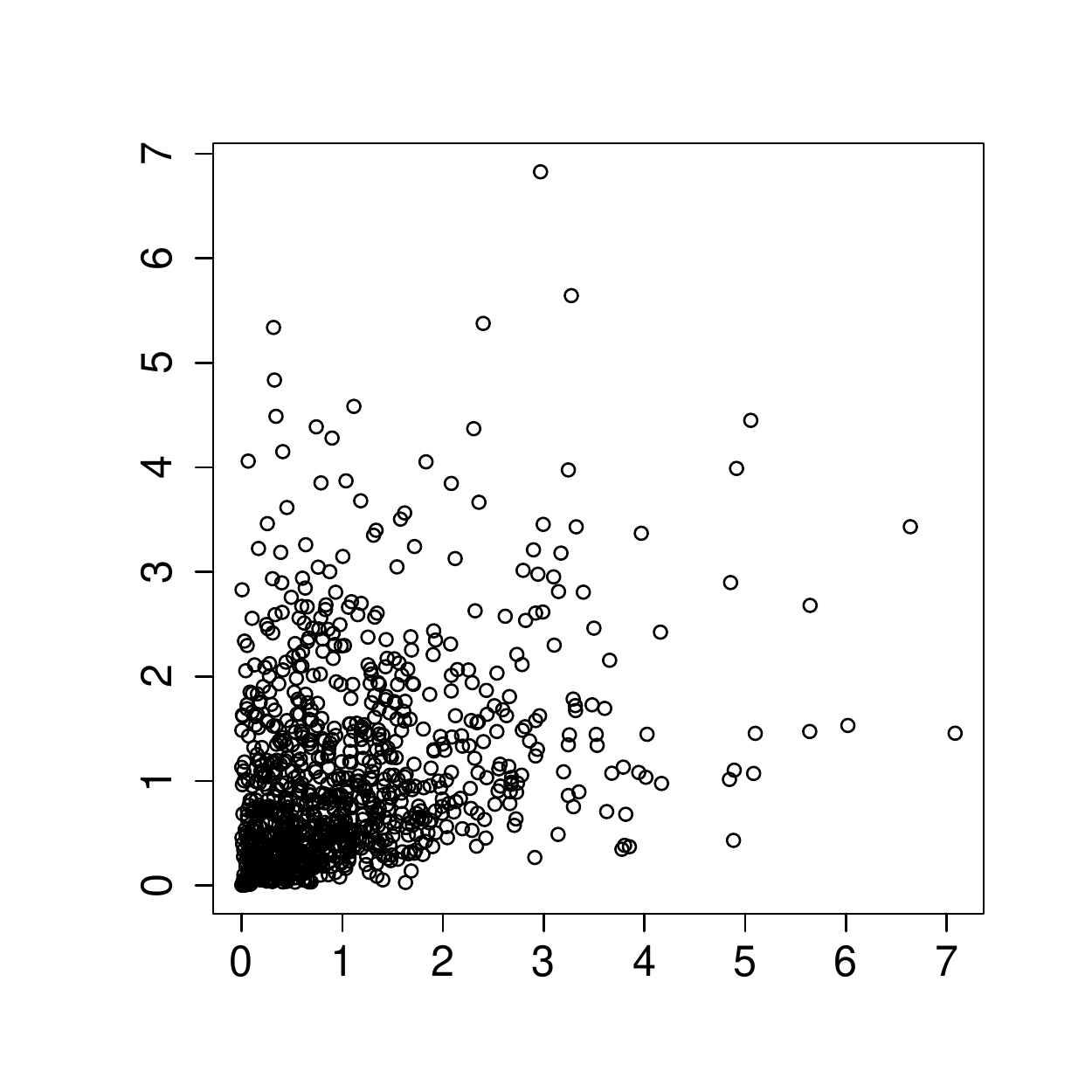}
			\includegraphics[scale = 0.35, trim = 0 50 0 60]{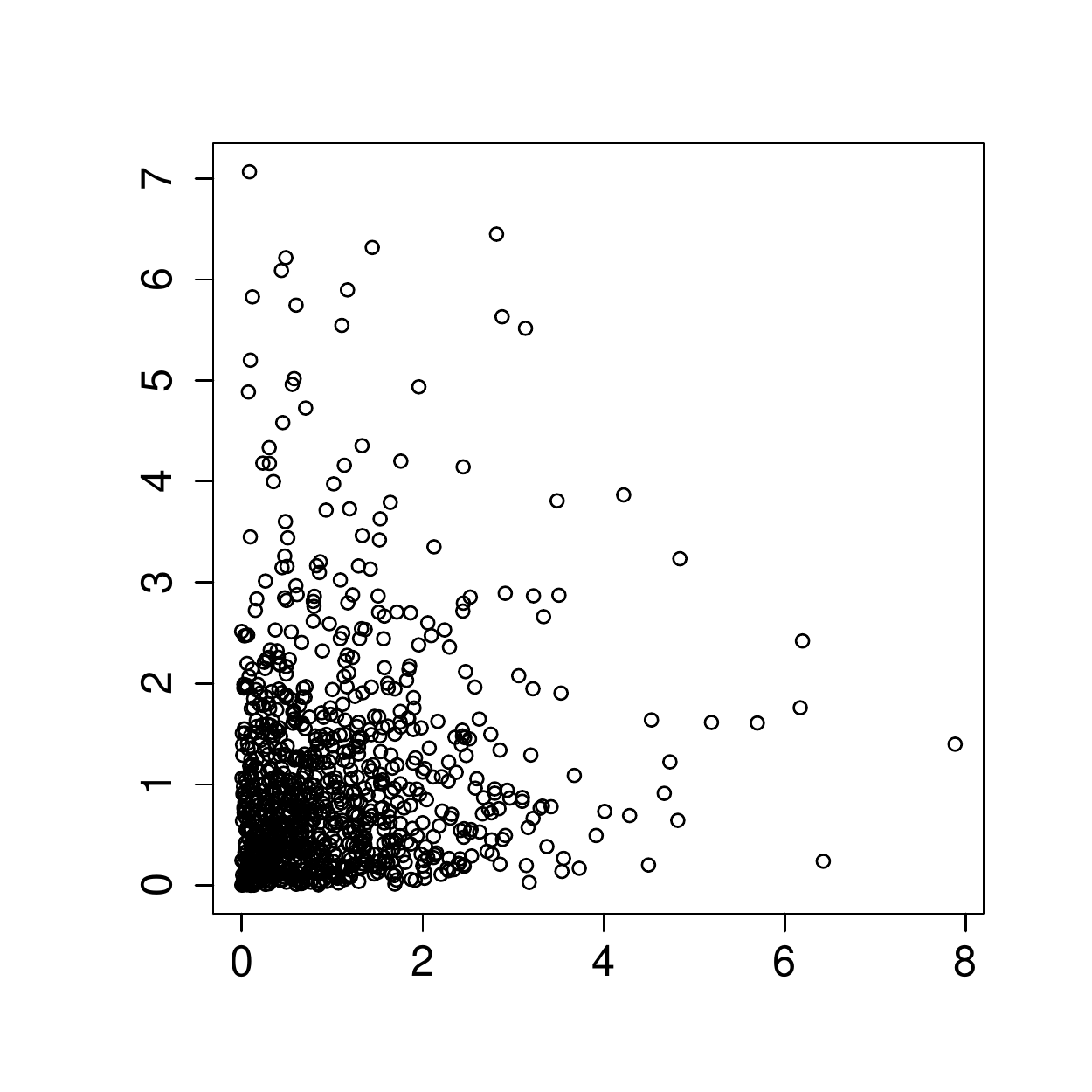}
			\caption{Samples of 1\,000 data points from the inverted asymmetric logistic distribution with parameter $\theta$ equal to $(0.72, 0.72)$, $(0.75, 0.91)$ and $(0.91, 0.91)$, from left to right. The marginal distributions are unit exponential.} \label{fig:IAlog-rea}
		\end{figure}

		\begin{figure}[H]
			\centering
			\includegraphics[scale = 0.35, trim = 0 50 0 60]{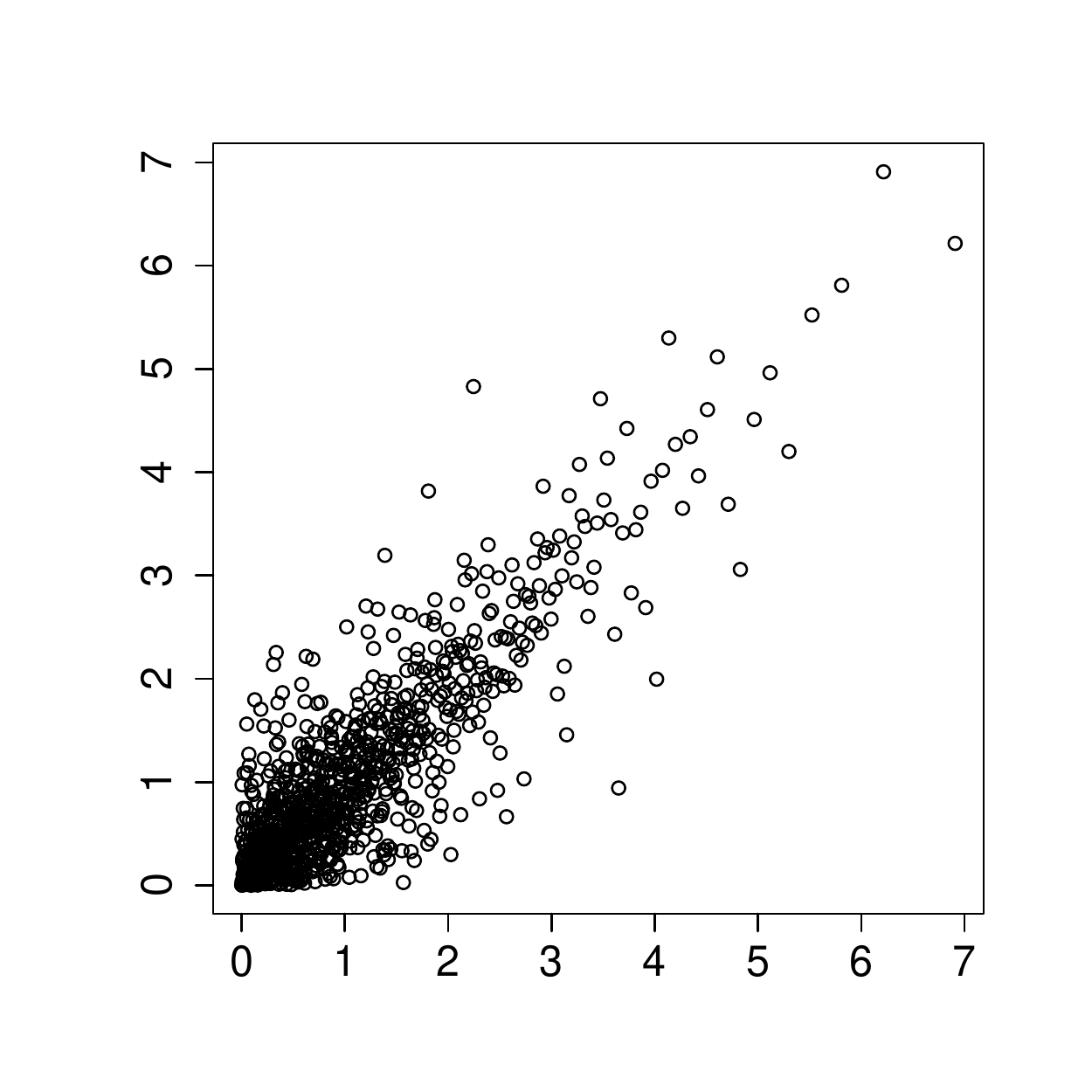}
			\includegraphics[scale = 0.35, trim = 0 50 0 60]{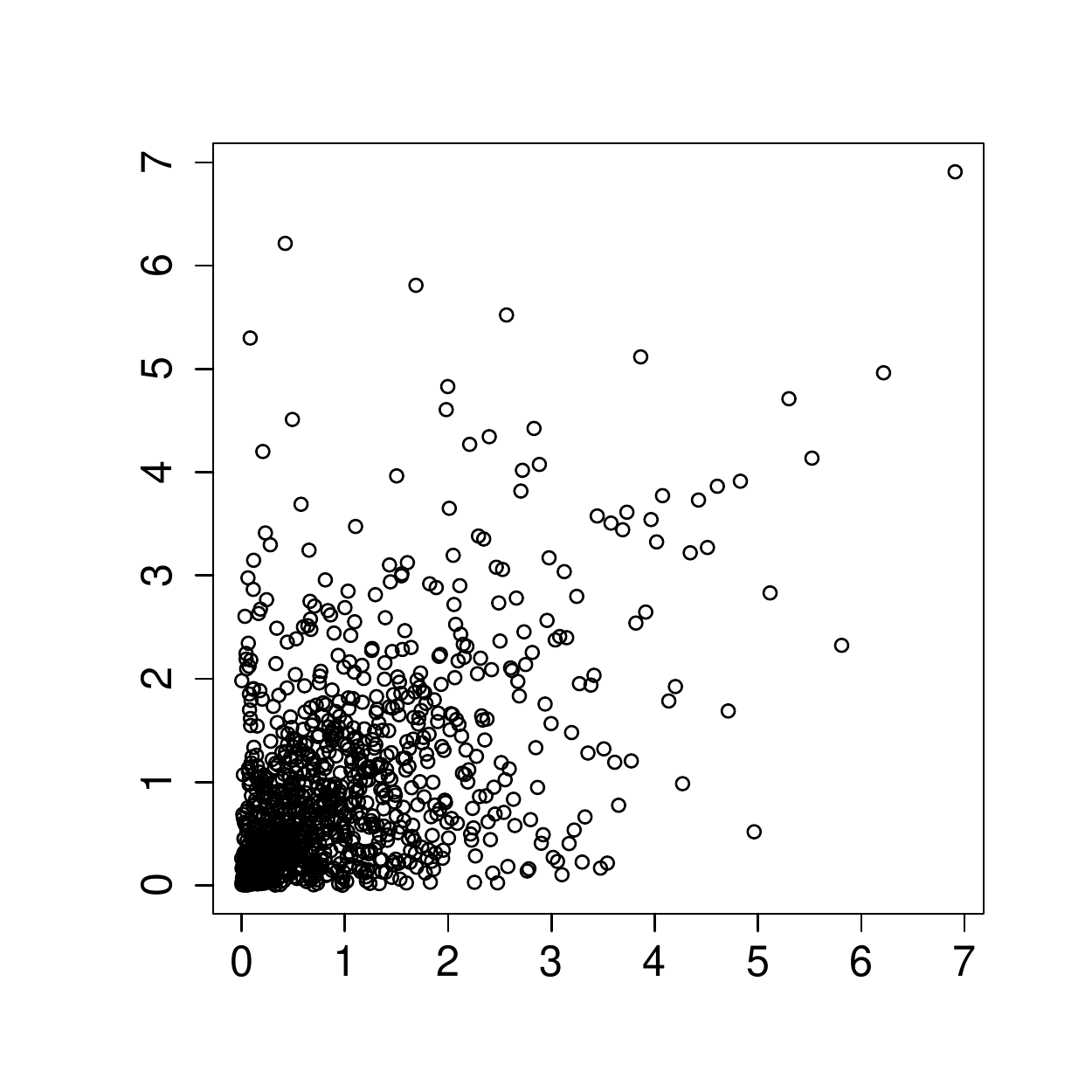}
			\includegraphics[scale = 0.35, trim = 0 50 0 60]{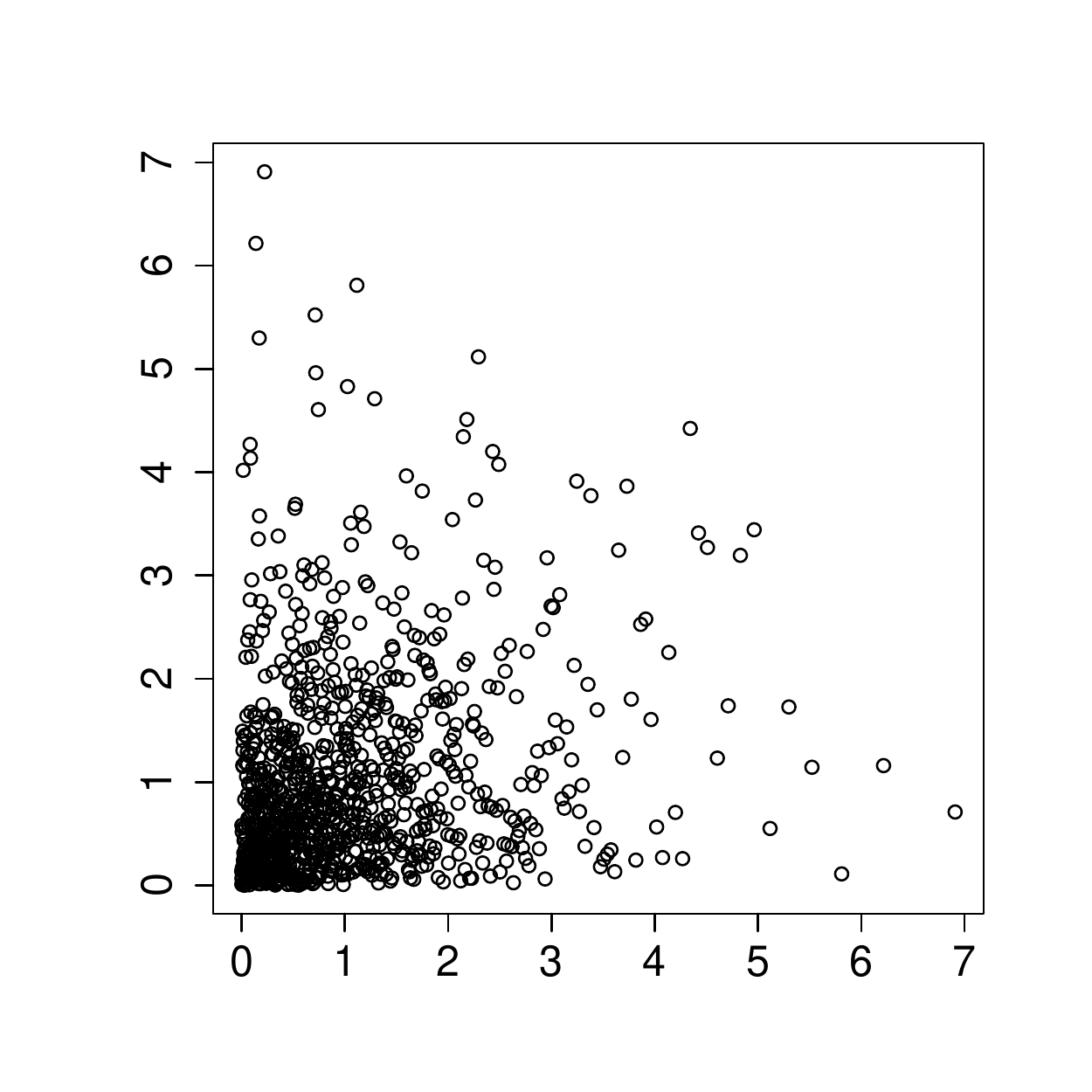}
			\caption{Samples of 1\,000 data points from the Pareto random scale model with parameter $\lambda$ equal to 0.4, 1 and 1.6, from left to right. The marginal distributions are approximately unit exponential.} \label{fig:P-rea}
		\end{figure}

		\subsubsection{Sensitivity with respect to the weight function}

		Recall the weight function in \cref{eq:defg} that is used throughout the paper. It is composed of the weighted indicator functions of the five rectangles $I_1 := [0, 1]^2$, $I_2 := [0, 2]^2$, $I_3 := [1/2, 3/2]^2$, $I_4 := [0, 1] \times [0, 3]$ and $I_5 := [0, 3] \times [0, 1]$. As explained in \cref{sim bivariate}, those rectangles are chosen specifically to ensure identifiability in every model, so that a unique weight function may be used for all simulations.

		We now consider different subsets of the five rectangles above and repeat the simulation study with each of the associated lower dimensional weight functions. Precisely, we define $g^{(1)}$ as the function $g$ in \cref{eq:defg} and by the same principle we construct $g^{(2)}, \dots, g^{(7)}$, using the rectangles in \cref{table:wf}.

		\begin{table}[H]
		\begin{center}
		\begin{tabular}{l|c|c|c|c|c|c|c}
			Weight fct. & $g^{(1)}$ & $g^{(2)}$ & $g^{(3)}$ & $g^{(4)}$ & $g^{(5)}$ & $g^{(6)}$ & $g^{(7)}$ \\
			\hline
			Rectangles & $I_1, I_2, I_3, I_4, I_5$ & $I_1, I_2$ & $I_1, I_3$ & $I_1, I_4, I_5$ & $I_1, I_2, I_3$ & $I_1, I_2, I_4, I_5$ & $I_1, I_3, I_4, I_5$
		\end{tabular}
		\end{center}
		\caption{Rectangles used to construct each weight function.}
		\label{table:wf}
		\end{table}

		We repeat the simulation study from \cref{sim bivariate}; 1\,000 data sets of size $n=5\,000$ are drawn from each of the three models, with the same noise mechanism as before, and from each data set seven estimators are computed based on the seven weight functions. We use the values $k$ that were deemed good previously, that is 800 for the two inverted max-stable models (M1 and M2) and 400 for the Pareto random scale model (M3). For each model and each parameter value, we compare the weight functions based on the estimated RMSE of the M-estimator in \cref{fig:wf}.

		In the inverted H\"usler--Reiss model, the parameter has a one-to-one relation with the coefficient of homogeneity $1/\eta$ of $c$. In order to identify that coefficient, it is sufficient to compare the integral of $c$ over the rectangles $I_1$ and $I_2$. It can moreover be deduced from the developments in \cref{proof IMS} that in this model, the bias arising from the pre-asymptotic approximation of $c$ is largest around the axes. Thus, as can be observed below, adding the non required rectangles $I_4$ and $I_5$, which contain a large portion of the axes, adds bias to the estimator. The best strategy for this model seems to be using $I_1$, $I_2$ and possibly $I_3$.

		In contrast, the parameter in the inverted asymmetric logistic model is not identifiable if the rectangles used are all symmetric, since then $(\theta_1, \theta_2)$ cannot be distinguished from $(\theta_2, \theta_1)$. Therefore the estimator is not uniquely defined when neither $I_4$ nor $I_5$ is used, so the functions $g^{(2)}$, $g^{(3)}$ and $g^{(5)}$ were not included. It is to be noted that $g^{(4)}$ does not include either of $I_2$ and $I_3$, and as such is not able to estimate the homogeneity coefficient $\theta_1 + \theta_2$ well, even if it is able to recover the asymmetry. This explains the monotonic behavior of the error with respect to $\theta_1 + \theta_2$. The other three weight functions perform similarly to each other.

		Finally, in the Pareto random scale model, the weight function $g^{(2)}$ only estimates the homogeneity and as such, it is unable to distinguish the parameters in the range $(0, 1)$, corresponding to asymptotic dependence. It was thus ignored. Among the other functions, the ones that use $I_4$ and $I_5$ ($g^{(1)}$, $g^{(4)}$, $g^{(6)}$, $g^{(7)}$) all have a similar performance whereas the other two ($g^{(3)}$ and $g^{(5)}$) incur a noticeably larger error. It seems that those rectangles help estimating characteristics that are strongly different from the coefficient of homogeneity, which explains why they significantly reduce the RMSE under asymptotic dependence ($\lambda<1$).

		\begin{figure}[H]
			\centering
			\includegraphics[scale = 0.55, trim = 0 35 0 50]{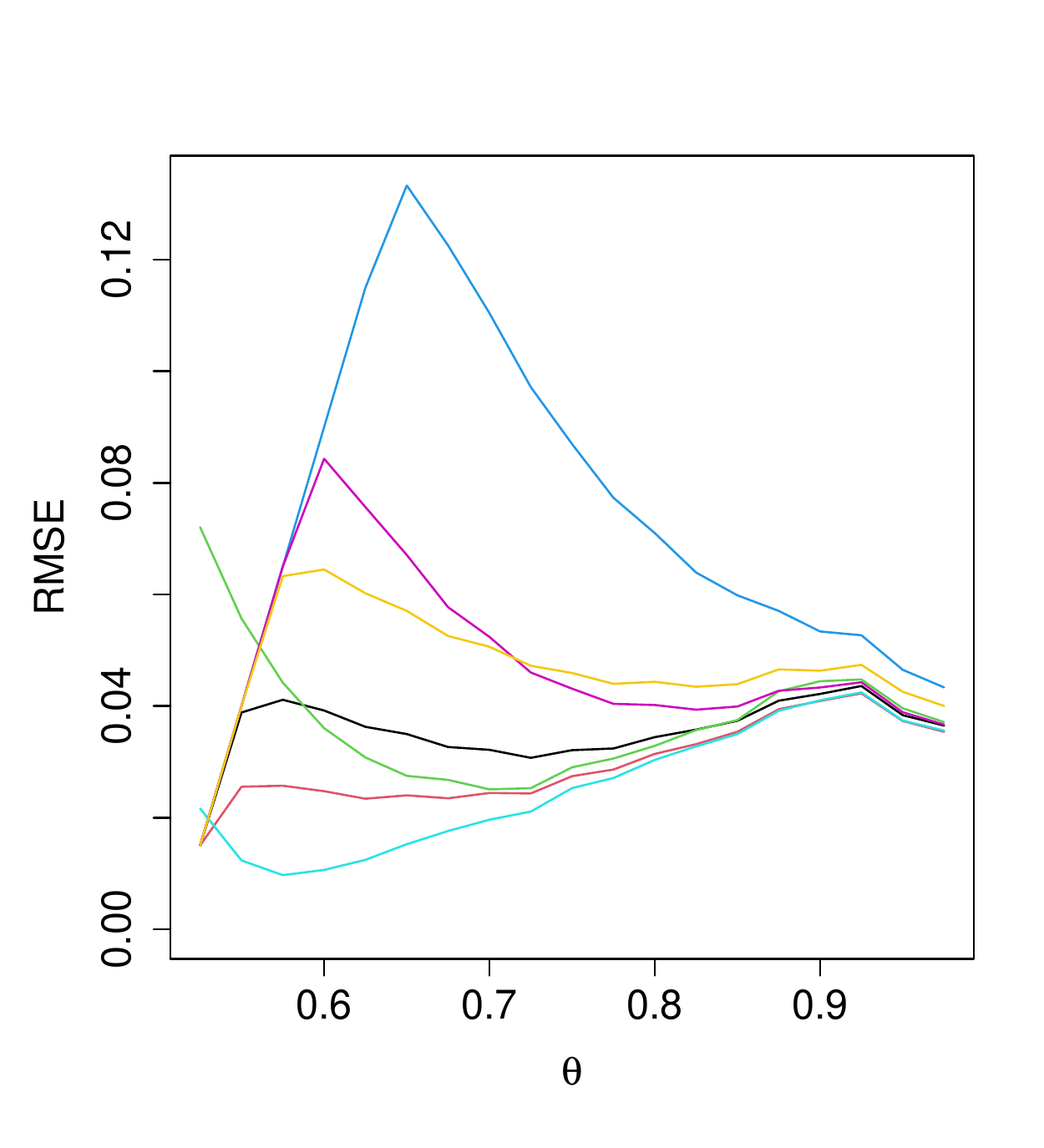}
			\includegraphics[scale = 0.55, trim = 0 35 0 50]{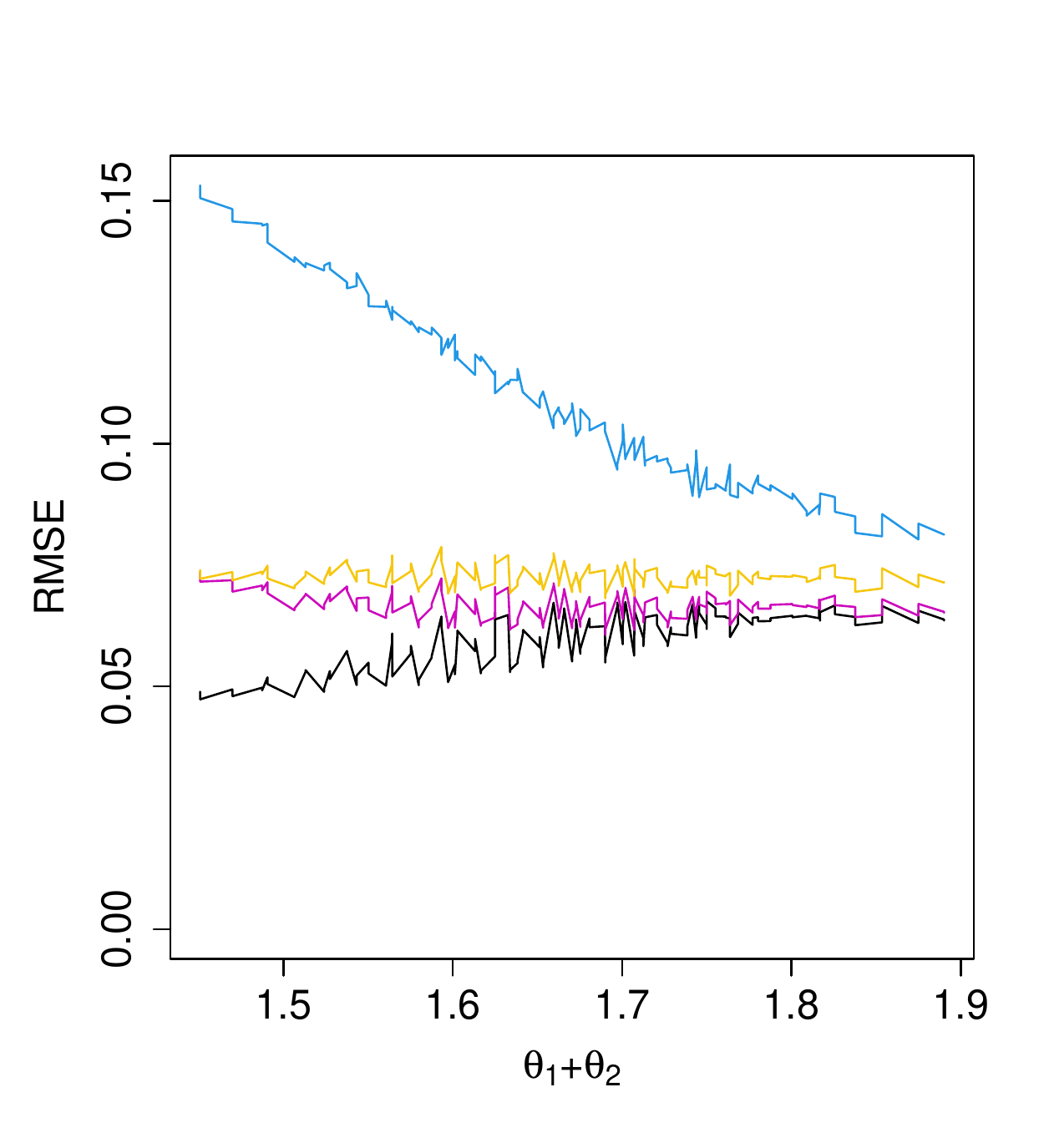}
			\includegraphics[scale = 0.55, trim = 0 35 0 50]{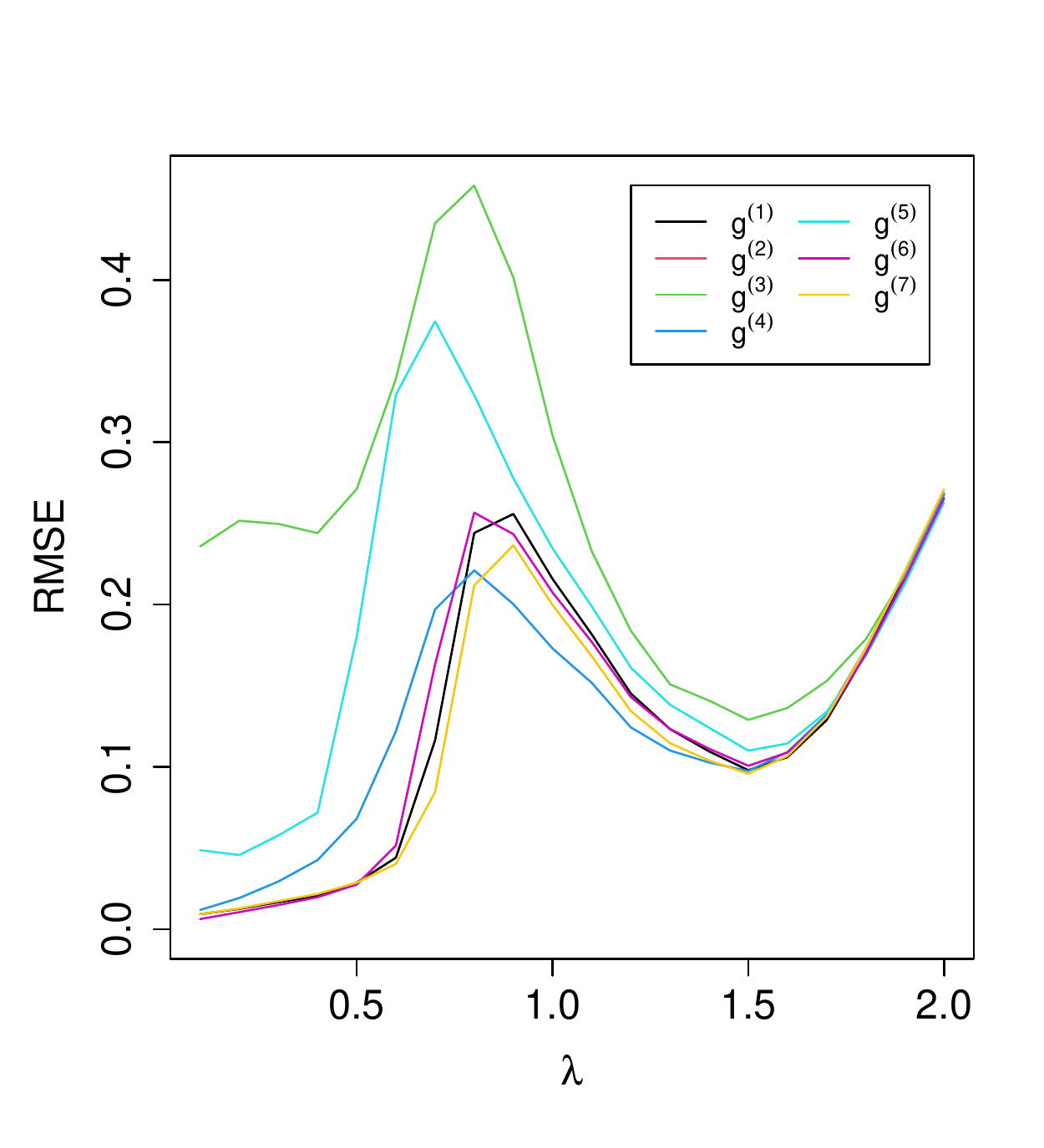}
			\caption{RMSE of the M-estimator in the models M1--M3 as a function of the parameter, based on 1\,000 data sets of size $n=5\,000$, $k=800$ (for M1 and M2) and $k=400$ (for M3). Colors represent the seven weight functions from \cref{table:wf}.}
			\label{fig:wf}
		\end{figure}

		\subsection{Spatial models}

		\Cref{fig:distances} shows the distribution of the distances of all the pairs that are used in the analysis in \cref{sim spatial}. \Cref{Choice of k sp 2,fig:sp-multibox-2} present the same results as in \cref{sim spatial} when the estimator \eqref{eq:defthetahat} is used instead of \eqref{eq:defthetahatls}.

		\begin{figure}[H]
		\centering
		\includegraphics[scale = 0.55, trim = 0 35 0 50]{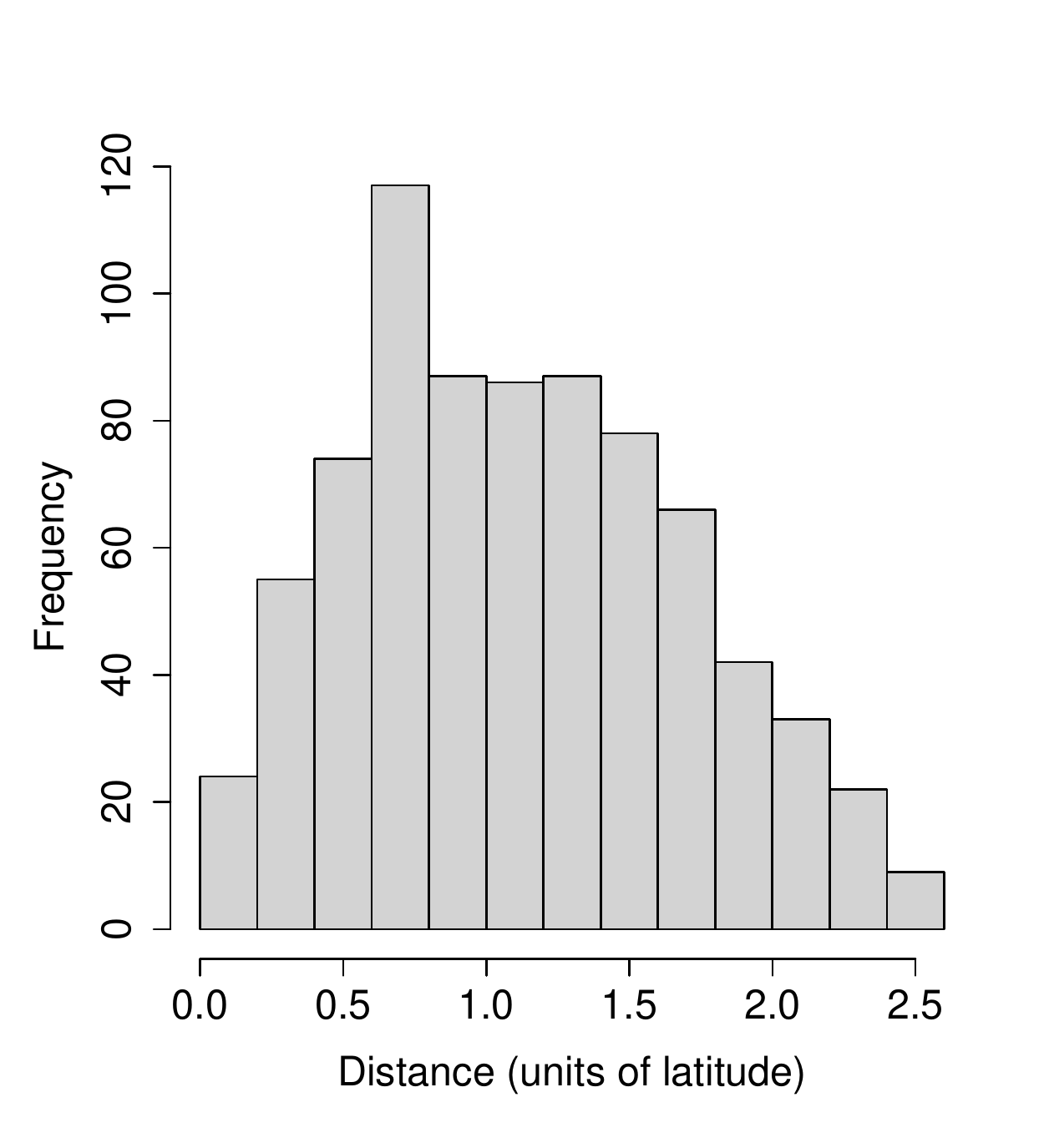}
		\caption{Distribution of the distances $\Delta^{(s)}$ for the 780 pairs used.}
		\label{fig:distances}
		\end{figure}

		\begin{figure}[H]
		\centering
		\includegraphics[scale = 0.35, trim = 0 50 0 60]{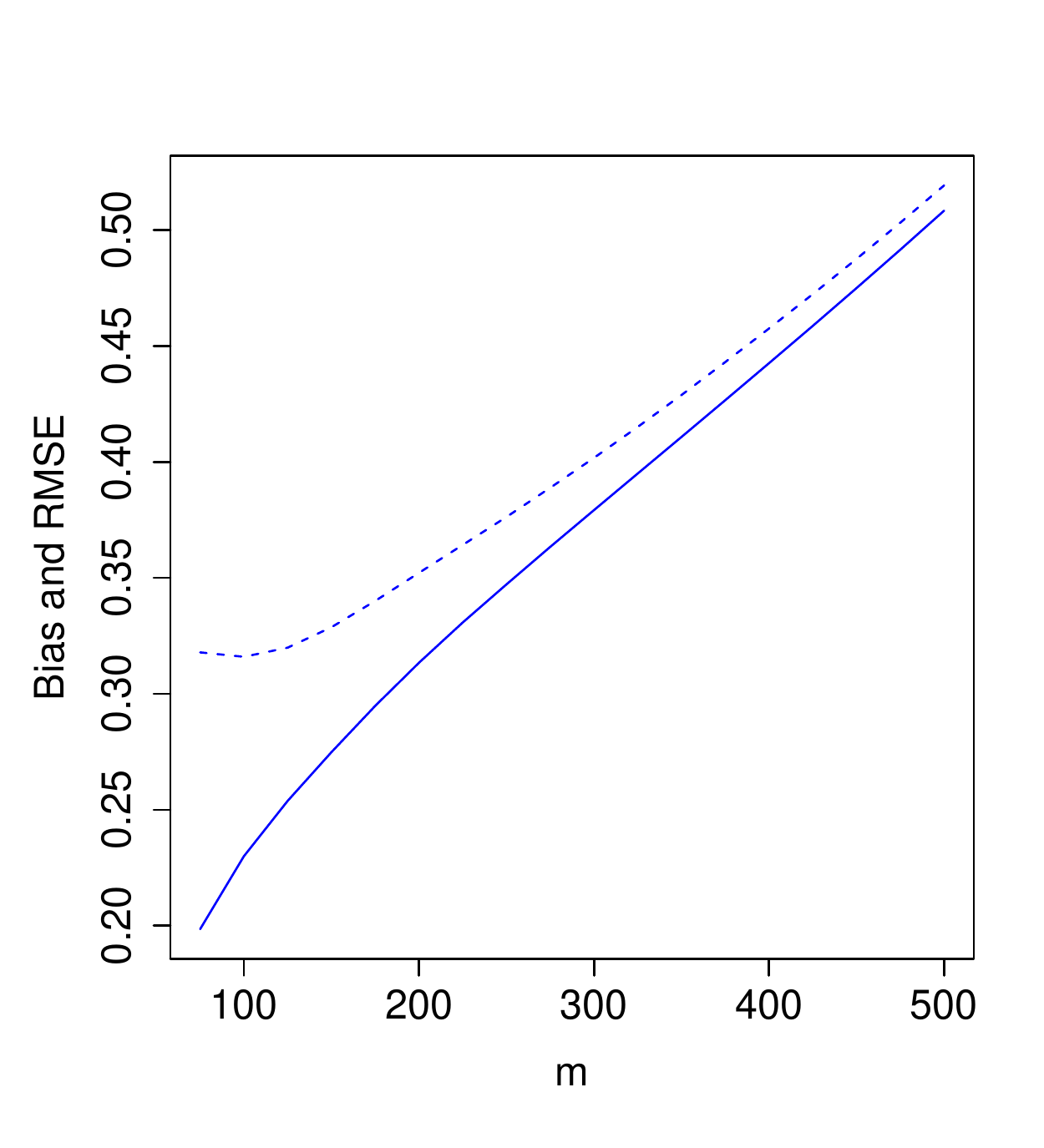}
		\includegraphics[scale = 0.35, trim = 0 50 0 60]{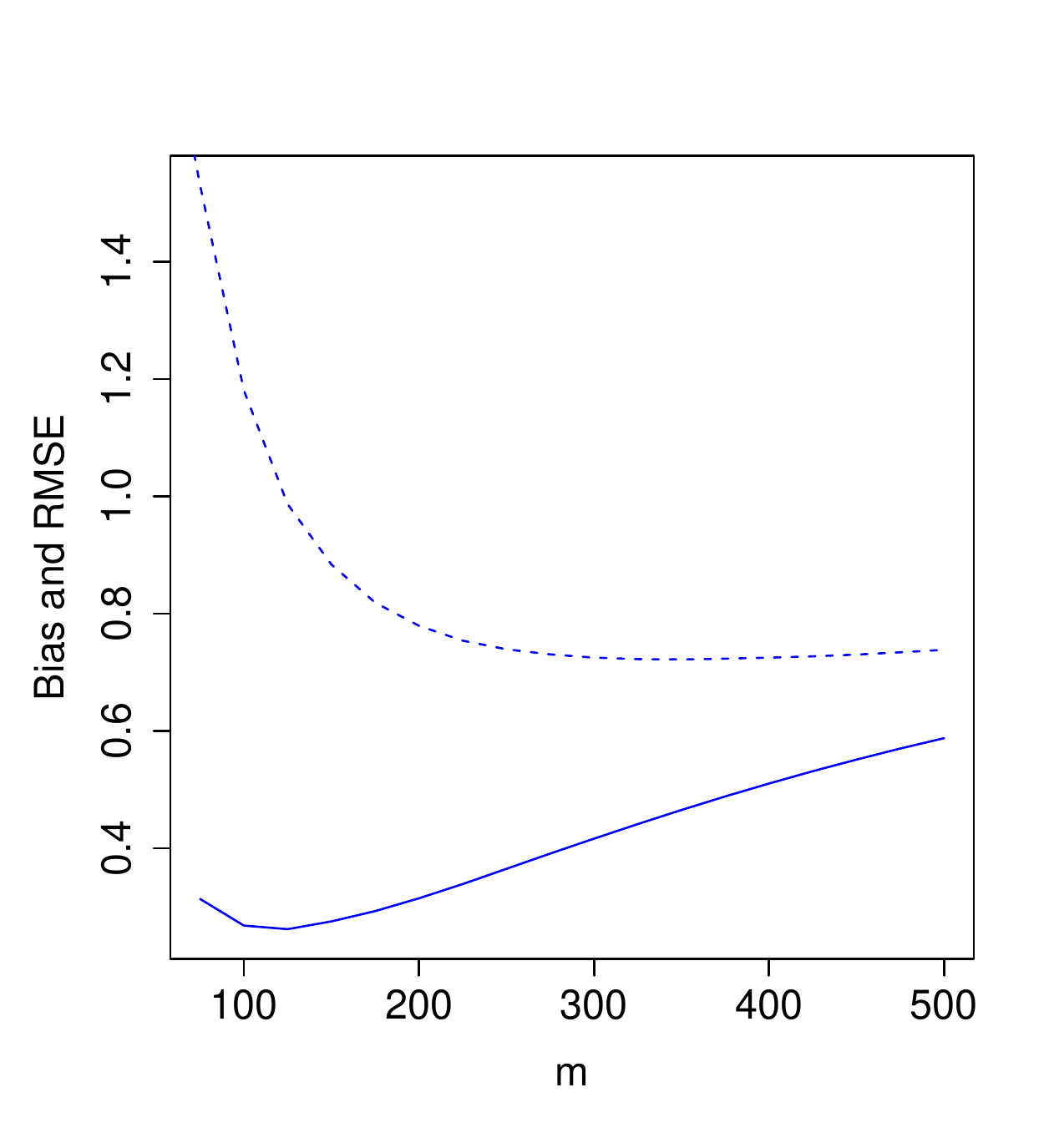}
		\includegraphics[scale = 0.35, trim = 0 50 0 60]{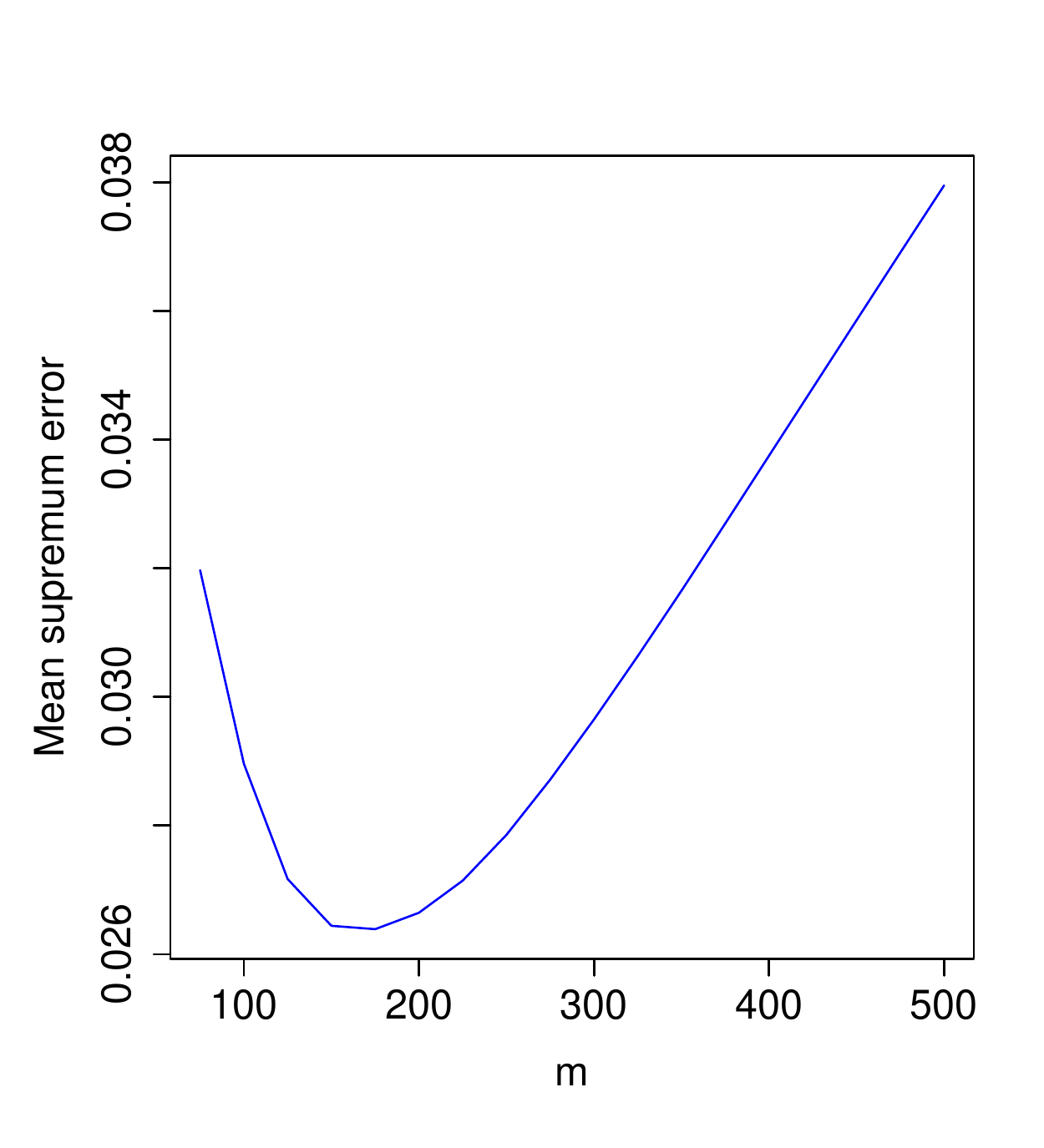}
		\caption{Left and middle columns: Bias (solid line) and RMSE (dotted line) of the estimators of the two spatial parameters $\alpha$ (left) and $\beta$ (middle) as a function of $m$. Right: Mean of the supremum error $\sup_{0 \leq \Delta \leq 3} |\theta(\Delta; \hat\alpha, \hat\beta) - \theta(\Delta; \alpha, \beta)|$ as a function of $m$.}
		\label{Choice of k sp 2}
		\end{figure}

		\begin{figure}[H]
		\centering
		\includegraphics[scale = 0.55, trim = 0 35 0 50]{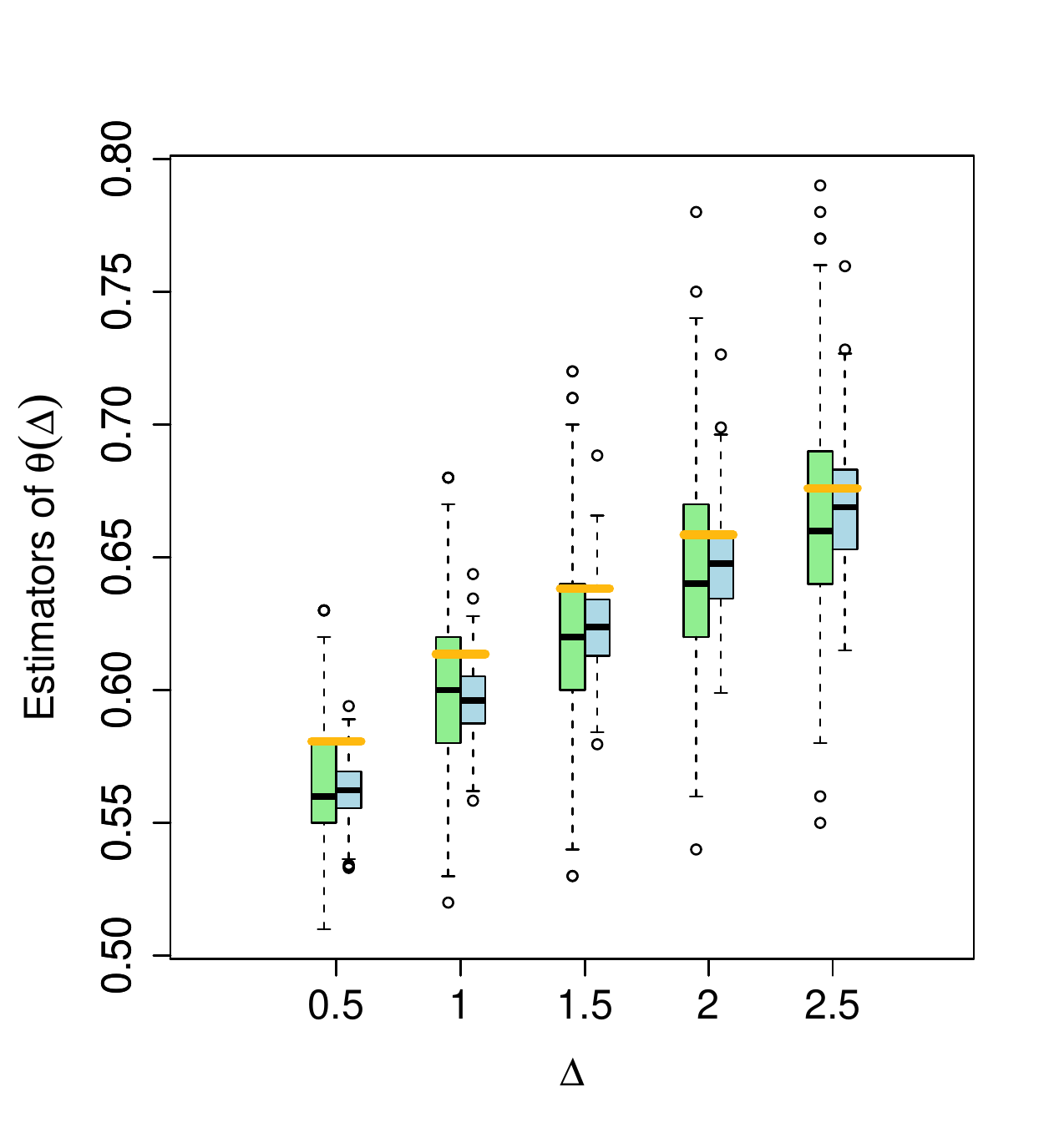}
		\includegraphics[scale = 0.55, trim = 0 35 0 50]{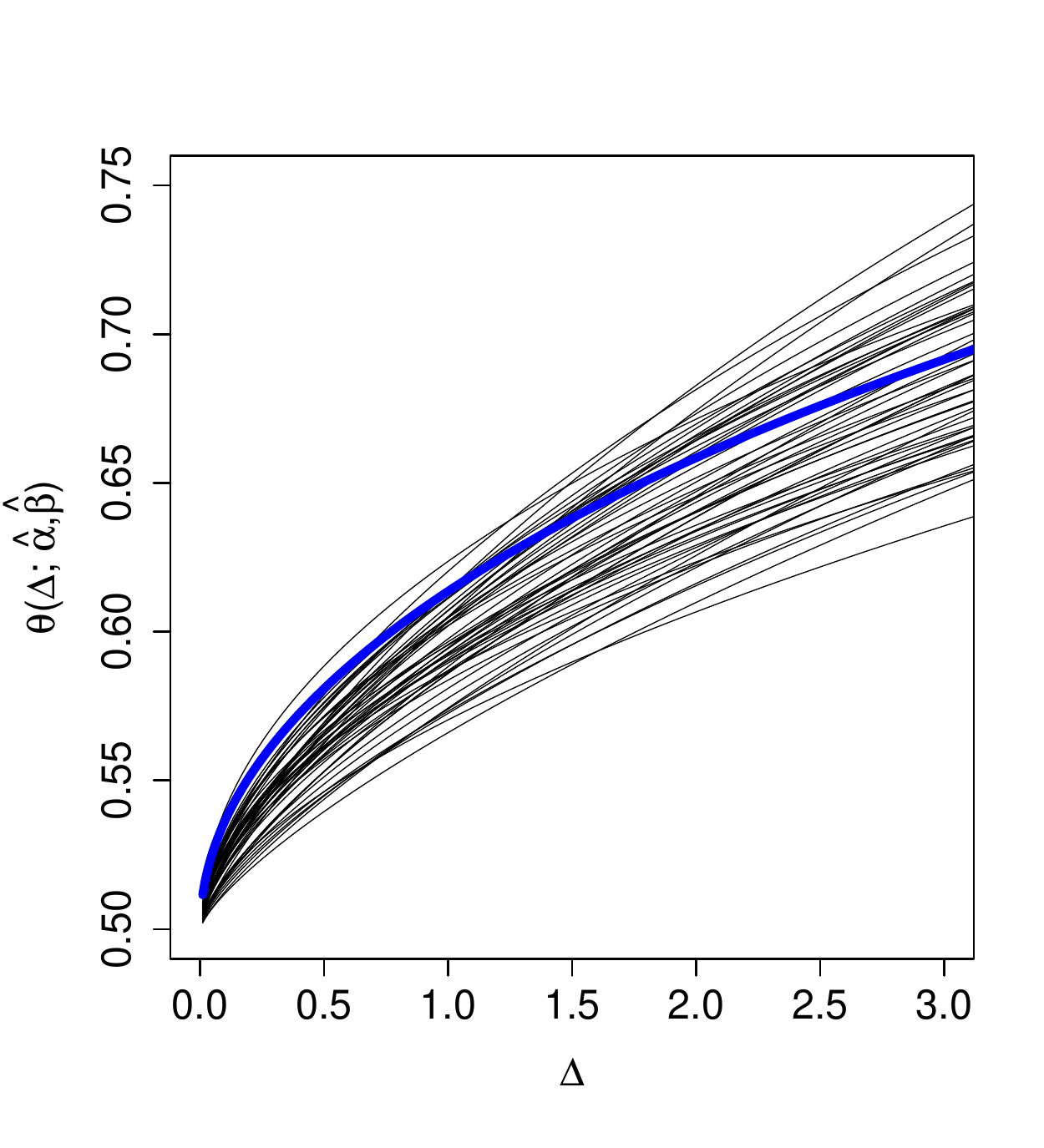}
		\caption{Left panel: Estimators of $\theta(\Delta)$ for 5 different distances. For each distance, bivariate M-estimator $\hth_n^{(s)}$ (green) and spatial estimator $\theta(\Delta^{(s)}; \hat\alpha, \hat\beta)$ (blue) based on the $d=40$ locations. Right panel: 50 sampled curves $\theta(\cdot; \hat\alpha, \hat\beta)$. Blue represents the true curve $\theta(\cdot; \alpha, \beta)$.}
		\label{fig:sp-multibox-2}
		\end{figure}


\end{document}